\documentclass[10pt,letterpaper]{article}

\usepackage[top=2.54cm,left=2.54cm,right=2.54cm,bottom=2.54cm]{geometry}

\usepackage{amsmath,amssymb,amsthm}

\usepackage{setspace}
\usepackage{authblk}

\theoremstyle{plain}
\newtheorem{theorem}{Theorem}[section]

\newtheorem{proposition}[theorem]{Proposition}

\theoremstyle{definition}
\newtheorem{definition}[theorem]{Definition}
\newtheorem{example}[theorem]{Example}
\newtheorem{notation}[theorem]{Notation}

\newtheorem{remark}{Remark}[section]

\usepackage{graphicx}
\usepackage{tikz}

\usepackage[]{hyperref}
\hypersetup{
linkcolor  = blue,
citecolor  = magenta,
urlcolor   = magenta,
colorlinks = true,
}

\usepackage{color}

\usepackage{caption}
\usepackage{subcaption}

\usepackage{tikz}
\usepackage{tkz-graph}
\usetikzlibrary{shapes.geometric}
\usetikzlibrary{shapes.arrows}
\usepackage{array}

\usepackage{comment}
\usepackage{bm}

\usepackage{mathtools}

\usepackage{tikz-cd}

\usepackage{epigraph}

\usepackage[]{algorithm2e}

\usepackage{csquotes}
\usepackage{epigraph}

\usepackage{etoolbox}
\makeatletter
\newlength\epitextskip
\pretocmd{\@epitext}{\em}{}{}
\apptocmd{\@epitext}{\em}{}{}
\patchcmd{\epigraph}{\@epitext{#1}\\}{\@epitext{#1}\\[\epitextskip]}{}{}
\makeatother

\usepackage{float}
\usepackage{ragged2e}

\makeatletter
\def\bbordermatrix#1{\begingroup \m@th
\@tempdima 4.75\p@
\setbox\z@\vbox{%
\def\cr{\crcr\noalign{\kern2\p@\global\let\cr\endline}}%
\ialign{$##$\hfil\kern2\p@\kern\@tempdima&\thinspace\hfil$##$\hfil
&&\quad\hfil$##$\hfil\crcr
\omit\strut\hfil\crcr\noalign{\kern-\baselineskip}%
#1\crcr\omit\strut\cr}}%
\setbox\tw@\vbox{\unvcopy\z@\global\setbox\@ne\lastbox}%
\setbox\tw@\hbox{\unhbox\@ne\unskip\global\setbox\@ne\lastbox}%
\setbox\tw@\hbox{$\kern\wd\@ne\kern-\@tempdima\left[\kern-\wd\@ne
\global\setbox\@ne\vbox{\box\@ne\kern2\p@}%
\vcenter{\kern-\ht\@ne\unvbox\z@\kern-\baselineskip}\,\right]$}%
\null\;\vbox{\kern\ht\@ne\box\tw@}\endgroup}
\makeatother

\usepackage[thinlines]{easytable}
\usepackage{lscape}
\usepackage{booktabs}
\usepackage{xcolor,colortbl}
\definecolor{lavender}{rgb}{0.9, 0.9, 0.98}

\usepackage{comment}
\usepackage{bm}

\NewDocumentCommand{\dgal}{sO{}m}{%
\IfBooleanTF{#1}
{\dgalext{#3}}
{\dgalx[#2]{#3}}%
}

\NewDocumentCommand{\dgalext}{m}{%
\sbox0{%
\mathsurround=0pt % just for safety
$\left\{\vphantom{#1}\right.\kern-\nulldelimiterspace$%
}%
\sbox2{\{}%
\ifdim\ht0=\ht2
\{\kern-.625\wd2 \{#1\}\kern-.625\wd2 \}%
\else
\left\{\kern-.7\wd0\left\{#1\right\}\kern-.7\wd0\right\}%
\fi
}

\NewDocumentCommand{\dgalx}{om}{%
\sbox0{\mathsurround=0pt$#1\{$}%
\sbox2{\{}%
\ifdim\ht0=\ht2
\{\kern-.625\wd2 \{#2\}\kern-.625\wd2 \}%
\else
\mathopen{#1\{\kern-.7\wd0 #1\{}
#2
\mathclose{#1\}\kern-.7\wd0 #1\}}
\fi
}

\newcommand\phantomarrow[2]{%
\setbox0=\hbox{$\displaystyle #1\to$}%
\hbox to \wd0{%
$#2\mapstochar
\cleaders\hbox{$\mkern-1mu\relbar\mkern-3mu$}\hfill
\mkern-7mu\rightarrow$}%
\,}

\usepackage{pgfgantt} %para gráfico de gantt usado no cronograma
\definecolor{tartorange}{rgb}{0.96, 0.28, 0.25}
\definecolor{lightkaki}{rgb}{0.96,0.94,0.53}

\usepackage{caption}
\usepackage{tabularx}% <-- added
\usepackage[export]{adjustbox}% <-- added

\usepackage{geometry}
\usepackage{graphicx}

\usepackage[matrix, arrow]{xy}
\usepackage{array}
\newcolumntype{M}[1]{>{\centering\arraybackslash}m{#1}}
\newcolumntype{N}{@{}m{0pt}@{}}

\usepackage[thinlines]{easytable}
\usepackage{lscape}
\usepackage{booktabs}
\usepackage{xcolor,colortbl}
\definecolor{lavender}{rgb}{0.9, 0.9, 0.98}
\usepackage{booktabs}
\usepackage{multirow}
\usepackage{threeparttable}
\usepackage{longtable}
\usepackage{rotating}
\usepackage{tabularx}
\usepackage{enumitem}
\usepackage{titlesec}

\newcolumntype{P}[1]{>{\centering\arraybackslash}p{#1}}

\usepackage{xcolor}
\usepackage{caption}
\usepackage{subcaption}
\usepackage{bm}

\usepackage{tikz}
\usepackage{tkz-graph}
\usetikzlibrary{shapes.geometric}
\usetikzlibrary{shapes.arrows}
\usepackage{array}

\makeatletter
\def\bbordermatrix#1{\begingroup \m@th
\@tempdima 4.75\p@
\setbox\z@\vbox{%
\def\cr{\crcr\noalign{\kern2\p@\global\let\cr\endline}}%
\ialign{$##$\hfil\kern2\p@\kern\@tempdima&\thinspace\hfil$##$\hfil
	&&\quad\hfil$##$\hfil\crcr
	\omit\strut\hfil\crcr\noalign{\kern-\baselineskip}%
	#1\crcr\omit\strut\cr}}%
\setbox\tw@\vbox{\unvcopy\z@\global\setbox\@ne\lastbox}%
\setbox\tw@\hbox{\unhbox\@ne\unskip\global\setbox\@ne\lastbox}%
\setbox\tw@\hbox{$\kern\wd\@ne\kern-\@tempdima\left[\kern-\wd\@ne
\global\setbox\@ne\vbox{\box\@ne\kern2\p@}%
\vcenter{\kern-\ht\@ne\unvbox\z@\kern-\baselineskip}\,\right]$}%
\null\;\vbox{\kern\ht\@ne\box\tw@}\endgroup}
\makeatother
\usepackage{enumitem}% http://ctan.org/pkg/enumitem

\usepackage{bm}

\definecolor{lightgray}{gray}{0.95}

\newcommand\musd[2]{\begin{tabular}[c]{@{}c@{}}\small$#1$\\\scriptsize$(#2)$\end{tabular}}

\usepackage{titlesec}
\titleformat{\paragraph}
{\normalfont\normalsize\bfseries}{\theparagraph}{1em}{}
\titlespacing*{\paragraph}
{0pt}{3.25ex plus 1ex minus .2ex}{1.5ex plus .2ex}

\title{Directed Q-Analysis and Directed Higher-Order Connectivity on Digraphs: A Quantitative Approach}

\author[1, 2]{Heitor Baldo\footnote{Email address: \texttt{heitor.baldo@igdore.org}}}
\author[3]{Luiz A. Baccalá}
\author[1,4]{André Fujita}
\author[5]{Koichi Sameshima}
\affil[1]{Department of Computer Science, Institute of Mathematics, Statistics, and Computer Science, University of São Paulo, \emph{Brazil}}
\affil[2]{Institute for Globally Distributed Open Research and Education (IGDORE), \emph{Sweden}}
\affil[3]{Department of Telecommunications and Control Engineering, Escola Politécnica, University of São Paulo, \emph{Brazil}}
\affil[4]{Division of Network AI Statistics, Medical Institute of Bioregulation, Kyushu University, \emph{Japan}}
\affil[5]{Department of Radiology and Oncology, School of Medicine, University of São Paulo, \emph{Brazil}}

\date{}

\begin{document}

\maketitle

\begin{abstract}
Traditional graph analysis focuses on nodes and edges, that is, pairwise relationships. Yet many real-world networks, including biological, social, and communication networks, involve higher-order relationships in which multiple nodes interact simultaneously. This has led many to develop network topology analysis methods based on higher-order structures and higher-order connectivity, seeking to reveal complex interactions beyond node pairs. Many of the latter address only undirected networks. To overcome this, we lay out a mathematical formalism resting on directed clique complexes constructed from directed graphs (their ``higher-order structures'' or ``simplicial structures''), stressing the interrelations between directed cliques (their ``directed higher-order connectivities''), leading towards a more complete directed Q-analysis that allows quantifying, characterizing, and comparing similarities involving simplicial structures.
\end{abstract}

%{
%  \hypersetup{linkcolor=black}
%  \tableofcontents
%}

%%%%%%%%%%%%%%%%%%%%%%%%%%%%%%%%%%%%%%%%%%%%%%%%%%%%%%%%
\section{Introduction}

Graphs, composed of nodes and edges connecting them, are often used to represent real-world systems, where nodes represent system entities and edges represent interactions. They have been applied to fields as diverse as social science, biology, computer science, linguistics, and transportation planning \cite{Chung2006, Newman}. Graph structure is typically studied through its subgraphs, i.e., smaller graphs contained within the larger graph. Complete subgraphs (\textit{cliques}), i.e., subgraphs where each node connects to all other nodes in the subgraph, are of special interest so that graphs can be seen as hierarchically structured: ``higher-level'' (larger) subgraphs contain other ``lower-level'' (smaller) subgraphs. This is the basis of the \textit{clique organization} or \textit{clique topology} of a graph, since larger cliques contain smaller cliques. This hierarchy enables the construction of simplicial complexes associated with such cliques, known as \textit{clique complexes} \cite{Aharoni}. When dealing with \textit{directed graphs}  (\textit{digraphs}), whose edges have orientation, this same type of organization is present. We may profit from taking edge directionality into account in the form of \textit{directed clique complexes} (\textit{directed flag complexes}) \cite{Masulli, Reimann}, representing their ``higher-order structures'' or ``simplicial structures.''

Traditionally, graph analysis (Section \ref{sec:Graph-Theory}) focuses on pairwise relationships, yet higher-order graph relationships, i.e., those involving more than just node pairs, reveal novel and more complex interactions. Consequently, many have proposed methods that incorporate graph/digraph topology by looking at their higher-order relationships/connectivity (social science \cite{Atkin1977, Beaumont}, biology \cite{Rabadan}, network neuroscience \cite{Giusti2016, Giusti2015, Reimann, Singh, Sizemore}, and transportation networks \cite{Johnson}) comprising computational (algebraic) topology and computational geometry concepts such as simplicial complexes (Section \ref{sec:ADSC}), homology, homotopy, Betti numbers, and persistent homology (PH) into what is called \textit{topological data analysis} (TDA) \cite{Chazal}. For example, some applications in network neuroscience have included assessing neural functions and structures through clique complexes built from functional or structural brain networks \cite{Giusti2016, Giusti2015, Petri, Reimann, Sizemore} in the context of characterizing brain networks in patients with attention deficit hyperactivity disorder (ADHD) and autism spectrum disorders (ASD) through PH \cite{Lee1-AD, Lee2-AD}. In addition, characterization \cite{Merelli} and detection \cite{Fernandez, Piangerelli, Sun2023, Wang} of epileptic seizures through PH, and analysis of functional brain connectivity dynamics through persistence vineyards \cite{Yoo}, were studied alongside the examination of functional connectivity networks \cite{Tadic} and human connectomes \cite{Andjelkovic2020}.

Here, we lay out a mathematical formalism to study digraph topology based on their \textit{directed clique complexes} and their \textit{directed higher-order connectivity} by focusing on the theoretical development of directed flag complexes towards a more complete directed Q-analysis in analogy with classical Q-analysis (Section \ref{sec:directed-q-analysis}) whereby higher-order connectivity between directed simplices is studied via the concept of \textit{directed higher-order adjacencies} (lower and upper directed adjacencies) (Section \ref{sec:dir-q-analysis}). Furthermore, we show all these constructs can be extended to the weighted graph case, i.e. when these complexes are obtained from weighted digraphs (Section \ref{sec:weighted-dfc}), allowing us to expand several standard quantifiers and similarity comparison methods defined primarily for (weighted and unweighted) graphs, such as distance-related measures, measures of centrality, segregation, spectrum-related measures, graph kernels, etc., which take into account their abstract, algebraic, and topological properties, to (weighted and unweighted) directed flag complexes, by explicitly taking into account their directed higher-order connectivity (Section \ref{sec:quantitative}). We end by presenting a simple application example to \textit{C. elegans} frontal neuronal network, and show that it reveals properties beyond those deduced via the usual graph analysis (Section \ref{sec:bio_appli}), topped by conclusions in Section \ref{sec:conclu}.

%%%%%%%%%%%%%%%%%%%%%%%%%%%%%%%%%%%%%%%%%%%%%%%%%%%%%%%%
\section{Elements of Graph Theory}
\label{sec:Graph-Theory}

We begin by briefly introducing the graph-theoretic concepts and notation used throughout this text, following \cite{Bang, Diestel, Harary}.

A \textit{graph} is a pair $G = (V, E)$, where $V$ is the set of \textit{vertices} (or \textit{nodes}) and $E \subseteq V \times V$ is the set of edges. If $E$ consists of unordered pairs, $G$ is \textit{undirected}; if ordered, $G$ is a \textit{directed graph} (or \textit{digraph}). For an arc $(v,u)$ on a digraph, the direction is from $v$ to $u$. The \textit{order} (or \textit{size}) of $G$ is $|V|$, and if $|V|, |E| < \infty$, $G$ is \textit{finite}; if $|V| = |E| = 0$, $G$ is \textit{null}. We consider only finite, non-null (di)graphs without loops or multiple edges/arcs, though a digraph may have a \textit{double edge} ($(u,v), (v,u) \in E$). Two vertices $v, u \in V$ are \textit{adjacent to} if $(v,u) \in E$. The \textit{neighborhood} of $v$ is $\mathcal{N}(v) = \{u \in V : (v, u) \in E\}$. The \textit{degree} of $v$ is $\deg(v) = |\mathcal{N}(v)|$; if $\deg(v) = 0$, $v$ is \textit{isolated}. In digraphs, \textit{in-neighborhood} is $\mathcal{N}^{-}(v) = \{u \in V : (u, v) \in E\}$ and \textit{out-neighborhood} is $\mathcal{N}^{+}(v) = \{u \in V : (v, u) \in E\}$, with corresponding degrees $\deg^{-}(v)$ and $\deg^{+}(v)$. A vertex $v$ is a \textit{source} if $\deg^{-}(v)=0$, $\deg^{+}(v)\neq0$, and a \textit{sink} if $\deg^{-}(v)\neq0$, $\deg^{+}(v)=0$. The \textit{underlying (undirected) graph} of a digraph is the undirected graph obtained by replacing each arc with an edge. A (di)graph $G' = (V', E')$ is a \textit{sub(di)graph} of $G = (V, E)$ if $V' \subseteq V$ and $E' \subseteq E$; it is \textit{induced} if all edges/arcs between vertices in $V'$ are included. A graph is \textit{complete} if every pair of vertices is adjacent. A \textit{$(k+1)$-clique} is a complete induced subgraph of size $k+1$, and it is \textit{maximal} if not contained in a larger clique. A \textit{weighted (di)graph} is a triplet $G^{\omega} = (V, E, \omega)$ where $\omega: V \times V \rightarrow \mathbb{R}_{\ge 0}$ assign weights to edges/arcs. The \textit{weighted degree} of $v$ on a weighted graph is $\deg_{\omega}(v) = \sum_{u \in \mathcal{N}(v)} \omega(v, u)$. In a weighted digraph, the \textit{in-degree} is $\deg_{\omega}^{-}(v) = \sum_{u \in \mathcal{N}^{-}(v)} \omega(u, v)$, and the \textit{out-degree} is $\deg_{\omega}^{+}(v) = \sum_{u \in \mathcal{N}^{+}(v)} \omega(v, u)$.

%---------------

%%%%%%%%%%%%%%%%%%%%%%%%%%%%%%%%%%%%%%%%%%
\section{Abstract (Directed) Simplicial Complexes}
\label{sec:ASC}

The notions of abstract simplicial complexes and abstract directed simplicial complexes are formally introduced. The primary bibliography of this section is \cite{Hatcher, Munkres, Reimann}.

%------------
\subsection{Abstract Simplicial Complexes}

\begin{definition}\label{def:ASC}
An \textit{abstract simplicial complex (ASC)} is a finite collection $\mathcal{X}$ of finite sets such that if $\sigma \in \mathcal{X}$, then for all $\tau \subseteq \sigma$ we have $\tau \in \mathcal{X}$ (closed under inclusion of subsets).
\end{definition}

Let $\mathcal{X}$ be an ASC. Each element $\sigma \in \mathcal{X}$ is an \textit{abstract simplex}, or $n$-\textit{simplex}, if $|\sigma| = n+1$. The \textit{dimension} of $\sigma$ is defined as $\dim \sigma = n$, and the dimension of $\mathcal{X}$ is $\dim \mathcal{X} = \max_{\sigma \in \mathcal{X}} (\dim \sigma)$. A vertex of $\sigma$ is any element $v_i \in \sigma$, and $\sigma$ is determined by its vertices. We denote $n$-simplices as $\sigma^{(n)}$, and by convention $\emptyset \in \mathcal{X}$. The set of $k$-simplices in $\mathcal{X}$ is $X_k$, and the \textit{vertex set} of $\mathcal{X}$ is $V_{\mathcal{X}}$, satisfying $\mathcal{X} \subseteq \mathcal{P}(V_{\mathcal{X}})$. For $\sigma, \sigma' \in \mathcal{X}$, either $\sigma \cap \sigma' \in \sigma$ and $\in \sigma'$, or $\sigma \cap \sigma' = \emptyset$. The term \textit{simplicial complex} is often used interchangeably with ASC. A \textit{$k$-face} of an $n$-simplex $\sigma$ is a $k$-simplex $\tau \subseteq \sigma$. If $\tau \subset \sigma$, then $\tau$ is a \textit{proper face}. The $(n-1)$-faces of $\sigma$ form its \textit{boundary}. A simplex is \textit{maximal} if it is not a face of any other simplex in $\mathcal{X}$. A \textit{simplicial family} \cite{Johnson} is a collection of simplices not necessarily closed under inclusion. Although more general than an ASC, it determines one.

\begin{example}
A graph $G$ with the vertex set $V$ is a 1-dimensional ASC on $V$, where 0-simplices are vertices and 1-simplices are edges.
\end{example}

%-----------
\subsection{Abstract Directed Simplicial Complexes}
\label{sec:ADSC}

The assignment of a total ordering to the vertices of an $n$-simplex allows us to define an \textit{ordered simplex}. For instance, given $\{v_0, \ldots, v_n\}$, we impose $v_i < v_j$ whenever $i < j$, giving the ordered simplex $[v_0, \ldots, v_n]$. We denote an ordered $n$-simplex by $\sigma^{(n)} = [v_0, \ldots, v_n]$. The edges inherit the ordering, so $[v_i, v_j]$ satisfies $v_i < v_j$ if $i < j$. To simplify notation, we may write $v_i = i$ and represent a simplex as $[0 \ldots n] = [0, \ldots, n] = [v_0, \ldots, v_n]$, while remembering that labels do not necessarily match positional indices.

We now define an analog of ASC (Definition \ref{def:ASC}) for ordered simplices (from now on called \textit{directed simplices} \cite{Reimann}).

\begin{definition}\label{def:ADSC}
An \textit{abstract directed simplicial complex (ADSC)} is a finite collection $\mathcal{X}$ of totally ordered finite sets such that if $\sigma \in \mathcal{X}$, then every totally ordered subset $\tau \subseteq \sigma$ (with inherited ordering) also belongs to $\mathcal{X}$.
\end{definition}

All definitions and terminology established for ASCs also apply to ADSCs.

\begin{remark}
Two directed simplices may share the same vertex set but differ in ordering. Therefore, ADSCs are not ASCs in general, except in cases where the ordering uniquely identifies each simplex.
\end{remark}

Strictly speaking, an ADSC is an example of \textit{semi-simplicial set} \cite{Eilenberg, Friedman}, which generalizes both ASCs and ADSCs. 

A practical way to deal with the faces of a (directed) simplex is through \textit{face maps}, which are defined as follows.

%To handle the faces of (directed) simplices, we introduce the following:

\begin{definition}\label{facemap}
Let $\mathcal{X} = \bigcup_{k=0}^{\dim \mathcal{X}} X_k$ be an ASC or ADSC, where $X_k$ is the set of all (directed) $k$-simplices. For $0 \leq k \leq \dim \mathcal{X} - 1$ and $0 \leq i \leq k+1$, the \textit{$i$-th face map} $\hat{d}_i : X_{k+1} \to X_k$ is defined by $\hat{d}_i([v_0, \ldots, v_{k+1}]) = [v_0, \ldots, \hat{v_i}, \ldots, v_{k+1}]$, where $\hat{v_i}$ denotes the deletion of the $i$-th vertex from the simplex.
\end{definition}

%%%%%%%%%%%%%%%%%%%%%%%%%%%%%%%%%%%%%%%%%%
\section{Flag Complexes and Directed Flag Complexes}
\label{sec:DFC}

As mentioned in the previous section, ASCs generalize graphs and can be constructed in various ways from a given graph $G = (V, E)$ \cite{Jonsson, Maletic}. A standard and intuitive method is through the \textit{clique complex} (or \textit{flag complex}), formed by treating each $(n+1)$-clique of $G$ as an $n$-simplex~\cite{Aharoni}.

\begin{definition}
Given a graph $G = (V, E)$, its \textit{flag complex} (or \textit{clique complex}), denoted by $\mathrm{Fl}(G)$, is the ASC on $V$ whose $k$-simplices are the subsets of $V$ that span $(k+1)$-cliques in $G$.
\end{definition}

For directed graphs, the concept is extended via \textit{directed flag complexes} \cite{Masulli, Reimann}, which rely on the notion of \textit{directed cliques}.

\begin{definition}
A \textit{directed $(k+1)$-clique} is a digraph $G = (V, E)$ with $V = \{ v_0, \ldots, v_k \}$ such that its underlying undirected graph is a $(k+1)$-clique, and there exists a directed edge from $v_i$ to $v_j$ for all $0 \le i < j \le k$, i.e., $(v_i, v_j) \in E$.
\end{definition}

By definition, every directed $(k+1)$-clique has a source ($v_0$) and a sink ($v_k$). Examples for $k=0,1,2,3,4$ are illustrated in Figure~\ref{dir-clique}.
\begin{figure}[h!]
\centering
\includegraphics[scale=1.15]{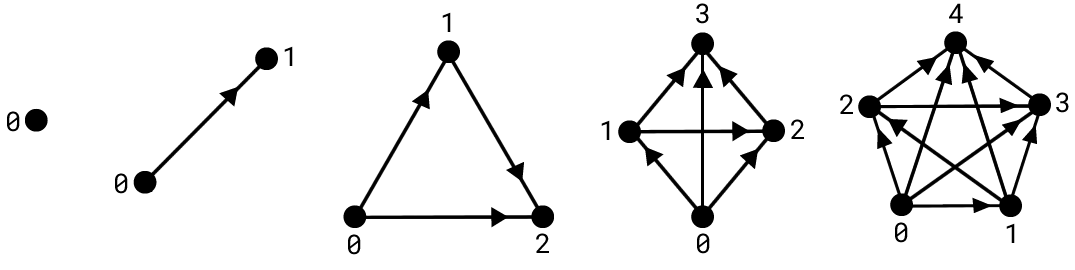}
\caption{Examples of directed $(k+1)$-cliques, for $k=0,1,2,3,4$.}
\label{dir-clique}
\end{figure}

\begin{definition}\label{def:DFC}
Given a digraph $G = (V, E)$, its \textit{directed flag complex} (DFC), denoted by $\mathrm{dFl}(G)$, is the ADSC whose directed $k$-simplices span directed $(k+1)$-cliques of $G$, i.e. for every $[v_{0}, \ldots,v_{k}] \in \mathrm{dFl}(G)$, we have $v_{i} \in V$, $\forall i$, and $(v_{i}, v_{j}) \in E$,  $\forall i < j$.
\end{definition}

\begin{definition}
The flag complex associated with the underlying undirected graph of a digraph is its \textit{underlying flag complex}.
\end{definition}

\begin{example}\label{ex:dir-flag-complex}
Figure~\ref{fig:flag-complex-2} represents a digraph $G$, its underlying flag complex $\mathrm{Fl(G)}$, and its DFC $\mathrm{dFl(G)}$.
\begin{figure}[h!]
\centering
\begin{subfigure}{.23\textwidth}
	\centering
	\includegraphics[scale=0.7]{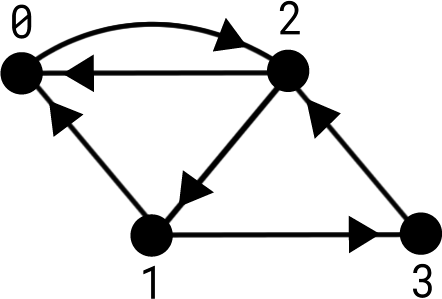}
	\caption{Digraph $G$.}
	\label{fig:dig}
\end{subfigure}%
\begin{subfigure}{.23\textwidth}
	\centering
	\includegraphics[scale=0.7]{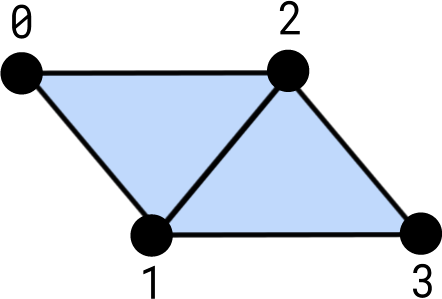}
	\caption{$\mathrm{Fl}(G)$.}
	\label{fig:flag}
\end{subfigure}
\begin{subfigure}{.23\textwidth}
	\centering
	\includegraphics[scale=0.7]{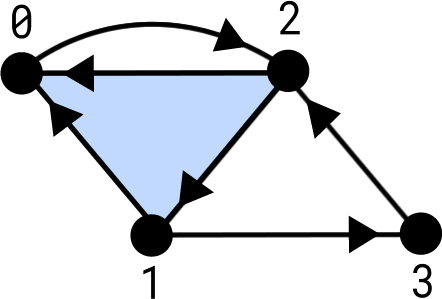}
	\caption{$\mathrm{dFl}(G)$.}
	\label{fig:dflag}
\end{subfigure}
\caption{A digraph $G$ together with its underlying flag complex $\mathrm{Fl(G)}$ and its DFC $\mathrm{dFl(G)}$.}
\label{fig:flag-complex-2}
\end{figure}
\end{example}

Although the flag complex associated with a graph is an example of ASC, a DFC associated with a digraph is an example of ADSC, since the directed simplices of a DFC may not be uniquely determined by their sets of vertices, as they may differ due to the ordering of their vertices (e.g., the set of vertices associated with a double edge spans two different directed cliques). On the other hand, the corresponding DFCs of digraphs \textit{without double edges} are ASCs with a specific order in their simplices.

\begin{example}
Figure~\ref{fig:dig} shows a digraph $G$ containing a double edge, and Figure~\ref{fig:dflag} shows its DFC. Note that the set $\{0,2\}$ spans two different directed $2$-cliques, namely $[0,2]$ and $[2,0]$.
\end{example}

\begin{remark}
In this text, the colors used to fill the cliques serve as a visual representation of the corresponding higher-dimensional directed simplices.
\end{remark}

%--------------------------------------
\section{Weighted Directed Flag Complexes}
\label{sec:weighted-dfc}

Previously, we focused on DFCs derived from unweighted digraphs. However, real-world digraphs often include weights on vertices and/or arcs, making it beneficial to extend DFCs to incorporate such weights. Dawson \cite{Dawson} introduced weighted simplicial complexes with natural number weights, which were later generalized to weights in a commutative ring by Ren et al. \cite{Ren}. The following definition extends this generalization to \textit{weighted DFCs}.

\begin{definition}
A \textit{weighted ADSC} is a pair $(\mathcal{X}, \widetilde{\omega})$, where $\mathcal{X}$ is an ADSC and $\widetilde{\omega} : \mathcal{X} \rightarrow \mathcal{R}$ is a function of a commutative ring $\mathcal{R}$, such that $\widetilde{\omega}(\tau)$ divides $\widetilde{\omega}(\sigma)$ whenever $\tau \subseteq \sigma$. The convention $0/0 = 0$ is assumed.
\end{definition}

To simplify the association of a single weight with each directed clique, it is useful to transform arc weights into node weights. One practical method is the following degree-based transformation.

\begin{definition}\label{def:weights-edge-node}
Let $G^{\omega} = (V, E, \omega)$ be a weighted digraph with arc-weight function $\omega : E \rightarrow \mathcal{R}$, where $\mathcal{R}$ is a commutative ring. The \textit{node-weight function} $\tilde{\omega} : V \rightarrow \mathcal{R}$ is defined as
\begin{equation}\label{eq:weight-func2}
\tilde{\omega}(i) = \max\big( \deg^{-}_{\omega}(i), \deg^{+}_{\omega}(i) \big).
\end{equation}
\end{definition}

This approach ensures nonzero weights for non-isolated nodes and considers both in- and out-degrees, though it may overemphasize a few strong connections relative to many weak ones.

Using this transformation, we define a weight function on simplices as follows.

\begin{definition}\label{def:prod-weight}
Let $G^{\omega}$ be a weighted digraph and $\mathrm{dFl}(G^{\omega})$ its DFC. Let $\tilde{\omega}(i)$ be the weight of node $i$ in a directed $n$-simplex $\sigma^{(n)} = [0,1,\ldots,n]$. The \textit{product-weight function} $\widetilde{\omega} : \mathrm{dFl}(G^{\omega}) \rightarrow \mathcal{R}$ is given by
\begin{equation}\label{eq:prod-weight}
\widetilde{\omega}(\sigma^{(n)}) = \prod_{i = 0}^{n} \tilde{\omega}(i).
\end{equation}
The pair $(\mathrm{dFl}(G^{\omega}), \widetilde{\omega})$ is called the \textit{weighted directed flag complex} (or \textit{weighted DFC}).
\end{definition}

Thus, the weighted DFC is a natural example of a weighted ADSC, defined via the product of node weights derived from arc weights.

%%%%%%%%%%%%%%%%%%%%%%%%%%%%%%%%%%%%%%%%%%
\section{Directed Q-Analysis and Directed Higher-Order Adjacencies}
\label{sec:directed-q-analysis}

Next, we briefly outline the main ideas of classical Q-analysis, followed by those of directed higher-order connectivity between directed simplices corresponding to a directed version of Q-analysis.

%-------------------------------------------
\subsection{A Brief Introduction to Q-Analysis}
\label{sec:q-analysis}

What later became known as Q-analysis was introduced by R. H. Atkin  \cite{Atkin1974a, Atkin1974b, Atkin1976}, whose initial goal was to develop mathematical tools for analyzing structures associated with relations in the context of social systems \cite{Atkin1977} by modeling social networks through simplicial complexes and studying the connections among their simplices, thereby highlighting the importance of the concept of topological connectivity. Sometimes, Q-analysis is referred to as a ``language of structure.'' Since Atkin's seminal work, several extensions and applications of Q-analysis have been proposed \cite{Barcelo, Kramer}, together with many new concepts \cite{Johnson}.

One of its cornerstones is \textit{$q$-connectivity} for simplicial complexes, representing \textit{higher-order connectivity} across different levels of organization. Classical Q-analysis dealt only with higher-order connectivity without any consideration of connection directionality. 

This section rests on \cite{Atkin1977, Earl, Johnson}, with $\mathcal{X}$ standing for an arbitrary simplicial complex.

\begin{definition}\label{def:q-nearness-q-connectivity}
Given $\sigma^{(n)}, \tau^{(m)} \in \mathcal{X}$, they are called \textit{$q$-near} with $0 \le q \le \min(n,m)$ if they share a $q$-face, and in this case, we denote $\sigma^{(n)} \sim_{q} \tau^{(m)}$. In particular, we say that an $n$-simplex is $n$-near to itself. Also, $\sigma^{(n)}$ and $\tau^{(m)}$ are called \textit{$q$-connected} if there exists a finite number of simplices $\alpha_{i}^{(n_{i})} \in \mathcal{X}$, let's put $i=1, \ldots,l$, with $0 \le q \le \min(n,m,n_{1}, \ldots,n_{l})$, such that
\begin{equation}
\sigma^{(n)} \sim_{q_{0}} \alpha_{1}^{(n_{1})} \sim_{q_{1}} \ldots  \sim_{q_{l-1}} \alpha_{l}^{(n_{l})} \sim_{q_{l}} \tau^{(m)},
\end{equation}

\noindent where $q \le q_{j}$, for all $j=0, \ldots,l$, and in this case, we denote\footnote{As a convention, here we use the symbol $\sim_{q}$ to denote $q$-nearness and its bold version $\bm{\sim}_{\bm{q}}$ to denote $q$-connectivity.} $\sigma^{(n)} \bm{\sim}_{\bm{q}} \tau^{(m)}$. We call this sequence a \textit{chain of $q$-connection}. In addition, we say that $\sigma^{(n)}$ and $\tau^{(m)}$ are $q$-connected by a \textit{chain} of length $l$. In particular, an $n$-simplex is said to be $n$-connected to itself by a chain of length $0$. 
\end{definition}

%\begin{example}
%Consider the simplicial complex shown in Figure \ref{fig:example-q-conn1}. The simplices $\sigma$ and $\tau$ are $0$-connected by a chain of length $4$: $\sigma \sim_{1} \alpha_{1} \sim_{1} \alpha_{2} \sim_{1} \alpha_{3} \sim_{1} \alpha_{4} \sim_{0} \tau$. At the same time, they are $1$-connected  by a chain of length $5$: $\sigma \sim_{1} \alpha_{1} \sim_{1} \alpha_{5} \sim_{1} \alpha_{6} \sim_{1} \alpha_{7} \sim_{1} \alpha_{8} \sim_{1} \tau$.
%\end{example}

Note that $q$-nearness is a particular case of $q$-connectivity. Moreover, by definition, if two simplices are $q$-connected, then they are $q'$-connected for all $q' < q$.

\begin{proposition}
Given $\mathcal{X}$, define the set $\mathcal{X}_{q} = \{\sigma^{(n)} \in \mathcal{X} : q \le n \}$. The relation  ``is $q$-connected to'' ($\bm{\sim}_{\bm{q}}$) is an equivalence relation on the set $\mathcal{X}_{q}$.
\end{proposition}

A simple proof of the previous proposition can be found in \cite{Riihimaki}. Furthermore, the \textit{$q$-connected components} of $\mathcal{X}$ are defined as the elements of the quotient set $\mathcal{X}_{q}/\bm{\sim}_{\bm{q}}$. As introduced in Atkin's work, performing a \textit{Q-analysis} on a simplicial complex means computing its $q$-connected components for $0 \le q \le \dim \mathcal{X}$ and then summarizing the number of such components existing at each level $q$ into a \textit{structure vector}, $\mathbf{f}(\mathcal{X})$, whose entries count the number of $q$-connected components at each dimension.

Moreover, for $0 \le q \le \dim \mathcal{X}$, below we list and briefly describe some essential Q-analysis concepts that will be useful in this text:

\begin{itemize}
\item \textbf{$q$-graph}: A graph whose vertices are simplices in $\mathcal{X}_q$, with edges connecting $q$-near simplices.

\item \textbf{$q$-star} of a simplex $\sigma$: The set of all simplices $q$-near to $\sigma$, i.e. $\mathrm{st}_{q}(\sigma) = \{ \tau \in \mathcal{X} : \sigma \sim_{q} \tau \}$.

\item \textbf{Hub} of a simplicial family $\mathcal{F}$: The intersection of all simplices in $\mathcal{F}$, i.e. $\mathrm{hub}(\mathcal{F}) = \bigcap_{\sigma \in \mathcal{F}} \sigma$.

\item \textbf{Link} of a simplex $\sigma$: The set $\mathrm{lk}(\sigma) = \{\tau \in \mathcal{X} : \sigma \cap \tau = \emptyset \text{ and } \sigma \cup \tau \in \mathcal{X}\}$, generalizing the notion of a neighborhood in graphs.
\end{itemize}

\begin{example}
Considering the simplicial complex, $\mathcal{X}$, presented in Figure~\ref{fig:example-q-conn1}, we have the following examples: its structure vector is $\mathbf{f}(\mathcal{X}) = (2,2,11,1)$;  $\sigma \sim_{1} \alpha_{1} \sim_{1} \alpha_{2} \sim_{1} \alpha_{3} \sim_{1} \alpha_{4} \sim_{0} \tau$ and $\sigma \sim_{1} \alpha_{1} \sim_{1} \alpha_{5} \sim_{1} \alpha_{6} \sim_{1} \alpha_{7} \sim_{1} \alpha_{8} \sim_{1} \tau$; the $1$-star of the simplex $\alpha_{1}$ is $\mathrm{st}_{1}(\alpha_{1}) = \{ \sigma, \alpha_{2}, \alpha_{5}  \}$; the hub of $\mathcal{F} = \{ \alpha_{1},\alpha_{2}, \alpha_{3}, \alpha_{5}, \alpha_{6} \}$ is $\mbox{hub}(\mathcal{F}) = \{ 3 \}$; the link of the $1$-simplex $\{12, 13\}$ is $\{10, 11\}$.
\end{example}

%--------------------------------------------------------------
\subsection{Directed Q-Analysis and Directed Higher-Order Adjacencies}
\label{sec:dir-q-analysis}

As we saw in the previous section, Atkin's Q-analysis defines $q$-connectivity between two simplices solely in terms of the faces they share and does not specify the direction of this connection. Recently, however, H. Riihimäki \cite{Riihimaki} introduced a formalism to treat the $q$-connectivity between directed simplices, thus creating a directed analog of Q-analysis for directed flag complexes. 

Moreover, unlike adjacencies between vertices in a graph, when dealing with simplices, we can distinguish between two types of adjacencies: \textit{lower} and \textit{upper} adjacencies. Lower adjacencies compare how two simplices share their faces, and upper adjacencies tell us how they are nested in other higher-dimensional simplices \cite{Estrada, Goldberg, Serrano2020}. Accordingly, this section provides suitable definitions of lower and upper adjacencies for directed simplices.

In what follows, we present the concept of \textit{$(q, \hat{d_{i}}, \hat{d_{j}})$-connectivity} as introduced in \cite{Riihimaki}. We then extend the concepts of lower, upper, and general adjacencies, as presented in \cite{Serrano2020}, to directed simplices. We will limit our discussion to digraphs without double edges for simplicity, although the theory applies to digraphs with such structures. Henceforth, let $G$ denote a simple digraph \textit{without double edges}, and  $\mathrm{dFl}(G)$ denote the DFC associated with $G$.

%---------
\subsubsection{$(q, \hat{d_{i}}, \hat{d_{j}})$-Connectivity Between Directed Simplices}

The direction of the connection between two directed simplices is determined by a slightly modified face map, as defined below.

\begin{definition}\label{def:mod-face-map}
Given $\sigma^{(n)} = [v_{0}, \ldots,v_{n}] \in \mathrm{dFl}(G)$, the \textit{$i$-th face map} $\hat{d_{i}}$ is defined as 
\begin{equation}\label{eq:face-map}
\hat{d}_{i}(\sigma^{(n)}) = \begin{cases}
	[v_{0}, \ldots,\hat{v}_{i}, \ldots,v_{n}], \mbox{ if } i < n,\\
	[v_{0}, \ldots,v_{n-1}, \hat{v}_{n}], \mbox{ if } i \ge n.
\end{cases}
\end{equation}
\end{definition}

\begin{definition}\label{def:qij-nearness}
Let $\sigma^{(n)}, \tau^{(m)} \in \mathrm{dFl}(G)$. For $0 \le q \le \min(n,m)$, we say that $\sigma^{(n)}$ is \textit{$(q, \hat{d}_{i},\hat{d}_{j})$-near} to $\tau^{(m)}$ if either of the following conditions is satisfied:
\begin{enumerate}
\item $\sigma^{(n)} \subseteq \tau^{(m)}$;
\item $\hat{d_{i}}(\sigma^{(n)}) \supseteq \alpha^{(q)} \subseteq \hat{d_{j}}(\tau^{(m)})$, for some $\alpha^{(q)} \in \mathrm{dFl}(G)$ (i.e. if they share a $q$-face).
\end{enumerate}
\end{definition}

The previous definition established a concept of directionality in the connection between two directed simplices based on \textit{how} their faces are shared. To draw analogies with digraph definitions for the higher-order case, we propose the following definitions/notations.

\begin{definition}\label{def:qij-nearness-new}
For $\sigma^{(n)}, \tau^{(m)} \in \mathrm{dFl}(G)$ and for $0 \le q \le \min(n,m)$, we have the following definitions:

\begin{enumerate}
\item $\sigma^{(n)}$ is said to be \textit{in-$q$-near} (or \textit{$(-)$-$q$-near}) to  $\tau^{(m)}$ if they are $(q, d_{i}(\sigma^{(n)}), d_{j}(\tau^{(m)}))$-near with $i \ge j$. In this case, we denote $\sigma^{(n)} \sim_{q}^{-} \tau^{(m)}$.

\item $\sigma^{(n)}$ is said to be \textit{out-$q$-near} (or \textit{$(+)$-$q$-near}) to $\tau^{(m)}$ if they are $(q, d_{i}(\sigma^{(n)}), d_{j}(\tau^{(m)}))$-near with $i \le j$. In this case, we denote $\sigma^{(n)} \sim_{q}^{+} \tau^{(m)}$.

\item $\sigma^{(n)}$ is said to be \textit{bidirectionally $q$-near} (or \textit{$(\pm)$-$q$-near}) to $\tau^{(m)}$ if $\sigma^{(n)} \sim_{q}^{-} \tau^{(m)}$ and $\sigma^{(n)} \sim_{q}^{+} \tau^{(m)}$. In this case, we denote $\sigma^{(n)} \sim_{q}^{\pm} \tau^{(m)} = \tau^{(m)}  \sim_{q}^{\pm} \sigma^{(n)}$.
\end{enumerate}
\end{definition}

By definition, $\sigma^{(n)} \sim_{q}^{-} \tau^{(m)} = \tau^{(m)} \sim_{q}^{+} \sigma^{(n)} $ and $\sigma^{(n)} \sim_{q}^{+} \tau^{(m)} = \tau^{(m)} \sim_{q}^{-} \sigma^{(n)} $. Also, if $\sigma^{(n)} \subseteq \tau^{(m)}$, then $\sigma^{(n)} \sim_{q}^{\pm} \tau^{(m)}$.

\begin{notation}
We use the notation $\sigma^{(n)} \sim_{q}^{\bullet} \tau^{(m)}$, where  $\bullet \in \{-,+,\pm \}$, to represent the respective connectivity relation by replacing $\bullet$ with its respective symbol.
\end{notation}

\begin{definition}\label{def:dir-q-connectivity}
Given $\sigma^{(n)}, \tau^{(m)} \in \mathrm{dFl}(G)$, we say that $\sigma^{(n)}$ is \textit{$(\bullet)$-$q$-connected} to $\tau^{(m)}$, where $\bullet \in \{-,+\}$, if there exists a finite number of simplices $\alpha_{i}^{(n_{i})} \in \mathrm{dFl}(G)$, let's put $i=1, \ldots, l$, with $0 \le q \le \min(n,m,n_{1}, \ldots,n_{l})$, such that
\begin{equation}
\sigma^{(n)} \sim_{q_{0}}^{\bullet} \alpha_{1}^{(n_{1})} \sim_{q_{1}}^{\bullet}  \ldots  \sim_{q_{l-1}}^{\bullet}  \alpha_{l}^{(n_{l})} \sim_{q_{l}}^{\bullet}  \tau^{(m)},
\end{equation}

\noindent where $q \le q_{j}$, for all $j=0, \ldots,l$, and in this case we denote $\sigma^{(n)} \bm{\sim}_{\bm{q}}^{\bullet} \tau^{(m)}$. We say that $\sigma^{(n)}$ is $(\bullet)$-$q$-connected to $\tau^{(m)}$ by a \textit{directed $(\bullet)$-$q$-chain} of length $l$. We denote $\sigma^{(n)} \bm{\sim}_{\bm{q}}^{\pm} \tau^{(m)}$ if $\sigma^{(n)} \bm{\sim}_{\bm{q}}^{+} \tau^{(m)}$ and $\sigma^{(n)} \bm{\sim}_{\bm{q}}^{-} \tau^{(m)}$ and we say they are \textit{$(\pm)$-$q$-connected}. In particular, a directed $n$-simplex is said to be $(\pm)$-$q$-connected to itself by a directed $(\pm)$-$q$-chain of length $0$. Finally, for $\bullet \in \{-,+\}$, if $\sigma^{(n)}$ is not $(\bullet)$-$q$-connected to $\tau^{(m)}$ for all $\bullet$, we say that they are \textit{$q$-disconnected}.
\end{definition}

If two directed simplices are $(\bullet)$-$q$-near, $\bullet \in \{-,+\}$, then they are $(\bullet)$-$q$-connected by a directed $(\bullet)$-chain of length $0$. Moreover, if two directed simplices are $(\bullet)$-$q$-connected, then they are $(\bullet)$-$q'$-connected for all $q' < q$.

\begin{example}
Consider the DFC presented in Figure~\ref{fig:dfc-example}. We have the following relations:

\smallskip

1) $\theta \sim_{0}^{\pm} \tau$, since $\hat{d}_{0}(\theta) \supseteq  [2] \subseteq \hat{d}_{2}(\tau)$ and $\hat{d}_{2}(\theta) \supseteq  [2] \subseteq \hat{d}_{1}(\tau)$.

2) $\sigma \sim_{0}^{\pm} \tau$, since $\hat{d}_{1}(\sigma) \supseteq  [2] \subseteq \hat{d}_{2}(\tau)$ and $\hat{d}_{2}(\sigma) \supseteq  [2] \subseteq \hat{d}_{1}(\tau)$.

3) $\theta \sim_{1}^{+} \sigma$, since $\hat{d}_{0}(\theta) \supseteq  [2,6] \subseteq \hat{d}_{2}(\sigma)$.

4) $\alpha \sim_{1}^{+} \theta$, since $\hat{d}_{0}(\alpha) \supseteq  [1,6] \subseteq \hat{d}_{1}(\theta)$.

5) $\alpha \bm{\sim}_{\bm{0}}^{+} \tau$, since $\alpha \sim_{1}^{+} \theta \sim_{0}^{+} \tau$.
\end{example}

\begin{figure}[h!]
\centering
\begin{subfigure}{.45\textwidth}
\centering
\includegraphics[scale=0.69]{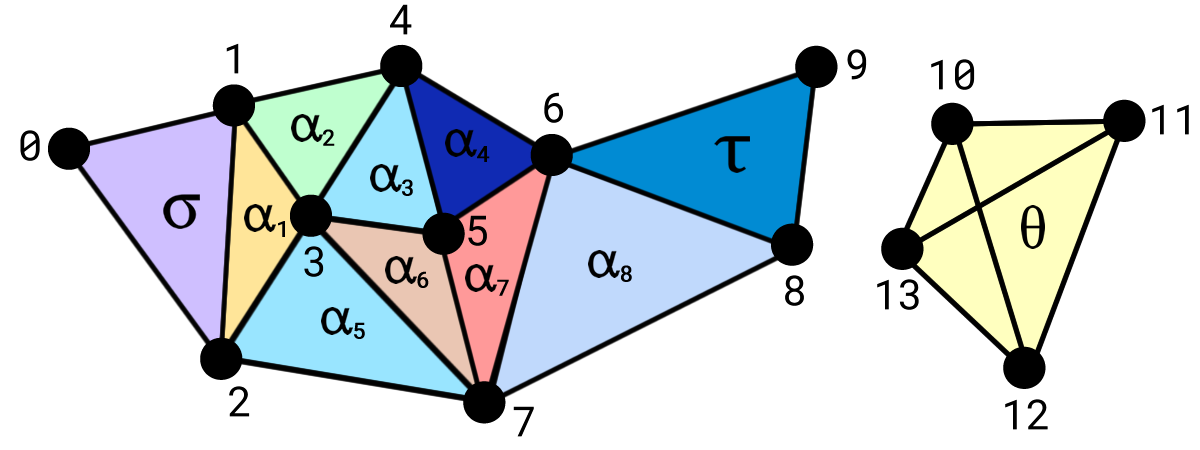}
\caption{A simplicial complex.}
\label{fig:example-q-conn1}
\end{subfigure}%
\begin{subfigure}{.45\textwidth}
\centering
\includegraphics[scale=0.7]{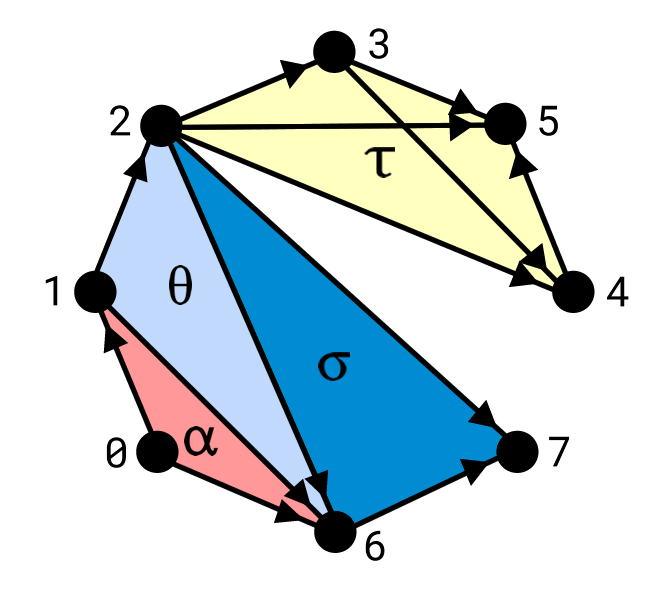}
\caption{A directed flag complex.}
\label{fig:dfc-example}
\end{subfigure}
\caption{A simplicial complex and a directed flag complex.}
\label{fig:flag-complex-3}
\end{figure}

%------ LOWER ADJ --------
\subsubsection{Lower, Upper, and General Adjacencies}

In what follows, we extend the definitions of lower, upper, and general adjacencies from \cite{Serrano2020} to directed simplices.

\begin{definition}
For $\sigma^{(n)}, \tau^{(m)} \in \mathrm{dFl}(G)$ and for $0 \le q \le \min(n,m)$, we have the following definitions:

\begin{enumerate}
\item $\sigma^{(n)}$ is \textit{lower $(\bullet)$-$q$-adjacent} to  $\tau^{(m)}$, where $\bullet \in \{-, +, \pm\}$, if and only if $\sigma^{(n)}$ is $(\bullet)$-q-near to $\tau^{(m)}$ , i.e.
$$
\sigma^{(n)} \sim_{L_{q}}^{\bullet} \tau^{(m)} \iff \sigma^{(n)}  \sim_{q}^{\bullet} \tau^{(m)}.
$$

\item $\sigma^{(n)}$ is \textit{strictly lower $(\bullet)$-$q$-adjacent} to $\tau^{(m)}$, where $\bullet \in \{-, +, \pm\}$, if and only if $\sigma^{(n)} \sim_{L_{q}}^{\bullet} \tau^{(m)}$ and $\sigma^{(n)}$ is not ($\star$)-$(q+1)$-adjacent to $\tau^{(m)}$, for all $\star \in \{-,+,\pm\}$, i.e.
$$
\sigma^{(n)} \sim_{L_{q^{*}}}^{\bullet} \tau^{(m)} \iff \sigma^{(n)} \sim_{L_{q}}^{\bullet} \tau^{(m)} \mbox{ and }  \sigma^{(n)} \not\sim_{L_{q+1}}^{\star} \tau^{(m)}, \mbox{ } \forall \star.
$$
\end{enumerate}
\end{definition}

We point out that lower $(\bullet)$-$q$-adjacency and $(\bullet)$-$q$-nearness are exactly the same definitions, thus we propose \textit{lower $(\bullet)$-$q$-nearness} as an alternative nomenclature to $(\bullet)$-$q$-nearness. Also, vertices are only lower $(\bullet)$-$0$-adjacent to themselves.

In addition to lower adjacency, which is the sharing of faces between two simplices, we can consider \textit{upper adjacency}, which is the nesting of simplices in other simplices of larger dimensions. We must specify exactly how the directionality between two directed simplices that are faces of other directed simplices of higher dimension is accounted for when extending upper adjacency to directed simplices. Thus, based on Definition \ref{def:qij-nearness}, we propose the definition of \textit{upper $(p, \hat{d}_{i}, \hat{d}_{j})$-nearness} as follows.

\begin{definition}
Let $\sigma^{(n)}, \tau^{(m)} \in \mathrm{dFl}(G)$. For $n,m < p \le \dim \mathrm{dFl}(G)$, $\sigma^{(n)}$ is said to be \textit{upper $(p, \hat{d}_{i}, \hat{d}_{j})$-near} to $\tau^{(m)}$ if the following condition is true:
$$
\sigma^{(n)} = \hat{d_{i}}(\Theta^{(n+1)}) \subseteq \Theta^{(p)} \supseteq  \hat{d_{j}}(\Theta^{(m+1)}) = \tau^{(m)}, \mbox{ for some } \Theta^{(n+1)}, \Theta^{(m+1)} \subseteq \Theta^{(p)} \in \mathrm{dFl}(G).
$$
\end{definition}

\begin{definition}
For $\sigma^{(n)}, \tau^{(m)} \in \mathrm{dFl}(G)$ and for $n,m  < p \le \dim \mathrm{dFl}(G)$, we have the following definitions:

\begin{enumerate}
\item $\sigma^{(n)}$ is \textit{upper $(-)$-$p$-adjacent} to $\tau^{(m)}$  if and only if $\sigma^{(n)}$ is upper $(p, \hat{d}_{i}, \hat{d}_{j})$-near to $\tau^{(m)}$ and $i \ge j$, i.e.
$$
\sigma^{(n)} \sim_{U_{p}}^{-} \tau^{(m)} \iff \sigma^{(n)} \mbox{ is upper } (p, \hat{d}_{i}, \hat{d}_{j})\mbox{-near to } \tau^{(m)} \mbox{ with } i \ge j.
$$

\item $\sigma^{(n)}$ is  \textit{upper $(+)$-$p$-adjacent} to $\tau^{(m)}$  if and only if $\sigma^{(n)}$ is upper $(p, \hat{d}_{i}, \hat{d}_{j})$-near to $\tau^{(m)}$ and $i \le j$, i.e.
$$
\sigma^{(n)} \sim_{U_{p}}^{+} \tau^{(m)} \iff \sigma^{(n)} \mbox{ is upper } (p, \hat{d}_{i}, \hat{d}_{j})\mbox{-near to } \tau^{(m)} \mbox{ with } i \le j.
$$

\item $\sigma^{(n)}$ is \textit{upper $(\pm)$-$p$-adjacent} to $\tau^{(m)}$  if and only if $\sigma^{(n)}$ is upper $(-)$-$p$-adjacent and upper $(+)$-$p$-adjacent to $\tau^{(m)}$, i.e.
$$
\sigma^{(n)} \sim_{U_{p}}^{\pm} \tau^{(m)} \iff \sigma^{(n)} \sim_{U_{p}}^{-} \tau^{(m)} \mbox{ and }  \sigma^{(n)} \sim_{U_{p}}^{+} \tau^{(m)}.
$$

\item $\sigma^{(n)}$ is \textit{strictly $(\bullet)$-$p$-upper adjacent} to $\tau^{(m)}$, where $\bullet \in \{-,+,\pm\}$, if and only if $\sigma^{(n)} \sim_{U_{p}}^{\bullet} \tau^{(m)}$ and $\sigma^{(n)}$ is not $(\star)$-$(p+1)$-adjacent to $\tau^{(m)}$, for all $\star \in \{-,+,\pm\}$, i.e.
$$
\sigma^{(n)} \sim_{U_{p^{*}}}^{\bullet} \tau^{(m)} \iff \sigma^{(n)} \sim_{U_{p}}^{\bullet} \tau^{(m)} \mbox{ and } \sigma^{(n)} \not\sim_{U_{p+1}}^{\star} \tau^{(m)}, \mbox{ } \forall \star.
$$
\end{enumerate}
\end{definition}

If $(v,u)$ is an arc in $G$, then $v \sim_{U_{1}}^{+} u$. On the other hand, if $(u,v)$ is an arc in $G$, then $v \sim_{U_{1}}^{-} u$.

\begin{definition}\label{def:general-adj}
For $\sigma^{(n)}, \tau^{(m)} \in \mathrm{dFl}(G)$, we have the following definitions:

\begin{enumerate}
\item $\sigma^{(n)}$ is \textit{$(\bullet)$-$q$-adjacent} to $\tau^{(m)}$, where $\bullet \in \{-,+,\pm\}$, if and only if $\sigma^{(n)}$ is strictly lower $(\bullet)$-$q$-adjacent to $\tau^{(m)}$ and $\sigma^{(n)}$ is not upper $(\star)$-$p$-adjacent to $\tau^{(m)}$, with $p = n+m-q$, for all $\star \in \{-, +, \pm\}$, i.e.
$$
\sigma^{(n)} \sim_{A_{q}}^{\bullet} \tau^{(m)} \iff \sigma^{(n)} \sim_{L_{q^{*}}}^{\bullet} \tau^{(m)}  \mbox{ and } \sigma^{(n)} \not\sim^{\star}_{U_{p}} \tau^{(m)}, \mbox{ } \forall \star.
$$

\item $\sigma^{(n)}$ is \textit{maximal $(\bullet)$-$q$-adjacent} to $\tau^{(m)}$, where $\bullet \in \{-,+,\pm\}$, if and only if  $\sigma^{(n)}$ is $(\bullet)$-$q$-adjacent to $\tau^{(m)}$ and $\sigma^{(n)}$ is not a face of any other directed simplex which is $(\star)$-$q$-adjacent to $\tau^{(m)}$,  for all $\star \in \{-, +, \pm\}$, i.e.
$$
\sigma^{(n)} \sim_{A_{q^{*}}}^{\bullet} \tau^{(m)} \iff \sigma^{(n)} \sim_{A_{q}}^{\bullet} \tau^{(m)} \mbox{ and } \sigma^{(n)} \not\subset \sigma^{(r)}, \mbox{ } \forall \sigma^{(r)} \mbox{ : } \sigma^{(r)} \sim_{A_{q}}^{\star} \tau^{(m)}, \mbox{ } \forall \star.
$$
\end{enumerate}
\end{definition}

\begin{remark}
The quantity $p=m+n-q$ comes from the fact that if $\sigma^{(n)} \sim_{L_{q^{*}}}^{\bullet} \tau^{(m)}$, then they share $(q+1)$ vertices and thus the smallest directed simplex which might contain $\sigma^{(n)}$ and $\tau^{(m)}$ as faces must have $(n+1)+(m+1) - (q+1)$ vertices.
\end{remark}

\begin{definition}\label{def:set-maximal-simplices}
Given $\mathrm{dFl}(G)$ and $0 \le q \le \dim \mathrm{dFl}(G)$, we define the following subsets:

\begin{enumerate}
\item $\mathrm{dFl}_{q}(G) = \{ \sigma^{(n)} \in \mathrm{dFl}(G) : q \le n \}.$

\item $\mathrm{dFl}^{\ast}(G) = \{ \sigma^{(n)} \in \mathrm{dFl}(G) : \sigma^{(n)} \mbox{ is maximal} \}.$

\item $\mathrm{dFl}^{\ast}_{q}(G) = \{ \sigma^{(n)} \in \mathrm{dFl}^{\ast}(G) : q \le n \}.$
\end{enumerate}
\end{definition}

algebraic-topological invariants
\begin{proposition}\label{prop:maximal-adjacency}
For $\sigma^{(n)}, \tau^{(m)} \in \mathrm{dFl}^{\ast}(G)$ and for $\bullet \in \{-,+,\pm \}$, we have the following equivalence:
$$
\sigma^{(n)} \sim^{\bullet}_{L_{q^{*}}} \tau^{(m)} \iff \sigma^{(n)} \sim^{\bullet}_{A_{q^{*}}} \tau^{(m)}.
$$
\end{proposition}

\begin{proof}
Suppose $\sigma^{(n)} \sim^{\bullet}_{L_{q^{*}}} \tau^{(m)}$. Since $\sigma^{(n)}, \tau^{(m)} \in \mathrm{dFl}^{\ast}(G)$, both simplices are not faces of any other simplices in $\mathrm{dFl}(G)$, then $\sigma^{(n)} \sim^{\bullet}_{A_{q^{*}}} \tau^{(m)}$. On the other hand, if $\sigma^{(n)} \sim^{\bullet}_{A_{q^{*}}} \tau^{(m)}$, by definition, $\sigma^{(n)} \sim^{\bullet}_{L_{q^{*}}} \tau^{(m)}$.
\end{proof}

%------ WALKS --------
\subsubsection{Directed Simplicial $q$-Walks and $q$-Distances} 

In Definition \ref{def:dir-q-connectivity}, we implicitly defined $(\bullet)$-$q$-connectivity in terms of \textit{lower} $(\bullet)$-$q$-adjacency; therefore, from now on, we will refer to this connectivity as \textit{lower $(\bullet)$-$q$-connectivity}. In what follows, we introduce the concept of \textit{maximal $(\bullet)$-$q$-connectivity}, which is the basis for defining the idea of \textit{maximal directed simplicial $q$-walk}.

\begin{definition}\label{def:maximal-q-connectivity}
Given $\sigma^{(n)}, \tau^{(m)} \in \mathrm{dFl}(G)$, we say that $\sigma^{(n)}$ is \textit{maximal $(\bullet)$-$q$-connected} to $\tau^{(m)}$, where $\bullet \in \{-,+\}$, if there exists a finite number of simplices $\alpha_{i}^{(n_{i})} \in \mathrm{dFl}(G)$, let's put $i=1, \ldots,l$, with $0 \le q \le \min(n,m,n_{1}, \ldots,n_{l})$, such that
\begin{equation}
\sigma^{(n)} \sim_{A_{q_{0}^{*}}}^{\bullet} \alpha_{1}^{(n_{1})} \sim_{A_{q_{1}^{*}}}^{\bullet}  \ldots  \sim_{A_{q_{l-1}^{*}}}^{\bullet}  \alpha_{l}^{(n_{l})} \sim_{A_{q_{l}^{*}}}^{\bullet}  \tau^{(m)},
\end{equation}

\noindent where $q \le q_{j}$, for all $j=0, \ldots,l$, and in this case we denote $\sigma^{(n)} \bm{\sim}_{\bm{A_{q^{*}}}}^{\bullet} \tau^{(m)}$. We say that there is a \textit{directed simplicial $q$-walk} of length $l$ from $\sigma^{(n)}$ to $\tau^{(m)}$ if $\sigma^{(n)} \bm{\sim}_{\bm{A_{q^{*}}}}^{+} \tau^{(m)}$ or from $\tau^{(m)}$ to $\sigma^{(n)}$ if $\sigma^{(n)} \bm{\sim}_{\bm{A_{q^{*}}}}^{-} \tau^{(m)}$. We denote $\sigma^{(n)} \bm{\sim}_{\bm{A_{q^{*}}}}^{\pm} \tau^{(m)}$ if $\sigma^{(n)} \bm{\sim}_{\bm{A_{q^{*}}}}^{+} \tau^{(m)}$ and $\sigma^{(n)} \bm{\sim}_{\bm{A_{q^{*}}}}^{-} \tau^{(m)}$ and we say they are \textit{maximal $(\pm)$-$q$-connected}. Moreover, for $\bullet \in \{-,+\}$, if $\sigma^{(n)}$ is not maximal $(\bullet)$-$q$-connected to $\tau^{(m)}$ for all $\bullet$,  we say that they are \textit{maximal $q$-disconnected}.
\end{definition}

\begin{remark}\label{rem:lower_case}
If we replace the maximal adjacency $A_{q^{*}}$ by the lower adjacency $L_{q}$ throughout Definition~\ref{def:maximal-q-connectivity}, we obtain the analogous notion of \textit{lower $(\bullet)$-$q$-connectivity}. All remaining conventions (directed simplicial $q$-walks, the $(-, +, \pm)$ notations, and $q$-disconnectedness) carry over with $A_{q^{*}}$ replaced by $L_{q}$ and  ``maximal'' replaced by ``lower'' throughout.
\end{remark}

Notice that if $\sigma \bm{\sim}_{\bm{A_{q^{*}}}}^{+} \tau$ (respec. $\bm{L_{q}}$) with $\sigma = \tau$, then we have a \textit{maximal} (respec. \textit{lower}) \textit{directed simplicial $q$-cycle}.

\begin{definition}
Given two directed simplices $\sigma, \tau \in \mathrm{dFl}(G)$, the (\textit{lower}) \textit{maximal directed simplicial $q$-distance} from $\sigma$ to $\tau$, denoted by $\vec{d}_{q}(\sigma, \tau)$, is equal to the length of the shortest (lower) maximal directed simplicial $q$-walk from $\sigma$ to $\tau$. If $\sigma$ and $\tau$ are $q$-disconnected, then we define $\vec{d}_{q}(\sigma, \tau) = \infty$.
\end{definition}

\begin{remark}
When not explicitly stated, we may use the notations $\vec{d}_{q}^{L}(\sigma, \tau)$ and $\vec{d}_{q}^{A}(\sigma, \tau)$ to distinguish between the lower and maximal cases.
\end{remark}

As a matter of fact, $\vec{d}_{q}$ is a quasi-distance (that is, it satisfies the identity property, $\vec{d}_{q}(\sigma, \tau)=0$ $\Leftrightarrow$ $\sigma=\tau$, and the triangular inequality) since the property of symmetry is not necessarily satisfied.

%------------
\subsubsection{Weakly and Strongly $q$-Connected Components and Structure Vectors}

As mentioned, performing a Q-analysis of a simplicial complex involves calculating its $q$-connected components and constructing its structure vector. In analogy to digraphs, however, DFCs feature two different kinds of ``$q$-connected components,'' namely: \textit{weakly $q$-connected components} and \textit{strongly $q$-connected components}.

In analogy to Serrano and Gómez \cite{Serrano2020}, a directed simplex is not maximal $(\bullet)$-$q$-connected to itself $\bullet \in \{-, +, \pm\}$; thus, this relation is not reflexive; accordingly, to obtain an equivalence relation, we introduce the following relation:
\begin{equation}\label{eq:strong-q-connected-rel}
(\sigma^{(n)}, \tau^{(m)}) \in S^{s}_{q} \iff \begin{cases}
\sigma^{(n)} \bm{\sim}_{\bm{A_{q^{*}}}}^{\pm} \tau^{(m)}, \\
\mbox{or } \sigma^{(n)} = \tau^{(m)}.
\end{cases}
\end{equation}

\begin{proposition}\label{prop:strongly-q-connected-comp}
The relation $S^{s}_{q}$ is an equivalence relation on $\mathrm{dFl}_{q}(G)$.
\end{proposition}
\begin{proof}
By definition, for every $\sigma \in \mathrm{dFl}_{q}(G)$, we have $(\sigma, \sigma) \in S_{q}^{s}$ (reflexivity). For arbitrary $\sigma, \tau, \theta \in \mathrm{dFl}_{q}(G)$, if $(\sigma, \tau) \in S_{q}^{s}$, then, by definition, $(\tau, \sigma) \in S_{q}^{s}$ (symmetry), and if $(\sigma, \tau), (\tau, \theta) \in S_{q}^{s}$, then we can find a finite number of directed simplices in $\mathrm{dFl}_{q}(G)$ such that $\sigma \bm{\sim}_{\bm{A_{q^{*}}}}^{+} \theta$ and $\sigma \bm{\sim}_{\bm{A_{q^{*}}}}^{-} \theta$, i.e. $(\sigma, \theta)  \in S_{q}^{s}$ (transitivity).  
\end{proof}

\begin{definition}\label{def:strongly-q-connected-comp}
The \textit{maximal strongly $q$-connected components} of $\mathrm{dFl}(G)$ are the equivalence classes of the quotient set $\mathrm{dFl}_{q}(G)/ S^{s}_{q}$, which are the equivalence classes of maximal $(\pm)$-$q$-connected directed simplices.
\end{definition}

Furthermore, if we disregard the directionality of the connections between the directed simplices in the relation (\ref{eq:strong-q-connected-rel}), we obtain the maximal $q$-connectivity as defined in \cite{Serrano2020}; thus, denoting by $\sigma^{(n)} \bm{\sim}_{\bm{A_{q^{*}}}} \tau^{(m)}$ the maximal $q$-connectivity between $\sigma^{(n)}$ and $\tau^{(m)}$, we define the following equivalence relation:
\begin{equation}\label{eq:weak-q-connected-rel}
(\sigma^{(n)}, \tau^{(m)}) \in S^{w}_{q} \iff \begin{cases}
\sigma^{(n)} \bm{\sim}_{\bm{A_{q^{*}}}} \tau^{(m)}, \\
\mbox{or } \sigma^{(n)} = \tau^{(m)}.
\end{cases}
\end{equation}

\begin{definition}\label{def:weakly-q-connected-comp}
The \textit{maximal weakly $q$-connected components} of $\mathrm{dFl}(G)$ are the equivalence classes of the quotient set $\mathrm{dFl}_{q}(G)/ S^{w}_{q}$, which are the equivalence classes of maximal $q$-connected directed simplices.
\end{definition}

\begin{remark}
The previous definitions apply equally to the lower ($\pm$)-$q$-adjacency, since it's reflexive by definition.
\end{remark}

In Subsection \ref{sec:q-analysis} we have defined the (first) structure vector based on the number of $q$-connected components of a simplicial complex $\mathcal{X}$, nonetheless, Andjelković et al. \cite{Andjelkovic2015} described two additional different structure vectors, the \textit{second} and the \textit{third} structure vectors, which are based, respectively, on the number of simplices in the set $\mathcal{X}_{q}$ and the ``degree of connectedness'' among the simplices at level $q$. In Definition \ref{def:topological-structure-vectors-1}, we extend these vectors for the directed case.

%------q-DIGRAPHS --------
\subsubsection{Maximal and Lower $q$-Digraphs} 

Since we previously defined $\mathrm{dFl}^{\ast}_{q}(G)$ as the set of all maximal directed simplices with dimensions greater than or equal to $q$, advancing in level $q$, or changing the level of organization, of the complex (see Figure \ref{fig:DFC-level-of-organization}), gives us a a new perspective of its higher-order topology. Consequently, we may learn more about the higher-order topology (or the clique organization) of the underlying network.

\begin{figure}[h!]
\centering
\begin{subfigure}{.23\textwidth}
\centering
\includegraphics[scale=0.42]{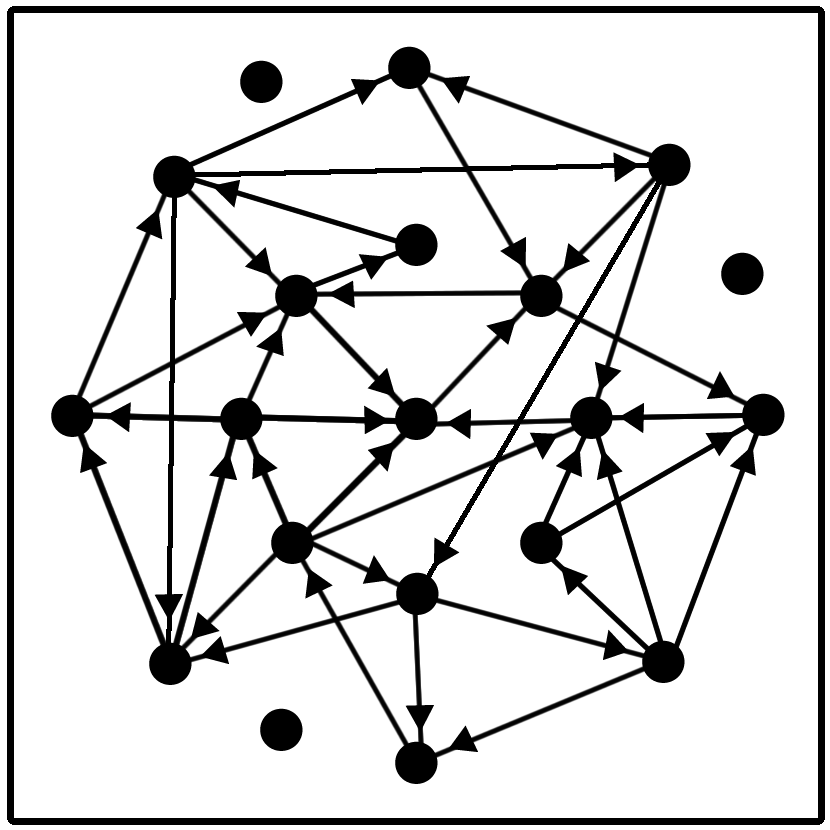}
\caption{Digraph $G$.}
\label{high-adj1}
\end{subfigure}%
\begin{subfigure}{.22\textwidth}
\centering
\includegraphics[scale=0.42]{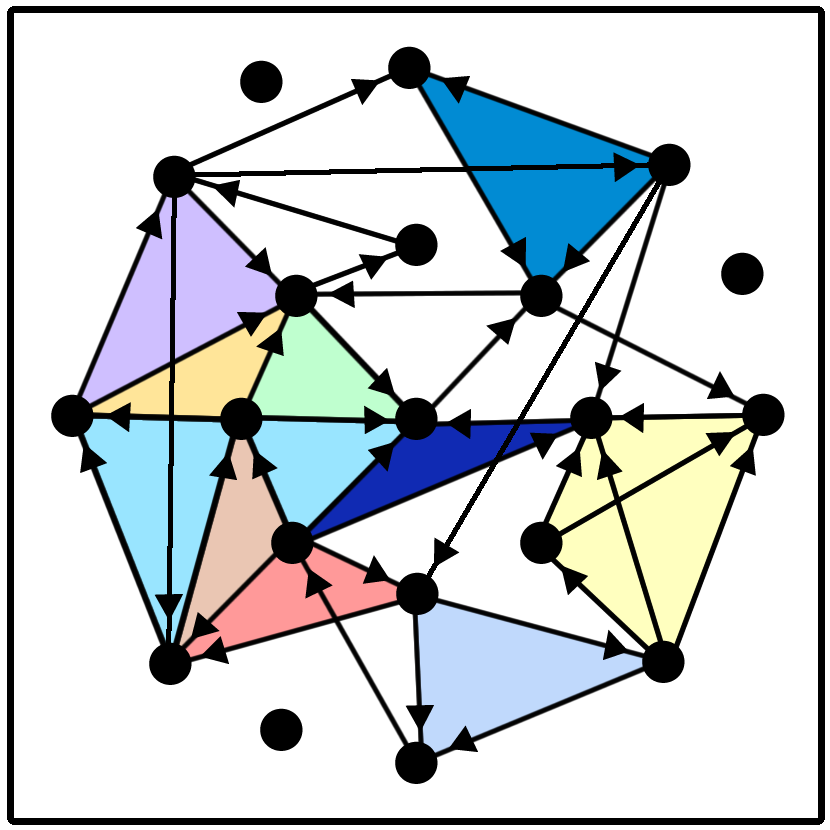}
\caption{$\mathrm{dFl}^{\ast}_{0}(G)$.}
\label{high-adj2}
\end{subfigure}
\begin{subfigure}{.22\textwidth}
\centering
\includegraphics[scale=0.42]{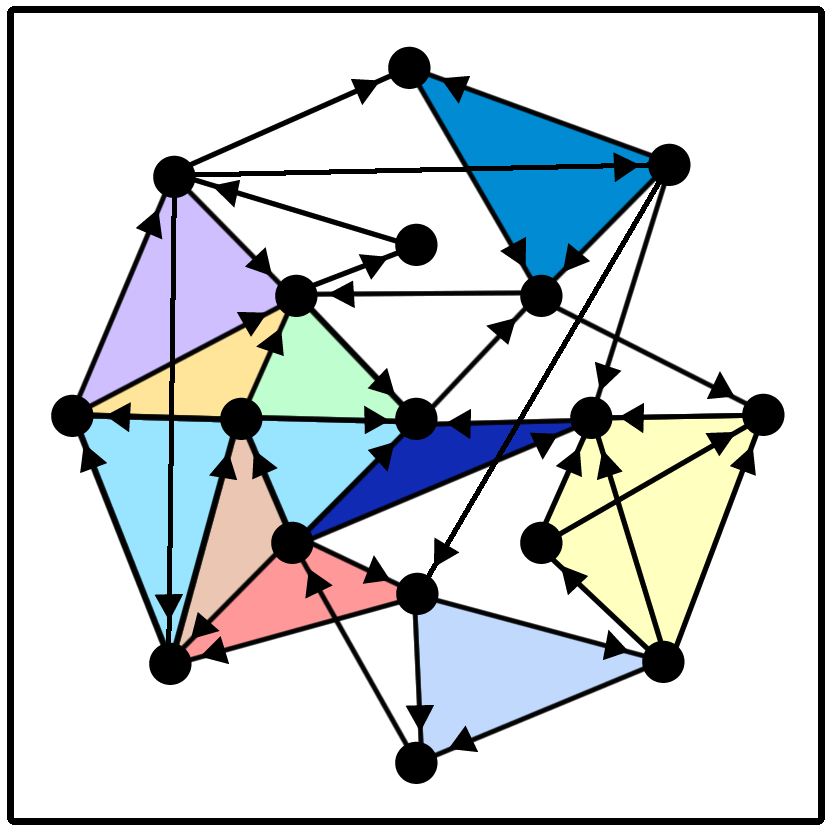}
\caption{$\mathrm{dFl}^{\ast}_{1}(G)$.}
\label{high-adj2}
\end{subfigure}
\begin{subfigure}{.22\textwidth}
\centering
\includegraphics[scale=0.42]{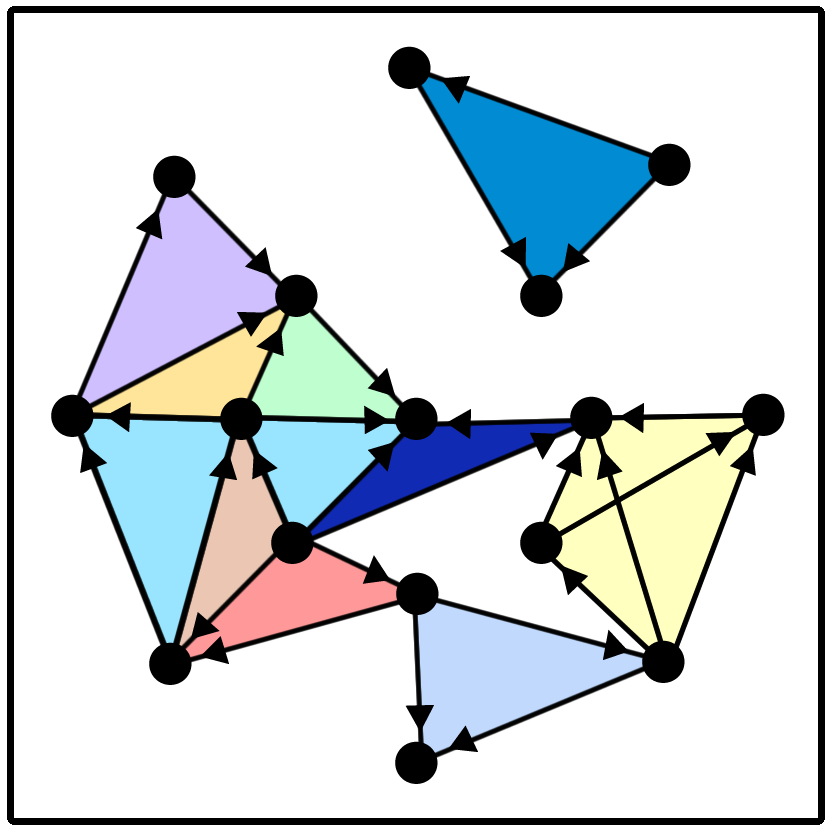}
\caption{$\mathrm{dFl}^{\ast}_{2}(G)$.}
\label{high-adj1}
\end{subfigure}%
\caption{Graphical representation of the maximal directed simplices for each level $q=0,1,2$.}
\label{fig:DFC-level-of-organization}
\end{figure}

Accordingly, we can go further and consider the directed higher-order connectivity among the simplices at each level $q$, and then define a ``higher-order digraph'' for each level. A definition of ``higher-order digraph'' was previously described in \cite{Caputi} using the concept of a \textit{$q$-digraph}; however, since we have introduced different types of adjacencies, we formalize two definitions here: \textit{maximal} and \textit{lower $q$-digraph}.

\begin{definition}\label{def:q-digraph}
The \textit{maximal} (respectively \textit{lower}) \textit{$q$-digraph} of $\mathrm{dFl}(G)$, denoted by $\mathcal{G}_{q}^{A}$ (respectively $\mathcal{G}_{q}^{L}$), is the digraph whose vertices are the simplices of $\mathrm{dFl}^{\ast}_{q}(G)$ and for each pair $\sigma, \tau \in \mathrm{dFl}^{\ast}_{q}(G)$ there is a directed edge from $\sigma$ to $\tau$ if $\sigma \sim^{+}_{A_{q^{*}}} \tau$ (respectively $\sigma \sim^{+}_{L_{q}} \tau$), with $0 \le q \le \dim \mathrm{dFl}(G)$.
\end{definition}

In addition, when $\sigma \sim^{+}_{A_{q^{*}}} \tau$ (respec $\sigma \sim^{+}_{L_{q}} \tau$), for $\sigma, \tau \in \mathrm{dFl}^{\ast}_{q}(G)$, we say that there is a \textit{maximal} (respec. \textit{lower}) \textit{$q$-arc} from $\sigma$ to $\tau$ and in this case we denote $(\sigma, \tau)_{A}$ (respec. $(\sigma, \tau)_{L}$). When the context is clear, we will simply denote $(\sigma, \tau)$ and call it a \textit{$q$-arc}. Also, we may use the notation $\mathcal{G}^{A}_{q} = (\mathcal{V}_{q}, \mathcal{E}^{A}_{q})$ (respec. $\mathcal{G}^{L}_{q} = (\mathcal{V}_{q}, \mathcal{E}^{L}_{q})$), where $\mathcal{V}_{q} = \mathrm{dFl}^{\ast}(G)$ and $\mathcal{E}^{A}_{q}$ (respec. $\mathcal{E}^{L}_{q}$) is the set of all maximal (respec. lower) $q$-arcs $(\sigma, \tau)$.

\begin{definition}\label{def:q-adjacency-matrix}
Let $\mathcal{G}^{A}_{q}$ be the maximal $q$-digraph of $\mathrm{dFl}(G)$. The \textit{maximal $q$-adjacency matrix} of $\mathcal{G}^{A}_{q}$, denoted by $\mathcal{H}^{A}_{q} = \mathcal{H}_{q}^{A}(\mathcal{G}_{q})$, is a real square matrix whose entries are given by
\begin{equation}\label{eq:q-adjacency-matrix}
\big( \mathcal{H}^{A}_{q}\big)_{ij} = \begin{cases}
	1, \mbox{ if } \sigma_{i} \sim^{+}_{A_{q^{*}}} \sigma_{j},\\
	0, \mbox{ if } i = j \mbox{ or } \sigma_{i} \not\sim^{+}_{A_{q^{*}}} \sigma_{j}.
\end{cases}
\end{equation}

Similarly, we define the \textit{lower $q$-adjacency matrix} of $\mathcal{G}^{L}_{q}$ by
\begin{equation}\label{eq:lower-q-adjacency-matrix}
\big( \mathcal{H}^{L}_{q}\big)_{ij} = \begin{cases}
	1, \mbox{ if } \sigma_{i} \sim^{+}_{L_{q}} \sigma_{j},\\
	0, \mbox{ if } i = j \mbox{ or } \sigma_{i} \not\sim^{+}_{L_{q}} \sigma_{j}.
\end{cases}
\end{equation}
\end{definition}

\begin{example}
Consider the DFC presented in Figure \ref{fig:dfc1}. Figures \ref{fig:q-digraph1}, \ref{fig:q-digraph2}, and \ref{fig:q-digraph3} represent its respective lower $q$-digraphs $\mathcal{G}_{q}$ for $q=0,1,2$.
\end{example}

\begin{figure}[h!]
\centering
\begin{subfigure}{.23\textwidth}
\centering
\includegraphics[scale=0.42]{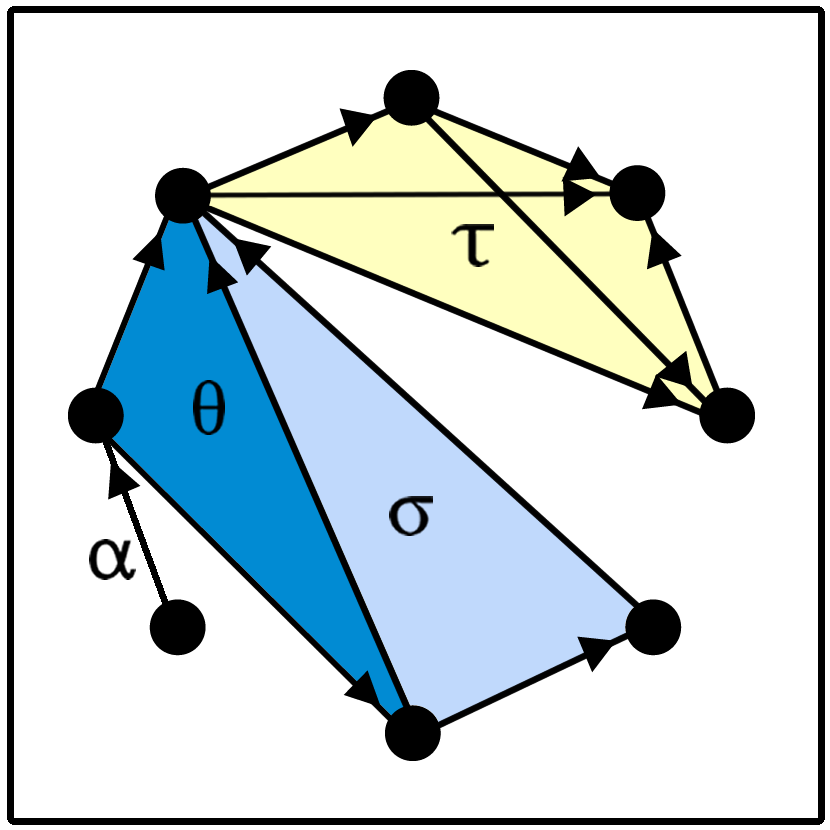}
\caption{$\mathrm{dFl}(G)$.}
\label{fig:dfc1}
\end{subfigure}%
\begin{subfigure}{.22\textwidth}
\centering
\includegraphics[scale=0.42]{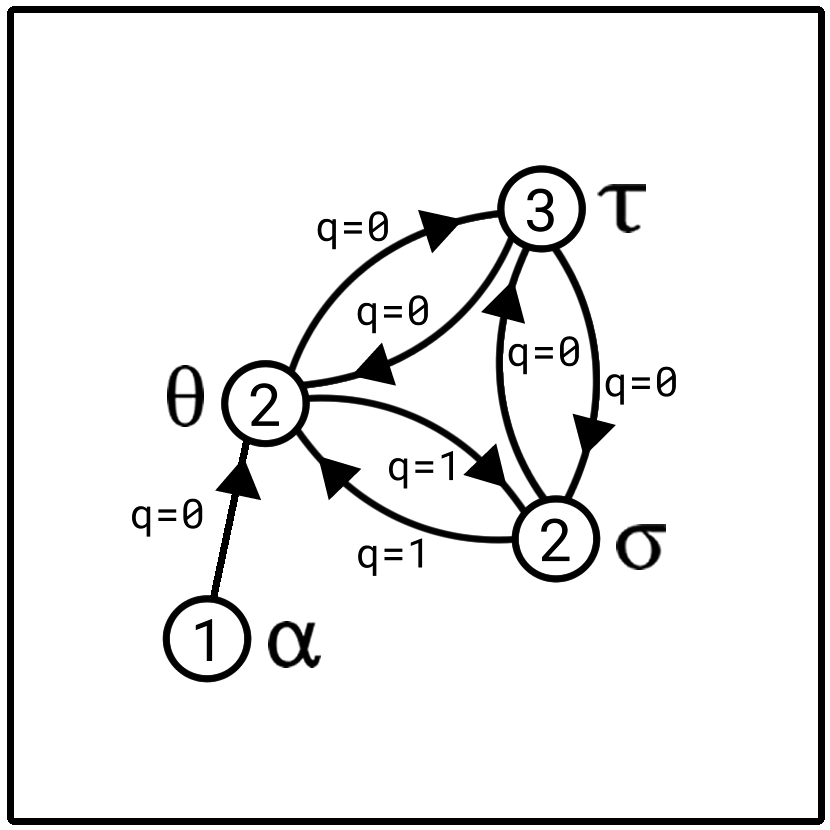}
\caption{$\mathcal{G}^{L}_{0}$.}
\label{fig:q-digraph1}
\end{subfigure}
\begin{subfigure}{.22\textwidth}
\centering
\includegraphics[scale=0.42]{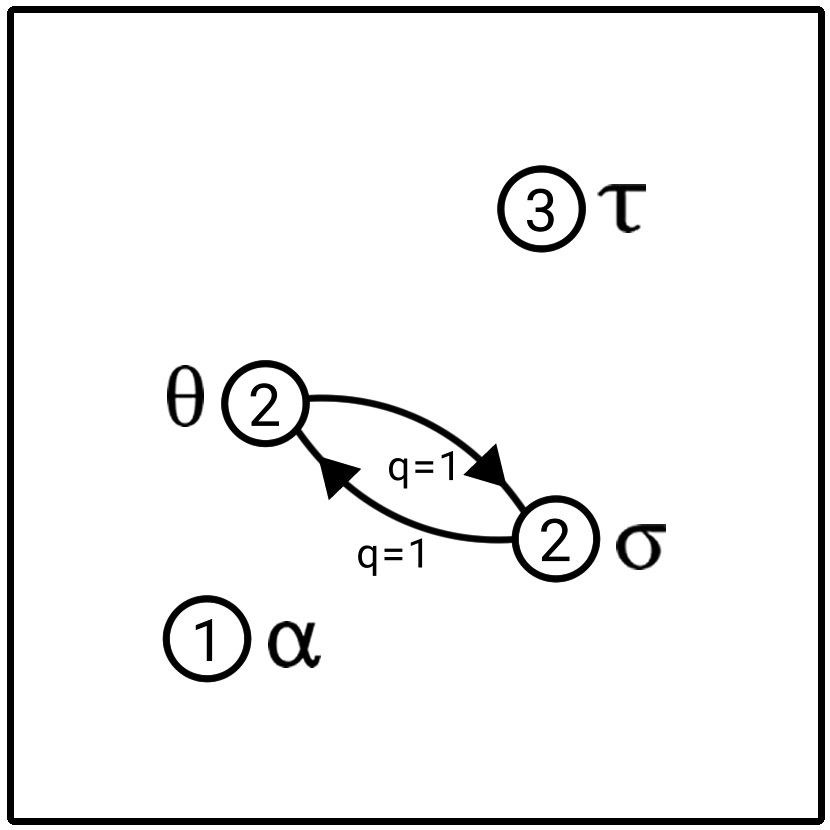}
\caption{$\mathcal{G}^{L}_{1}$.}
\label{fig:q-digraph2}
\end{subfigure}
\begin{subfigure}{.22\textwidth}
\centering
\includegraphics[scale=0.42]{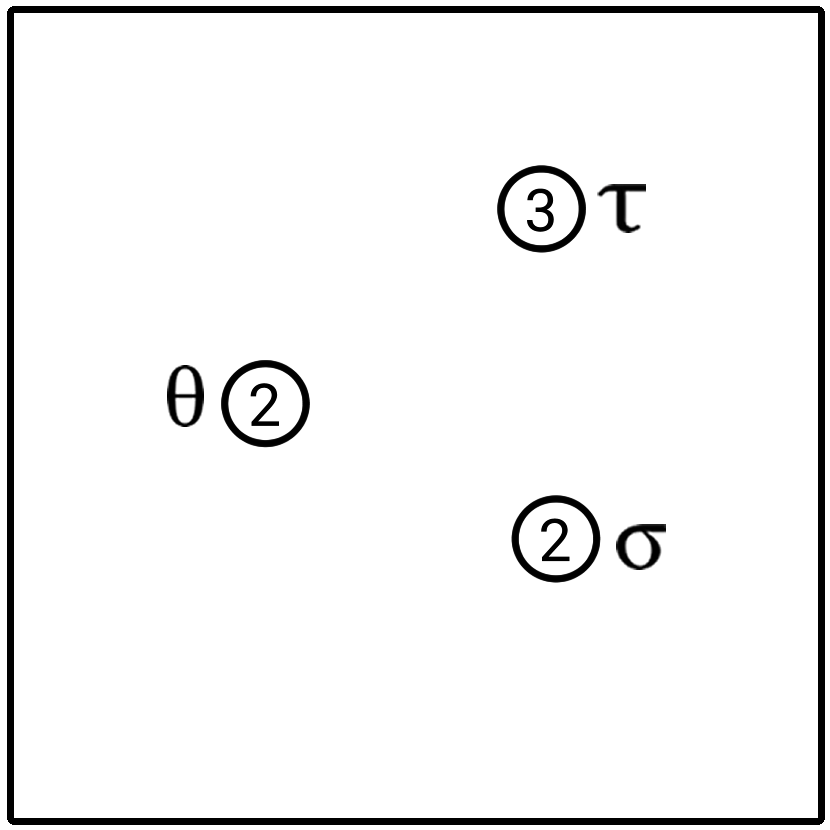}
\caption{$\mathcal{G}^{L}_{2}$.}
\label{fig:q-digraph3}
\end{subfigure}%
\caption{A DFC and its respective lower $q$-digraphs, for $q=0,1,2$. The numbers inside the nodes represent the dimensions of the respective directed simplices.}
\label{fig:q-levels-digraphs}
\end{figure}

Note that, by Proposition \ref{prop:maximal-adjacency}, we can replace the maximal $q$-adjacency in the expression (\ref{eq:q-adjacency-matrix}) with the strictly lower $q$-adjacency. From now on, we will adopt the generic notations $\mathcal{G}_{q} = (\mathcal{V}_{q},\mathcal{E}_{q})$ and $\mathcal{H}_{q} = \mathcal{H}_{q}(\mathcal{G}_{q})$ designating both the lower and maximal variants.

\begin{remark}
The strongly and weakly connected components of a maximal/lower $q$-digraph $\mathcal{G}_{q}$ are equivalent to the maximal/lower strongly and weakly $q$-connected components as defined in Definition \ref{def:strongly-q-connected-comp} and Definition \ref{def:weakly-q-connected-comp}, respectively, where the set $\mathrm{dFl}(G)$ is replaced by $\mathrm{dFl}^{*}(G)$. Moreover, the largest weakly $q$-connected component of $\mathcal{G}_{q}$ is called its \textit{giant $q$-component}.
\end{remark}

In the case where we have a weighted DFC obtained from a weighted digraph, by definition, the corresponding maximal/lower $q$-digraphs are node-weighted digraphs since their nodes represent directed simplices. Accordingly, to obtain arc-weighted digraphs, we need a node-to-arc weight function. In the literature, there is a myriad of methods to transform a node-weighted digraph into an arc-weighted digraph \cite{DelaCruzCabrera}. However, since the relations in a digraph can be non-symmetric, we would like a non-symmetric transformation function; thus, for a given $(\mathrm{dFl}(G), \widetilde{\omega})$ and for a given $q$-arc $(\sigma_{i}, \sigma_{j})$, we consider the following node-to-arc weight function:
\begin{equation}\label{eq:node-to-edge1}
f(\widetilde{\omega}(\sigma_{i}), \widetilde{\omega}(\sigma_{j})) = \widetilde{\omega}(\sigma_{i}).
\end{equation}

This leads us to extend our definition of maximal/lower $q$-digraphs to the weighted case as follows.

\begin{definition}\label{def:weig-q-digraph}
Given a weighted digraph $G^{\omega}$, the \textit{weighted maximal} (respec. \textit{lower}) \textit{$q$-digraph} of $(\mathrm{dFl}(G^{\omega}), $ $\widetilde{\omega})$, denoted by $\mathcal{G}_{q}^{\widetilde{\omega}}$, is the digraph whose vertices are the simplices of $\mathrm{dFl}^{\ast}_{q}(G^{\omega})$ and for each pair $\sigma, \tau \in \mathrm{dFl}^{\ast}_{q}(G^{\omega})$ there is a weighted arc from $\sigma$ to $\tau$ if $\sigma \sim^{+}_{A_{q^{*}}} \tau$ (respec. $\sigma \sim^{+}_{L_{q}} \tau$), with $0 \le q \le \dim \mathrm{dFl}(G^{\omega})$, such that the weight of the arc is given by a node-to-edge weight function.
\end{definition}

In addition, we may use the notation $\mathcal{G}_{q}^{\widetilde{\omega}} = (\mathcal{V}_{q}, \mathcal{E}_{q}, \widetilde{\omega})$, where $\mathcal{V}_{q} = \mathrm{dFl}_{q}^{\ast}(G^{\omega})$, $\mathcal{E}_{q}$ is the set of all $q$-arcs, and $\widetilde{\omega}$ is the product-weight function.

\begin{definition}
Let $\mathcal{G}_{q}^{\widetilde{\omega}}$ be the weighted maximal (respec. lower) $q$-digraph of $(\mathrm{dFl}(G^{\omega}), \widetilde{\omega})$. Let $f$ be some non-symmetric node-to-arc weight function. The \textit{weighted maximal} (respec. \textit{lower}) \textit{$q$-adjacency matrix} of $\mathcal{G}_{q}^{\widetilde{\omega}}$, denoted by $\mathcal{H}^{\widetilde{\omega}}_{q} = \mathcal{H}^{\widetilde{\omega}}_{q}(\mathcal{G}^{\widetilde{\omega}}_{q})$, is a real square matrix whose entries are given by
\begin{equation}\label{eq:weig-q-adjacency-matrix}
\big(\mathcal{H}^{\widetilde{\omega}}_{q}\big)_{ij} = \begin{cases}
	f(\widetilde{\omega}(\sigma_{i}), \widetilde{\omega}(\sigma_{j})), \mbox{ if } \sigma_{i} \sim^{+}_{A_{q^{*}}} \sigma_{j} \mbox{ (respec. } \sigma_{i} \not\sim^{+}_{L_{q}} \sigma_{j}),\\
	0, \mbox{ if } i = j \mbox{ or } \sigma_{i} \not\sim^{+}_{A_{q^{*}}} \sigma_{j} \mbox{ (respec. } \sigma_{i} \not\sim^{+}_{L_{q}} \sigma_{j}).
\end{cases}
\end{equation}
\end{definition}

Furthermore, the directed simplicial $q$-distance between two vertices of a maximal/lower $q$-digraph can be written in terms of the entries of its maximal/lower $q$-adjacency matrix, $\mathcal{H}_{q} = (h_{\sigma \tau})$, as follows. Let $s^{q}_{\sigma \rightarrow \tau}$ be the shortest directed simplicial $q$-walk from a vertex $\sigma$ to a vertex $\tau$. The directed simplicial $q$-distance between $\sigma$ and $\tau$ is 
\begin{equation}\label{eq:simp-weig-distance-entries}
\vec{d}_{q}(\sigma, \tau) = \sum_{\sigma', \tau' \in s^{q}_{\sigma \rightarrow \tau}} h_{\sigma' \tau'}.
\end{equation}

Similarly, we can do the same for the weighted case. Let $f$ be some non-symmetric node-to-edge weight function, let $F$ be a weight-to-distance function, and let $s^{q}_{\sigma \rightarrow \tau}(F)$ be the shortest directed simplicial $q$-walk from $\sigma$ to $\tau$ with respect to $F$. For a weighted maximal/lower $q$-digraph, we can define the \textit{weighted directed simplicial $q$-distance} in terms of the entries of its weighted maximal/lower $q$-adjacency matrix as 
\begin{equation}\label{eq:simp-weig-distance-entries}
\vec{d}_{q}^{\omega}(\sigma, \tau) = \sum_{\sigma', \tau' \in s^{q}_{\sigma \rightarrow \tau}(F)} F\big(f(\widetilde{\omega}(\sigma'), \widetilde{\omega}(\tau')\big).
\end{equation}

%-----------
\subsubsection{Stars, Hubs, and Links}

In this part, we extend the definitions of stars, hubs, and links introduced in Subsection \ref{sec:q-analysis} to directed simplices.

\begin{definition}
Given a directed simplex $\sigma^{(n)} \in \mathrm{dFl}(G)$, for $0 \le q \le n$ and for $\bullet \in \{ -,+,\pm \}$, we define the following $q$-stars associated with $\sigma^{(n)}$:

\begin{enumerate}
\item \textit{Lower $(\bullet)$-$q$-star}:  
$\mathrm{st}_{L_{q}} ^{\bullet}(\sigma^{(n)}) = \{ \tau^{(m)} \in \mathrm{dFl}(G) : \sigma^{(n)}  \sim^{\bullet}_{L_{q}}  \tau^{(m)}  \}.$

\item \textit{Strictly lower $(\bullet)$-$q$-star}:
$\mathrm{st}_{L_{q^{*}}} ^{\bullet}(\sigma^{(n)}) = \{ \tau^{(m)} \in \mathrm{dFl}(G) : \sigma^{(n)}  \sim^{\bullet}_{L_{q^{*}}} \tau^{(m)}  \}.$

\item \textit{Upper $(\bullet)$-$q$-star}: 
$\mathrm{st}_{U_{q}} ^{\bullet}(\sigma^{(n)}) = \{ \tau^{(m)} \in \mathrm{dFl}(G) : \sigma^{(n)}  \sim^{\bullet}_{U_{q}}  \tau^{(m)}  \}.$

\item \textit{Strictly upper $(\bullet)$-$q$-star}:
$\mathrm{st}_{U_{q^{*}}} ^{\bullet}(\sigma^{(n)}) = \{ \tau^{(m)} \in \mathrm{dFl}(G) : \sigma^{(n)}  \sim^{\bullet}_{U_{q^{*}}} \tau^{(m)}  \}.$

\item \textit{$(\bullet)$-$q$-star}:
$\mathrm{st}_{A_{q}} ^{\bullet}(\sigma^{(n)}) = \{ \tau^{(m)} \in \mathrm{dFl}(G) : \sigma^{(n)}  \sim^{\bullet}_{A_{q}}  \tau^{(m)}  \}.$

\item \textit{Maximal $(\bullet)$-$q$-star}:
$\mathrm{st}_{A_{q^{*}}} ^{\bullet}(\sigma^{(n)}) = \{ \tau^{(m)} \in \mathrm{dFl}(G) : \sigma^{(n)}  \sim^{\bullet}_{A_{q^{*}}}  \tau^{(m)}  \}.$
\end{enumerate}
\end{definition}

Notice that $\mathrm{st}^{\pm}(\sigma) = \mathrm{st}^{+}(\sigma) \cap \mathrm{st}^{-}(\sigma)$, for any of the $q$-stars defined above.

\begin{definition}\label{def:dir-hub}
Let $\mathcal{F}(G)$ denote a simplicial family of directed simplices obtained from $\mathrm{dFl}(G)$. The \textit{hub} of $\mathcal{F}(G)$ is the set formed by all directed simplices that are common faces of the elements of $\mathcal{F}(G)$ (regardless of the direction of the connection between them), i.e. $\mathrm{hub}(\mathcal{F}(G)) = \bigcap_{\sigma \in \mathcal{F}(G)} \sigma.$
\end{definition}

We can generalize the notion of in- and out-neighborhood of a node in a digraph to directed simplices by defining \textit{in-} and \textit{out-link} as follows.

\begin{definition}\label{def:dir-link}
The \textit{in-link} and \textit{out-link} of $\sigma^{(n)} \in \mathrm{dFl}(G)$ are defined, respectively, by
\begin{equation}\label{eq:in-link}
\mathrm{lk}^{-}(\sigma^{(n)}) = \{ \tau^{(m)} \in \mathrm{dFl}(G) |\sigma^{(n)} \cap \tau^{(m)} = \emptyset, \sigma^{(n)}\sim_{U_{p}^{-}}  \tau^{(m)}  \},
\end{equation}
\begin{equation}\label{eq:out-link}
\mathrm{lk}^{+}(\sigma^{(n)}) = \{ \tau^{(m)} \in \mathrm{dFl}(G) |\sigma^{(n)} \cap \tau^{(m)} = \emptyset, \sigma^{(n)} \sim_{U_{p}^{+}} \tau^{(m)}  \},
\end{equation}

\noindent where $p=n+m+1$.
\end{definition}

If $\sigma$ is a simplex in the underlying flag complex of $ \mathrm{dFl}(G)$, then $\mathrm{lk}(\sigma) = \mathrm{lk}^{-}(\sigma) \cup \mathrm{lk}^{+}(\sigma) - (\mathrm{lk}^{-}(\sigma) \cap \mathrm{lk}^{+}(\sigma))$.

%------ Degrees --------
\subsubsection{Lower, Upper, and General Degrees}

Building on the previous definitions of $q$-stars, we extend the definitions of lower, upper, and general degrees to directed simplices as follows.

\begin{definition}
Given a directed simplex $\sigma^{(n)} \in \mathrm{dFl}(G)$, for $0 \le q \le n$ and for $\bullet \in \{ -,+,\pm \}$, we define the following $q$-degrees associated with $\sigma^{(n)}$:

\begin{enumerate}
\item \textit{Lower $(\bullet)$-$q$-degree}:
$\deg^{\bullet}_{L_{q}}(\sigma^{(n)}) = |\mathrm{st}_{L_{q}}^{\bullet}(\sigma^{(n)})| = \# \{ \tau^{(m)} \in \mathrm{dFl}(G) : \sigma^{(n)} \sim_{L_{q}}^{\bullet} \tau^{(m)} \}.$

\item \textit{Strictly lower $(\bullet)$-$q$-degree}:
$\deg^{\bullet}_{L_{q^{*}}}(\sigma^{(n)}) = |\mathrm{st}_{L_{q^{*}}}^{\bullet}(\sigma^{(n)})| = \# \{ \tau^{(m)} \in \mathrm{dFl}(G) : \sigma^{(n)} \sim_{L_{q^{*}}}^{\bullet} \tau^{(m)} \}.$

\item \textit{Upper $(\bullet)$-$q$-degree}: 
$\deg^{\bullet}_{U_{q}}(\sigma^{(n)}) = |\mathrm{st}_{U_{q}}^{\bullet}(\sigma^{(n)})| = \# \{ \tau^{(m)} \in \mathrm{dFl}(G) : \sigma^{(n)} \sim_{U_{q}}^{\bullet} \tau^{(m)} \}.$

\item \textit{Strictly upper $(\bullet)$-$q$-degree}:
$\deg^{\bullet}_{U_{q^{*}}}(\sigma^{(n)}) = |\mathrm{st}_{U_{q^{*}}}^{\bullet}(\sigma^{(n)})| = \# \{ \tau^{(m)} \in \mathrm{dFl}(G) : \sigma^{(n)} \sim_{U_{q^{*}}}^{\bullet} \tau^{(m)} \}.$

\item \textit{$(\bullet)$-$q$-degree}:
$\deg^{\bullet}_{A_{q}}(\sigma^{(n)}) = |\mathrm{st}_{A_{q}}^{\bullet}(\sigma^{(n)})| = \# \{ \tau^{(m)} \in \mathrm{dFl}(G) : \sigma^{(n)} \sim_{A_{q}}^{\bullet} \tau^{(m)} \}.$

\item \textit{Maximal $(\bullet)$-$q$-degree}:
$\deg^{\bullet}_{A_{q^{*}}}(\sigma^{(n)}) = |\mathrm{st}_{A_{q^{*}}}^{\bullet}(\sigma^{(n)})| = \# \{ \tau^{(m)} \in \mathrm{dFl}(G) : \sigma^{(n)} \sim_{A_{q^{*}}}^{\bullet} \tau^{(m)} \}.$
\end{enumerate}
\end{definition}

Notice that if $\sigma$ is a simplex in the underlying flag complex of $\mathrm{dFl}(G)$, then its $q$-degree, for any of the lower, upper, and general adjacencies, is equal to $\deg(\sigma) =  \deg^{-}(\sigma) +  \deg^{+}(\sigma) -  \deg^{\pm}(\sigma).$

%\begin{example}
%Considering the DFC depicted in Figure \ref{fig:dfc-example-stars-degree}, we have the following examples: $\deg^{-}_{A_{1^{*}}}(\sigma_{1}) = 1$ and $\deg^{+}_{A_{1^{*}}}(\sigma_{1}) = 0$;  $\deg^{-}_{A_{1^{*}}}(\sigma_{6}) = 2$ and $\deg^{+}_{A_{1^{*}}}(\sigma_{6}) = 2$.
%\end{example}

It's important to note that if we are considering solely the elements of $\mathrm{dFl}_{q}^{\ast}(G)$, i.e., the maximal directed simplices, then the maximal/lower $(\bullet)$-$q$-degrees, $\bullet \in \{ -, + \}$, can be written in terms of the entries of the $q$-adjacency matrix of the maximal/lower $q$-digraph, $\mathcal{H}_{q} = (h_{\sigma\tau})$: $\deg^{+}_{q}(\sigma) = \sum_{\tau \in \mathrm{st}^{+}_{q}(\sigma)} h_{\sigma \tau}$ and $\deg^{-}_{q}(\sigma) = \sum_{\tau \in \mathrm{st}^{-}_{q}(\sigma)} h_{\tau \sigma}$, where the generic notations $\deg^{\bullet}_{q}$ and $\mathrm{st}^{\bullet}_{q}$ designate either the maximal or the lower variants. On the other hand, in the case where we have a weighted DFC $(\mathrm{dFl}(G^{\omega}), \widetilde{\omega})$, for some non-symmetric node-to-arc weight function $f$, the \textit{weighted maximal/lower $(\bullet)$-$q$-degrees} can be written as (considering $\deg^{\omega, \bullet}_{q}$ and $\mathrm{st}^{\bullet}_{q}$ as generic notations for the maximal and lower variants):

\noindent\begin{minipage}{.5\linewidth}
\begin{equation}\label{eq:simp-weig-in-degree}
\deg^{\omega, -}_{q}(\sigma) = \sum_{\tau \in \mathrm{st}^{-}_{q}(\sigma)} f(\widetilde{\omega}(\sigma), \widetilde{\omega}(\tau)),
\end{equation}
\end{minipage}%
\begin{minipage}{.5\linewidth}
\begin{equation}\label{eq:simp-weig-out-degree}
\deg^{\omega, +}_{q}(\sigma) = \sum_{\tau \in \mathrm{st}^{+}_{q}(\sigma)} f(\widetilde{\omega}(\tau), \widetilde{\omega}(\sigma)).
\end{equation}
\end{minipage}

%------------------------------------------------------------
%------------------------------------------------------------
\section{Quantitative Approaches to Directed Flag Complexes}
\label{sec:quantitative}

According to Dehmer et al. \cite{Dehmer-2017}, the ``quantitative graph theory (QGT) deals with the quantification of structural aspects of graphs, instead of characterizing graphs only descriptively.'' QGT forms a novel subfield within graph theory \cite{Dehmer-2017, Dehmer-2014}, filling the gap left by predominantly descriptive classical methods. Moreover, QGT can be broadly divided into two categories \cite{Dehmer-2014}: \textit{graph characterization} and \textit{comparative graph analysis}. Graph characterization is concerned with describing some network property through a local or global numerical graph invariant \cite{Arizmendi, Bavelas, Freeman1978, Wang2006, Watts}. Comparative graph analysis focuses on methods for comparing the structural similarity or distance between two or more graphs. There are two classes of methods of graph similarity comparison \cite{Mheich}: \textit{statistical comparison methods} and \textit{distance-based comparison algorithms}.

Drawing an analogy from QGT, a corresponding field can be proposed for simplicial complexes, termed \textit{quantitative simplicial theory}, which also focuses on quantifying structural aspects rather than solely describing them. This theory can similarly be divided into:

\begin{itemize}
\item \textbf{Simplicial Characterization Measures:} simplicial distances and simplicial eccentricities \cite{Johnson, Serrano2020}; simplicial centralities \cite{Estrada, Serrano2020}; simplicial clustering coefficients \cite{Maletic, Serrano2020}; simplicial entropies \cite{Baccini, Dantchev, Maletic2012}; discrete curvatures \cite{Yamada2023}; simplicial energies \cite{Knill}.

\item \textbf{Simplicial Similarity Comparison Methods:} distances between persistent diagrams \cite{Edelsbrunner}; distances between vectorized persistence summaries \cite{Fasy}; simplicial kernels \cite{Martino, Zhang2020}; distances between structure vectors.
\end{itemize}

Many of these simplicial quantitative methods can be further extended to DFCs, as discussed in subsequent sections.

%-----------------------------------------------------------
%-----------------------------------------------------------
\subsection{Simplicial Characterization Measures}

In this section, we introduce novel simplicial analogs of digraph measures for maximal/lower $q$-digraphs (henceforth, $q$-digraphs) associated with DFCs. We start by presenting distance-based simplicial measures, then simplicial centralities and simplicial segregation measures, and conclude with spectrum-related simplicial measures. It is worth noting that all these measures are based on the directed high-order connectivity of the complex in a specific manner.

\smallskip

\noindent \textbf{Conventions and notations.} Throughout the next subsections, $\mathrm{dFl}(G)$ will denote the DFC of a given simple digraph $G$ \textit{without double edges}, and $\mathcal{G}_{q} = (\mathcal{V}_{q}, \mathcal{E}_{q})$ will denote its (maximal or lower) $q$-digraph with (maximal or lower) $q$-adjacency matrix (henceforth simply referred to as $q$-adjacency matrix) $\mathcal{H}_{q} = (h_{\sigma \tau})$, for $0 \le q \le \dim \mathrm{dFl}(G)$. For each of the measures defined in the following subsections, corresponding maximal and lower variants are obtained by replacing the generic notations as described in Table~\ref{tab:lower-maximal-conventions} with the corresponding maximal or lower notations (e.g., $\mathcal{H}_{q}$ by $\mathcal{H}_{q}^{A}$ or $\mathcal{H}_{q}^{L}$, $\vec{d}_{q}$ by $\vec{d}^{A}_{q}$ or $\vec{d}^{L}_{q}$, and $\deg^{\bullet}_{q}$ by $\deg^{\bullet}_{A_{q^{*}}}$ or $\deg^{\bullet}_{L_{q}}$).

{\renewcommand{\arraystretch}{1.35}
\begin{table}[h!]
\centering
\small
\caption{Notation conventions for maximal and lower variants of the simplicial measures.  All reference tables of simplicial measures in this document use the \emph{generic} column.}
\begin{tabular}{lccc}
\hline
\textbf{Object} & \textbf{Generic notation} & \textbf{Maximal variant} & \textbf{Lower variant}\\
\hline

$q$-Digraph & $\mathcal{G}_{q}$        & $\mathcal{G}_{q}^{A}$        & $\mathcal{G}_{q}^{L}$ \\

$q$-Arc set &  $\mathcal{E}_{q}$        & $\mathcal{E}_{q}^{A}$        & $\mathcal{E}_{q}^{L}$ \\

$q$-Arc & $(\sigma, \tau)$  & $(\sigma, \tau)_{A}$   & $(\sigma, \tau)_{L}$  \\

$q$-Adjacency matrix & $\mathcal{H}_{q}$        & $\mathcal{H}_{q}^{A}$        & $\mathcal{H}_{q}^{L}$ \\

Directed $q$-distance & $\vec{d}_{q}$            & $\vec{d}_{q}^{A}$            & $\vec{d}_{q}^{L}$ \\

($\bullet$)-$q$-Star & $\mathrm{st}^{\bullet}_{q}$ & $\mathrm{st}^{\bullet}_{A_{q^{*}}}$ & $\mathrm{st}^{\bullet}_{L_{q}}$ \\

($\bullet$)-$q$-Degree & $\deg^{\bullet}_{q}$ & $\deg^{\bullet}_{A_{q^{*}}}$ & $\deg^{\bullet}_{L_{q}}$ \\

\hline
\end{tabular}
\label{tab:lower-maximal-conventions}
\end{table}
}

%---------------------------------------------------------
\subsubsection{Distance-Based Simplicial Measures}
\label{sec:simp-distance-measures}

In this part, we extend several \textit{distance-based measures} defined primarily for digraphs to $q$-digraphs. These new simplicial measures can be seen as measures of higher-order global integration of a directed network, that is, how the network is integrated at various levels of organization.

\begin{itemize}
\item The \textbf{average shortest directed simplicial q-walk length} of $\mathcal{G}_{q}$ is defined as the simplicial analog of the directed version of the average shortest path length \cite{Watts}, i.e.
\begin{equation}\label{eq:simp-average-sw}
\vec{\bar{L}}_{q}(\mathcal{G}_{q}) = \frac{1}{|\mathcal{V}_{q}|} \sum_{\sigma \in \mathcal{V}_{q}}\frac{\sum_{\tau \in \mathcal{V}_{q}, \tau \neq \sigma} \vec{d}_{q}(\sigma, \tau)}{|\mathcal{V}_{q}|-1} = \sum_{\substack{\sigma, \tau\in \mathcal{V}_{q} \\ \sigma\neq \tau}} \frac{\vec{d}_{q}(\sigma, \tau)}{|\mathcal{V}_{q}|(|\mathcal{V}_{q}|-1)}.
\end{equation}
For computational purposes, specifically in this case, we may consider $\vec{d}_{q}(\sigma, \tau) = 0$ instead of $\vec{d}_{q}(\sigma, \tau) = \infty$ when $(\sigma, \tau) \not\in \mathcal{E}_{q}$; otherwise, we would have to consider the giant $q$-component and thus replace $|\mathcal{V}_{q}|$ with its order.

\item The \textbf{directed simplicial $q$-eccentricity} of $\sigma \in \mathcal{V}_{q}$, the \textbf{directed simplicial $q$-diameter}, and \textbf{$q$-radius} of $\mathcal{G}_{q}$ (strongly $q$-connected $q$-digraph), are defined respectively as 

\noindent\begin{minipage}{.333\linewidth}
\begin{equation}\label{eq:simp-ecc}
	\vec{\mathrm{ecc}}_{q}(\sigma) = \max_{\tau \in \mathcal{V}_{q}} \vec{d}_{q}(\sigma, \tau),
\end{equation}
\end{minipage}%
\begin{minipage}{.333\linewidth}
\begin{equation}\label{eq:simp-diam}
	\mathrm{diam}(\mathcal{G}_{q}) = \max_{\sigma \in \mathcal{V}_{q}} \vec{\mathrm{ecc}}_{q}(\sigma),
\end{equation}
\end{minipage}
\begin{minipage}{.333\linewidth}
\begin{equation}\label{eq:simp-rad}
	\mathrm{rad}(\mathcal{G}_{q}) = \min_{\sigma \in \mathcal{V}_{q}} \vec{\mathrm{ecc}}_{q}(\sigma).
\end{equation}
\end{minipage}

In the literature, there are different ways to define the eccentricity of a simplex \cite{Atkin1977, Johnson}. Here, however, we extended the definition of simplicial eccentricity, as proposed in \cite{Serrano2020}, to directed simplices.

\item The \textbf{directed simplicial global q-efficiency} of $\mathcal{G}_{q}$ is defined as the simplicial analog of the directed version of the global efficiency as proposed by Latora and Marchiori \cite{Latora}, i.e.
\begin{equation}\label{eq:simp-global-efficiency}
\vec{E}_{glob}^{q}(\mathcal{G}_{q}) = \frac{1}{|\mathcal{V}_{q}|} \sum_{\sigma \in \mathcal{V}_{q}}\frac{\sum_{\tau \in \mathcal{V}_{q}, \tau \neq \sigma} \vec{d}^{-1}_{q}(\sigma, \tau) }{|\mathcal{V}_{q}|-1}.
\end{equation}

\item For two directed simplices $\sigma, \tau \in \mathcal{V}_{q}$, let $(\mathcal{H}_{q}^{k})_{\sigma\tau}$ be the number of directed simplicial $q$-walks of length $k$ from $\sigma$ to $\tau$. The \textbf{simplicial $q$-communicability} between $\sigma$ and $\tau$ is defined as the simplicial analog of the communicability in a digraph \cite{Estrada-2009a, Estrada-2015}, and therefore it can be written in terms of the powers of the $q$-adjacency matrix as
\begin{equation}\label{eq:simp-communicability}
CM_{q}(\sigma, \tau) = \sum^{\infty}_{k=0} \frac{(\mathcal{H}_{q}^{k})_{\sigma\tau}}{k!} = (\exp(\mathcal{H}_{q}))_{\sigma\tau}.
\end{equation}

\item Let $\mathcal{H}_{q}'$ be the $q$-adjacency matrix of the underlying $q$-graph of $\mathcal{G}_{q}$. The \textbf{simplicial $q$-returnability} and \textbf{relative simplicial $q$-returnability} of $\mathcal{G}_{q}$ are defined as the simplicial analogs of the returnability and relative returnability of a digraph \cite{Estrada-2009b, Estrada-2015}, and are given, respectively, by

\noindent\begin{minipage}{.5\linewidth}
\begin{equation}\label{eq:simp-return}
	K_{r,q}(\mathcal{G}_{q}) = \sum^{\infty}_{k=2} \frac{\mathrm{Tr}(\mathcal{H}_{q}^{k})}{k!} = \mathrm{Tr}(\exp(\mathcal{H}_{q})) - |\mathcal{V}_{q}|,
\end{equation}
\end{minipage}%
\begin{minipage}{.5\linewidth}
\begin{equation}\label{eq:simp-rel-return}
	K_{r, q}'(\mathcal{G}_{q}) = \frac{\mathrm{Tr}(\exp(\mathcal{H}_{q})) - |\mathcal{V}_{q}|}{\mathrm{Tr}(\exp(\mathcal{H}_{q}')) - |\mathcal{V}_{q}|}.
\end{equation}
\end{minipage}

\end{itemize}

For a weighted $q$-digraph $\mathcal{G}^{\widetilde{\omega}}_{q}$, the weighted versions of the formulas (\ref{eq:simp-average-sw}) and (\ref{eq:simp-global-efficiency}) are obtained by replacing $\vec{d}_{q}$ with $\vec{d}_{q}^{\omega}$.

%--------------------------------------------------------------
\subsubsection{Simplicial Centrality Measures}
\label{sec:simp-centrality-measures}

In this part, we generalize node centrality measures primarily defined for digraphs, such as in-degree, out-degree, directed closeness, and directed betweenness centrality, to $q$-digraphs. While classical centrality measures quantify a node's capacity to transmit, receive, or mediate information within a network, their simplicial counterparts extend this notion to higher-order structures. Specifically, they evaluate the ``importance,'' ``influence,'' or ``centrality'' of directed cliques across multiple organizational scales of the network.

\begin{itemize}
\item The \textbf{simplicial in/out-$q$-degree centralities} of $\sigma \in \mathcal{V}_{q}$ are defined as the simplicial analogs of the in/out-degree centralities of a node, i.e.,

\noindent\begin{minipage}{.5\linewidth}
\begin{equation}\label{eq:simp-in-degree}
	C_{\deg_{q}}^{-}(\sigma) = \frac{\deg^{-}_{q}(\sigma) }{|\mathcal{V}_{q}| - 1},
\end{equation}
\end{minipage}%
\begin{minipage}{.5\linewidth}
\begin{equation}\label{eq:simp-out-degree}
	C_{\deg_{q}}^{+}(\sigma) = \frac{\deg^{+}_{q}(\sigma) }{|\mathcal{V}_{q}| - 1}.
\end{equation}
\end{minipage}

Also, if $\sigma$ is a simplex in the corresponding underlying $q$-graph, then its $q$-degree centrality can be written as $C_{\deg_{q}}(\sigma) = C_{\deg_{q}}^{-}(\sigma) + C_{\deg_{q}}^{+}(\sigma) - \deg^{\pm}_{q}(\sigma)/(|\mathcal{V}_{q}| - 1).$ To obtain the weighted versions of (\ref{eq:simp-in-degree}) and  (\ref{eq:simp-out-degree}), we simply replace $\deg^{-}_{q}$ with $\deg^{\omega, -}_{q}$ (\ref{eq:simp-weig-in-degree}), and $\deg^{+}_{q}$ with $\deg^{\omega, +}_{q}$ (\ref{eq:simp-weig-out-degree}), respectively.

\item Let $N_{q}$ be the size of the giant $q$-component of $\mathcal{G}_{q}$. The \textbf{directed simplicial $q$-closeness centrality} of $\sigma \in \mathcal{V}_{q}$ and its normalized version are defined as the simplicial analogs of the directed version of the closeness centrality and its normalized version \cite{Bavelas}, and are given, respectively, by

\noindent\begin{minipage}{.5\linewidth}
\begin{equation}\label{eq:simp-closeness}
	\vec{Cl}_{q}(\sigma) = \frac{1}{\sum_{\substack{\tau \in \mathcal{V}_{q} \\ \tau \neq \sigma}} \vec{d}_{q}(\sigma, \tau)},
\end{equation}
\end{minipage}%
\begin{minipage}{.5\linewidth}
\begin{equation}\label{eq:simp-norma-closeness}
	\vec{Cl}_{q}(\sigma) = \frac{N_{q} - 1}{\sum_{\substack{\tau \in \mathcal{V}_{q} \\ \tau \neq \sigma}} \vec{d}_{q}(\sigma, \tau)}.
\end{equation}
\end{minipage}

Notice that (\ref{eq:simp-closeness}) and (\ref{eq:simp-norma-closeness}) are defined for weakly $q$-connected $q$-digraphs.

\item The \textbf{directed simplicial $q$-harmonic centrality} of $\sigma \in \mathcal{V}_{q}$ is defined as the simplicial analog of the directed version of the harmonic centrality \cite{Boldi}, i.e.,
\begin{equation}\label{eq:simp-harmonic}
\vec{HC}_{q}(\sigma) = \sum_{\substack{\tau \in \mathcal{V}_{q} \\ \tau \neq \sigma}} \frac{1}{\vec{d}_{q}(\sigma, \tau)}.
\end{equation}

Unlike (\ref{eq:simp-closeness}) and (\ref{eq:simp-norma-closeness}), (\ref{eq:simp-harmonic}) can be computed for disconnected $q$-digraphs.

\item Considering vertices $\sigma, \tau, \tau' \in \mathcal{V}_{q}$, let $\vec{l}^{q}_{\tau'\tau}(\sigma)$ be the number of shortest directed simplicial $q$-walks from $\tau'$ to $\tau$ passing through $\sigma$, and let $\vec{l}^{q}_{\tau'\tau}$ be the total number of shortest directed simplicial $q$-walks from $\tau'$ to $\tau$. Let $N_{q}$ be the size of the giant $q$-component of $\mathcal{G}_{q}$. The \textbf{directed simplicial $q$-betweenness centrality} of $\sigma \in \mathcal{V}_{q}$ and its normalized version are defined as the simplicial analogs of the directed version of the betweenness centrality and its normalized version \cite{Freeman1978}, and are given, respectively, by

\noindent\begin{minipage}{.5\linewidth}
\begin{equation}\label{eq:simp-betweenness}
	\vec{B}_{q}(\sigma) = \sum_{\substack{\tau, \tau' \in \mathcal{V}_{q} \\ \tau' \neq \tau \neq \sigma}} \frac{\vec{l}^{q}_{\tau'\tau}(\sigma)}{\vec{l}^{q}_{\tau'\tau}},
\end{equation}
\end{minipage}%
\begin{minipage}{.5\linewidth}
\begin{equation}\label{eq:simp-norma-betweenness}
	\vec{B}_{q}(\sigma) = \frac{1}{(N_{q} - 1)(N_{q}-2)} \sum_{\substack{\tau, \tau' \in \mathcal{V}_{q} \\ \tau' \neq \tau \neq \sigma}} \frac{\vec{l}^{q}_{\tau'\tau}(\sigma)}{\vec{l}^{q}_{\tau'\tau}}.
\end{equation}
\end{minipage}

Notice that (\ref{eq:simp-betweenness}) and (\ref{eq:simp-norma-betweenness}) are defined for weakly $q$-connected $q$-digraphs

\item Let $r_{\mathcal{G}_{q}}(\sigma)$ be the number of vertices in $\mathcal{G}_{q}$ which are reachable from $\sigma \in \mathcal{V}_{q}$, and define $C_{R, q}^{max} = \max_{\sigma \in \mathcal{V}_{q}} C_{R, q}(\sigma)$. The \textbf{simplicial local $q$-reaching centrality} of $\sigma \in \mathcal{V}_{q}$ and the \textbf{simplicial global $q$-reaching centrality} of $\mathcal{G}_{q}$ are defined as the simplicial analogs of the local and global reaching centralities \cite{Mones}, and are defined respectively by

\noindent\begin{minipage}{.5\linewidth}
\begin{equation}\label{eq:simp-local-reaching-centrality}
	C_{R, q}(\sigma) = \frac{r_{\mathcal{G}_{q}}(\sigma)}{|\mathcal{V}_{q}|-1},
\end{equation}
\end{minipage}%
\begin{minipage}{.5\linewidth}
\begin{equation}\label{eq:simp-global-reaching-centrality}
	GRC_{q}(\mathcal{G}_{q}) = \frac{\sum_{\sigma \in \mathcal{V}_{q}} [C_{R, q}^{max} - C_{R, q}(\sigma)]}{|\mathcal{V}_{q}|-1}.
\end{equation}
\end{minipage}

\end{itemize}

In addition, for the weighted case, we simply replace $\vec{d}_{q}$ with $\vec{d}_{q}^{\omega}$ in the formulas (\ref{eq:simp-closeness}), (\ref{eq:simp-norma-closeness}), and (\ref{eq:simp-harmonic}); for (\ref{eq:simp-betweenness}) and (\ref{eq:simp-norma-betweenness}), we consider the weighted version $\vec{l}^{\omega, q}_{\tau'\tau}(\sigma)$, where the shortest directed simplicial $q$-walks are computed concerning a weight-to-distance function.

%-------------------------------------------------------------
\subsubsection{Simplicial Segregation Measures}
\label{sec:simp-segregation-measures}

In this part, we extend several measures of segregation, defined initially for digraphs, to $q$-digraphs. These new simplicial measures can be interpreted as attempts to quantify the tendency of directed cliques to segregate into higher-order clusters or communities in a directed network.

\begin{itemize}
\item Considering a vertex $\sigma \in \mathcal{V}_{q}$, let $\deg_{q}^{tot}(\sigma) = \deg_{q}^{-}(\sigma) + \deg_{q}^{+}(\sigma)$ and let $\vec{T}_{q}(\sigma)$ be the number of directed triangles at the level $q$ containing $\sigma$. The \textbf{average directed simplicial $q$-clustering coefficient} of $\mathcal{G}_{q}$ is defined as the simplicial analogue of the average directed $q$-clustering coefficient \cite{Fagiolo}, i.e.
\begin{equation}\label{eq:simp-average-cc}
\vec{\bar{C}}_{q}(\mathcal{G}_{q}) = \frac{1}{|\mathcal{V}_{q}|} \sum_{\sigma \in \mathcal{V}_{q}} \frac{\vec{T}_{q}(\sigma)}{\deg_{q}^{tot}(\sigma)(\deg_{q}^{tot}(\sigma) - 1) - 2\deg_{q}^{\pm}(\sigma)}.
\end{equation}

We notice that $\vec{T}_{q}$ can be written in terms of the $q$-adjacency matrix entries as $\vec{T}(\sigma)  = \frac{1}{2} \sum_{\tau', \tau \in \mathcal{V}_{q}} (h_{\sigma \tau'} + h_{\tau' \sigma})(h_{\sigma \tau} + h_{\tau \sigma})(h_{\tau' \tau} + h_{\tau \tau'}).$

\item Let $f_{\sigma\tau}$ be the dimension of the face shared between two vertices $\sigma^{(n)}$ and $\tau^{(m)}$ in $\mathcal{G}_{q}$. The \textbf{simplicial $(\bullet)$-$q$-clustering coefficient} of $\sigma^{(n)}$, with $\bullet \in \{ -,+,\pm \}$, is defined as the directed variants of the simplicial clustering coefficient introduced by Maletic et al. \cite{Maletic}, i.e.
\begin{equation}\label{eq:in-out-clustering-coef}
\vec{C}_{q}^{\bullet}(\sigma^{(n)}) = \sum_{\tau^{(m)} \in \mathrm{st}^{\bullet}_{q}(\sigma)} \frac{2^{1+f_{\sigma\tau}} - 1}{2^{n} + 2^{m} - 1}.
\end{equation}

\item Let $F^{-}_{>k}(q)$ (respec. $F^{+}_{>k}(q)$) be the number of vertices $\sigma \in \mathcal{V}_{q}$ having $\deg_{q}^{-}(\sigma) > k$ (respec. $\deg_{q}^{+}(\sigma) > k$), and $E^{-}_{>k}(q)$ (respec. $E^{+}_{>k}(q)$) be the number of $q$-arcs connecting those $F^{-}_{>k}(q)$ (respc. $F^{+}_{>k}(q)$) vertices. The \textbf{simplicial in/out-$q$-degree rich-club coefficients} of $\mathcal{G}_{q}$, for an integer $k \ge 0$, are defined as the simplicial analogs of the in/out-degree rich-club coefficients \cite{Smilkov}, i.e.

\begin{minipage}{.5\linewidth}
\begin{equation}\label{eq:simp-in-rich-club}
	\phi_{q, k}^{-}(\mathcal{G}_{q}) = \frac{{E}_{>k}^{-}(q)}{F_{>k}^{-}(q)(F_{>k}^{-}(q) - 1)},
\end{equation}
\end{minipage}%
\begin{minipage}{.5\linewidth}
\begin{equation}\label{eq:simp-out-rich-club}
	\phi_{q, k}^{+}(\mathcal{G}_{q}) = \frac{{E}_{>k}^{+}(q)}{F_{>k}^{+}(q)(F_{>k}^{+}(q) - 1)}.
\end{equation}
\end{minipage}

\item Let $\mathcal{G}_{q}(\sigma)$ be the induced subdigraph of $\mathcal{G}_{q}$ formed by elements of $\mathrm{st}^{-}_{q}(\sigma) \cup \mathrm{st}^{+}_{q}(\sigma) - (\mathrm{st}^{-}_{q}(\sigma) \cap \mathrm{st}^{+}_{q}(\sigma))$, excluding $\sigma$. The \textbf{directed simplicial local $q$-efficiency} of $\mathcal{G}_{q}$ is defined as the simplicial analog of the directed version of the local efficiency \cite{Latora}, i.e.
\begin{equation}\label{eq:simp-local-efficiency}
\vec{E}^{q}_{loc}(\mathcal{G}_{q}) = \frac{1}{|\mathcal{V}_{q}|} \sum_{\sigma \in \mathcal{V}_{q}} \vec{E}^{q}_{glob}(\mathcal{G}_{q}(\sigma)).
\end{equation}

For the weighted case, we replace $\vec{E}^{q}_{glob}$ with its weighted version.

\end{itemize}

%----------------------------------------------------------
%----------------------------------------------------------
\subsubsection{Spectrum-Related Simplicial Measures}
\label{sec:simp-spectrum-measures}

In this last part, we extend spectrum-based measures to DFCs. These simplicial measures are related to the spectra of the $q$-adjacency matrices and can be interpreted as quantifying the network's global ``higher-order structures'' through its ``higher-order spectra.''

\begin{itemize}
\item Let $\{ \varsigma_{k}^{q} \}_{k}$ be the square roots of the eigenvalues (singular values) of the matrix $\mathcal{H}_{q}^{T}\mathcal{H}_{q}$. The \textbf{simplicial $q$-energy} of $\mathcal{G}_{q}$ is defined as the simplicial analogue of the graph energy \cite{Arizmendi}, i.e. 
\begin{equation}\label{eq:simp-energy}
\varepsilon_{q}(\mathcal{G}_{q}) = || ( \mathcal{H}_{q}\mathcal{H}_{q}^{T} )^{1/2}||_{*} = \mathrm{Tr}( (\mathcal{H}_{q}\mathcal{H}_{q}^{T})^{1/2}) = \mathrm{Tr}( (\mathcal{H}_{q}^{T}\mathcal{H}_{q})^{1/2}) = \sum_{k} \varsigma_{k}^{q}.
\end{equation}

\item Consider $\mathcal{H}_{q} \in \mathbb{R}^{n_{q} \times n_{q}}$. Let's denote $|1\rangle = (1, \ldots,1)^{T}$, and let $I_{n_{q}}$ be the $n_{q} \times n_{q}$ identity matrix. The \textbf{simplicial in/out-$q$-Katz centralities} of $\sigma \in \mathcal{V}_{q}$ are defined as the simplicial analogs of the directed versions of the Katz centrality, which consider the incoming/outgoing arcs \cite{Estrada-2011}, respectively, i.e.

\begin{minipage}{.5\linewidth}
\begin{equation}\label{eq:simp-dir-katz-centrality-in}
	K^{-}_{q}(\sigma) = \big[ \langle 1| (I_{n_{q}} - \alpha \mathcal{H}_{q})^{-1}  \big]_{\sigma},
\end{equation}
\end{minipage}%
\begin{minipage}{.5\linewidth}
\begin{equation}\label{eq:simp-dir-katz-centrality-out}
	K^{+}_{q}(\sigma) = \big[ (I_{n_{q}} - \alpha \mathcal{H}_{q})^{-1} |1 \rangle  \big]_{\sigma}.
\end{equation}
\end{minipage}

The subscript $\sigma$ in the brackets represents the position corresponding to $\sigma$ in the vector inside the brackets. The attenuation factor must be $\alpha \neq 1/\lambda_{1}^{q}$ (typically chosen to be $\alpha < 1/\lambda_{1}^{q}$), where $\lambda_{1}^{q}$ is the largest eigenvalue of $\mathcal{H}_{q}$.

\item The \textbf{right/left simplicial $q$-eigenvector centralities} of $\sigma \in \mathcal{V}_{q}$ are defined as the simplicial analogues of the right/left eigenvector centralities \cite{Bonacich, Estrada-2011}, i.e. as the entries corresponding to $\sigma$ of the right/left eigenvectors associated with the largest eigenvalue ($\lambda_{1}^{q}$) of $\mathcal{H}_{q}$:

\begin{minipage}{.5\linewidth}
\begin{equation}\label{eq:right-simp-eigenvector-centrality}
	C_{e, r}^{q}(\sigma) = \Bigg( \frac{1}{\lambda_{1}^{q}} \mathcal{H}_{q} v^{q}_{1} \Bigg)_{\sigma},
\end{equation}
\end{minipage}%
\begin{minipage}{.5\linewidth}
\begin{equation}\label{eq:left-simp-eigenvector-centrality}
	C_{e, l}^{q}(\sigma) = \Bigg( \frac{1}{\lambda_{1}^{q}} \mathcal{H}^{T}_{q} v^{q}_{1} \Bigg)_{\sigma}.
\end{equation}
\end{minipage}

The Perron-Frobenius theorem \cite{Beineke-alg} guarantees that the right and left eigenvectors associated with $\lambda_{1}^{q}$ are non-negative, thus (\ref{eq:right-simp-eigenvector-centrality}) and (\ref{eq:left-simp-eigenvector-centrality}) are non-negative. 

\end{itemize}

%-------------------------------
\subsubsection{Worked Examples and Simulation Studies}

To see examples illustrating the applications of the concepts and simplicial measures developed so far to small digraphs, along with simulations using random digraphs, please refer to the supplementary material accompanying this article.

%=======================================================
\subsection{Simplicial Similarity Comparison Methods}
\label{sec:simp-distances-kernels}

In this section, we introduce two similarity comparison methods for DFCs, namely \textit{structure distances} and \textit{simplicial kernels}, based on structural properties such as the numbers of weakly or strongly $q$-connected components and of directed simplices of specific dimensions.

Throughout this part, all DFCs are considered to be associated with simple digraphs \textit{without double edges}.

%------------
\subsubsection{Structure Vectors and Structure Distances}

In what follows, we define structure vectors for DFCs, analogous to those described in \cite{Andjelkovic2015} for simplicial complexes, each capturing a distinct structural property of the complex.

\begin{definition}\label{def:topological-structure-vectors-1}
Given a DFC $\mathcal{X}$, with $\dim \mathcal{X} = N$, we define the \textit{n-th structure vector}, $\mathrm{Str}_{n}(\mathcal{X})$, associated with $\mathcal{X}$ as following:

\begin{enumerate}
\item $\mathrm{Str}_{1}(\mathcal{X}) = (s^{1}_{0}, \ldots, s^{1}_{N})$, where $s^{1}_{k}$ is the number of directed $k$-simplices in $\mathcal{X}$, $k=0, \ldots,N$.

\item $\mathrm{Str}_{2}(\mathcal{X}) = (s^{2}_{0}, \ldots, s^{2}_{N})$, where $s^{2}_{q}$ is the number of weakly $q$-connected components of $\mathcal{X}$, $q=0, \ldots,N$.

\item $\mathrm{Str}_{3}(\mathcal{X}) = (s^{3}_{0}, \ldots, s^{3}_{N})$, where $s^{3}_{q}$ is the number of strongly $q$-connected components of $\mathcal{X}$, $q=0, \ldots,N$.

\item $\mathrm{Str}_{4}(\mathcal{X}) = (s^{4}_{0}, \ldots, s^{4}_{N})$, where $s^{4}_{q} = 1 - s^{2}_{q}/|\mathcal{V}_q|$, $q=0, \ldots,N$.

\item $\mathrm{Str}_{5}(\mathcal{X}) = (s^{5}_{0}, \ldots, s^{5}_{N})$, where $s^{5}_{q} = 1 - s^{3}_{q}/|\mathcal{V}_q|$, $q=0, \ldots,N$.
\end{enumerate}
\end{definition}

The quantities $s^{4}_{q}$ and $s^{5}_{q}$ can be interpreted as the ``degree of directed connectedness'' among the directed simplices at level $q$.

Building on the previous structure vectors, we now propose a general similarity formula to compare two DFCs.

\begin{definition}
Given two DFCs, $\mathcal{X}_{1}$ and $\mathcal{X}_{2}$, let $\mathrm{Str}_{n}(\mathcal{X}_{1})=(s^{1,n}_{0}, \ldots,s^{1,n}_{N^{n}_{1}})$ and $\mathrm{Str}_{n}(\mathcal{X}_{2})=(s^{2,n}_{0}, \ldots,s^{2,n}_{N^{n}_{2}})$ be their respective $n$-th structure vectors. Without loss of generality, suppose $N_{1}^{n} \ge N_{2}^{n}$. Then consider the new vector $\mathrm{Str}^{\ast}_{n}(\mathcal{X}_{2})=(s^{2,n}_{0}, \ldots,s^{2,n}_{N^{n}_{1}})$, such that $s^{2,n}_{k}=0$ for all $N^{n}_{2} < k \le N^{n}_{1}$. Let $|| \cdot ||_{2}$ denote the Euclidean norm in $\mathbb{R}^{N^{n}_{1}}$. We define the (normalized) \textit{$n$-th structure distance} between $\mathcal{X}_{1}$ and $\mathcal{X}_{2}$ by

\begin{equation}\label{eq:topological-distance}
\widehat{T}^{n}_{sd}(\mathcal{X}_{1}, \mathcal{X}_{2}) = \begin{cases}
	\frac{||\mathrm{Str}_{n}(\mathcal{X}_{1}) - \mathrm{Str}^{\ast}_{n}(\mathcal{X}_{2})||_{2}}{|| \mathrm{Str}_{n}(\mathcal{X}_{1}) ||_{2} + || \mathrm{Str}^{\ast}_{n}(\mathcal{X}_{2}) ||_{2}}, \mbox{ if } \mathcal{X}_{1} \neq \emptyset \mbox{ or } \mathcal{X}_{2} \neq \emptyset,\\
	0, \mbox{ if } \mathcal{X}_{1} = \mathcal{X}_{2} = \emptyset.
\end{cases}
\end{equation}

One can easily verifies that $0 \le \widehat{T}_{sd}^{n} \le 1$. Indeed, since all entries of the vectors $\mathrm{Str}_{n}$ are non-negative, by the triangular inequality of the Euclidean norm, we have
$||\mathrm{Str}_{n}(\mathcal{X}_{1}) - \mathrm{Str}^{\ast}_{n}(\mathcal{X}_{2}) ||_{2} \le ||\mathrm{Str}_{n}(\mathcal{X}_{1})||_{2} + ||\mathrm{Str}^{\ast}_{n}(\mathcal{X}_{2})||_{2}$.

\end{definition}

%---------------------------------------------------------------
\subsubsection{Simplicial Kernels}

Graph kernels are distance-based algorithms that compute a similarity score between two graphs \cite{Borgwardt}. Martino et al. \cite{Martino} proposed four kernels for simplicial complexes, namely: \textit{histogram cosine kernel}, \textit{weighted Jaccard kernel}, \textit{edit kernel}, and \textit{stratified edit kernel}. The first two kernels are based on the count of simplices belonging simultaneously to both simplicial complexes being compared, and the last two kernels are based on the count of edit operations. In what follows, we propose adaptations of these four simplicial kernels to DFCs.

\begin{definition}
Given two DFCs, $\mathcal{X}_{1}$ and $\mathcal{X}_{2}$, let $\mathrm{Str}_{1}(\mathcal{X}_{1})=(s^{1,1}_{0}, \ldots,s^{1,1}_{N^{1}_{1}})$ and $\mathrm{Str}_{1}(\mathcal{X}_{2})=(s^{2,1}_{0}, \ldots,s^{2,1}_{N^{1}_{2}})$ be their respective $1$st structure vectors. Without loss of generality, suppose $N_{1}^{1} \ge N_{2}^{1}$. Then consider the new vector $\mathrm{Str}^{\ast}_{1}(\mathcal{X}_{2})=(s^{2,1}_{0}, \ldots,s^{2,1}_{N^{1}_{1}})$, such that $s^{2,1}_{k}=0$ for all $N^{1}_{2} < k \le N^{1}_{1}$. Let $\langle \cdot, \cdot \rangle$ denote the Euclidean inner product in $\mathbb{R}^{N^{1}_{1}}$. We define the following kernels:

\begin{itemize}
\item The \textbf{(normalized) histogram cosine kernel} (HCK) is defined as
\begin{equation}\label{eq:HCK}
	K_{HC}(\mathcal{X}_{1}, \mathcal{X}_{2}) = \frac{\langle \mathrm{Str}_{1}(\mathcal{X}_{1}), \mathrm{Str}_{1}^{\ast}(\mathcal{X}_{2}) \rangle}{\sqrt{\langle \mathrm{Str}_{1}(\mathcal{X}_{1}), \mathrm{Str}_{1}(\mathcal{X}_{1}) \rangle} \sqrt{\langle \mathrm{Str}_{1}^{\ast}(\mathcal{X}_{2}), \mathrm{Str}_{1}^{\ast}(\mathcal{X}_{2}) \rangle}}.
\end{equation}

\item Let $|\mathcal{X}_{1} \cap \mathcal{X}_{2}|$ be the cardinality of the intersection and $|\mathcal{X}_{1} \cup \mathcal{X}_{2}|$ be the cardinality of the union. The \textbf{Jaccard kernel} is defined as
\begin{equation}\label{eq:jaccard-kernel}
	K_{J}(\mathcal{X}_{1}, \mathcal{X}_{2}) = 
	\begin{cases}
		\frac{|\mathcal{X}_{1} \cap \mathcal{X}_{2}|}{|\mathcal{X}_{1} \cup \mathcal{X}_{2}|}, \mbox{ if } \mathcal{X}_{1} \neq \emptyset \mbox{ or } \mathcal{X}_{2} \neq \emptyset,\\
		0, \mbox{ if } \mathcal{X}_{1} = \mathcal{X}_{2} = \emptyset.
	\end{cases}
\end{equation}

\noindent This kernel is a normalized similarity measure. In fact, since $|\mathcal{X}_{1} \cap \mathcal{X}_{2}| \le |\mathcal{X}_{1} \cup \mathcal{X}_{2}|$, we have $0 \le K_{J} \le 1$, and $K_{J} = 1$ when $\mathcal{X}_{1} = \mathcal{X}_{2}$ and $K_{J} = 0$ when $\mathcal{X}_{1} \cap \mathcal{X}_{2} = \emptyset$.

\item Let $e(\mathcal{X}_{1}, \mathcal{X}_{2})$ be an edit distance between the DFCs (i.e., a distance based on the number of changes necessary to convert one into the other \cite{Gao}). The \textbf{(normalized) edit kernel} is defined as
\begin{equation}\label{eq:edit-kernel1}
	K_{E}(\mathcal{X}_{1}, \mathcal{X}_{2}) = 
	\begin{cases}
		1 - \frac{ 2e(\mathcal{X}_{1}, \mathcal{X}_{2}) }{ |\mathcal{X}_{1}| + |\mathcal{X}_{2}| + e(\mathcal{X}_{1}, \mathcal{X}_{2}) }, \mbox{ if } \mathcal{X}_{1} \neq \emptyset \mbox{ or } \mathcal{X}_{2} \neq \emptyset,\\
		0, \mbox{ if } \mathcal{X}_{1} = \mathcal{X}_{2} = \emptyset.
	\end{cases}
\end{equation}

\item Let $\mathcal{D}$ denote the set of all different dimensions of the simplices present in $\mathcal{X}_{1}$ and $\mathcal{X}_{2}$. Let $\mathcal{X}_{1}^{k} \subseteq \mathcal{X}_{1}$ and $\mathcal{X}_{2}^{k} \subseteq \mathcal{X}_{2}$ be the subsets of all directed $k$-simplices in the respective complexes, and let $K_{E}$ be the edit kernel (\ref{eq:edit-kernel1}). The \textbf{stratified edit kernel} (SEK) is defined as
\begin{equation}\label{eq:stratified-edit-kernel}
	K_{SE}(\mathcal{X}_{1}, \mathcal{X}_{2}) = \frac{1}{|\mathcal{D}|}\sum_{k \in \mathcal{D}}  K_{E}(\mathcal{X}_{1}^{k}, \mathcal{X}_{2}^{k}).
\end{equation}

\end{itemize}

\end{definition}

%%%%%%%%%%%%%%%%%%%%%%%%%%%%%%%%%%%%%%%%%%
\section{Application to Biological Data}
\label{sec:bio_appli}

We chose three analysis measures (lower variants): (a) \textit{(directed simplicial) $q$-harmonic centrality} (HC) (\ref{eq:simp-harmonic}), (b) \textit{(directed simplicial) $q$-betweenness centrality} (BC) (\ref{eq:simp-betweenness}), and (c) \textit{(simplicial) in-$q$-Katz centrality} (KC) (\ref{eq:simp-dir-katz-centrality-in}) and applied them to a subnetwork of neurons within the rostral ganglia of \textit{C. elegans}, consisting of 131 neurons (nodes) and 764 directed connections (arcs) \cite{Kaiser}, where we removed all double edges in such a way that all directed cliques were preserved. This led to a digraph with 687 arcs. The dataset is available from the Stanford Network Analysis Project (SNAP) webpage\footnote{\url{https://snap.stanford.edu/data/C-elegans-frontal.html}}. The choice of the \textit{C. elegans} nematode neuronal network was dictated by its relatively small size, and because it has been fully worked out \cite{Varshney}.

%------------------------------------------------------
\subsection{Methodology}

We computed the associated (lower) $q$-digraph for $q=0,1,2,3$ (Figure~\ref{fig:Networks}), focusing on the network without double edges, and examined the maximum values of the measures across all nodes.

For structural comparison we performed 30 simulations of two types of random digraphs representing distinct null models with identical characteristics as to number of nodes, edges, and same total degree (in-degree plus out-degree) sequence: the first one constructed according to the the Maslov-Sneppen rewiring model  \cite{Maslov} and the second one according to the lattice rewiring model \cite{Fornito-2016} and employed a z-test (at a significance level $\alpha = 0.05$) as the benchmark.

For $q$-digraph construction and simplicial measures computations, we employed the \texttt{DigplexQ} Python package\footnote{\url{https://github.com/heitorbaldo/DigplexQ}}.

\begin{figure}[h!]
\centering
\includegraphics[scale=0.47]{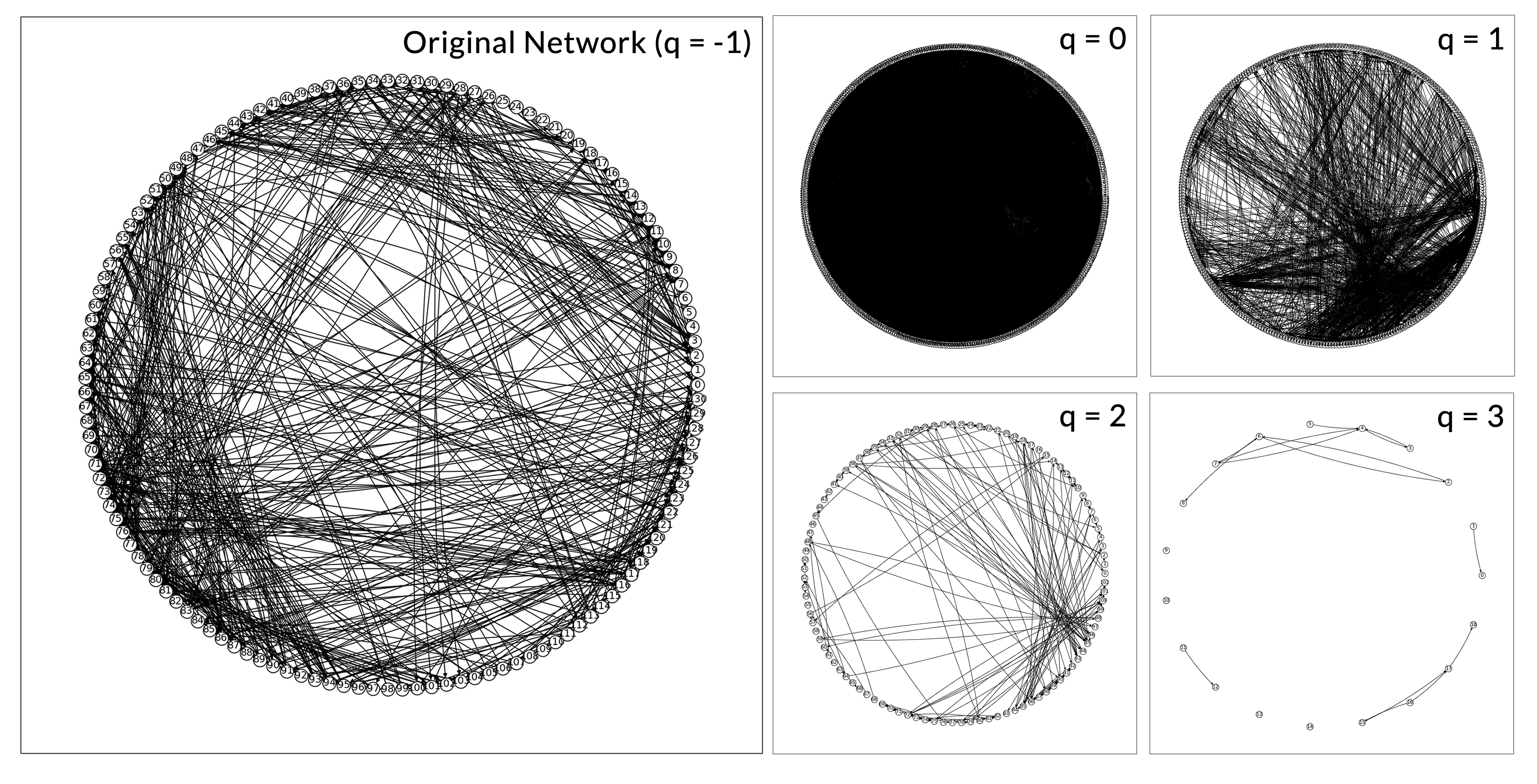}
\caption{\textit{C. elegans} frontal neuronal network and its $q$-digraphs ($q=0,1,2,3$). Here, we adopt the nomenclature ($-1$)-digraph (i.e., $q=-1$) to denote the original network.}
\label{fig:Networks}
\end{figure}

%------------------------------------------------------
\subsection{Results and Discussion}

The results are summarized in Figure \ref{fig:results} with the associated z-test results in Table \ref{Table2}. The nomenclature ($-1$)-digraphs (or $q=-1$) denotes the original networks (at the node level).

\begin{figure}[h!]\label{montage-eeg}
\centering
\includegraphics[scale=1.2]{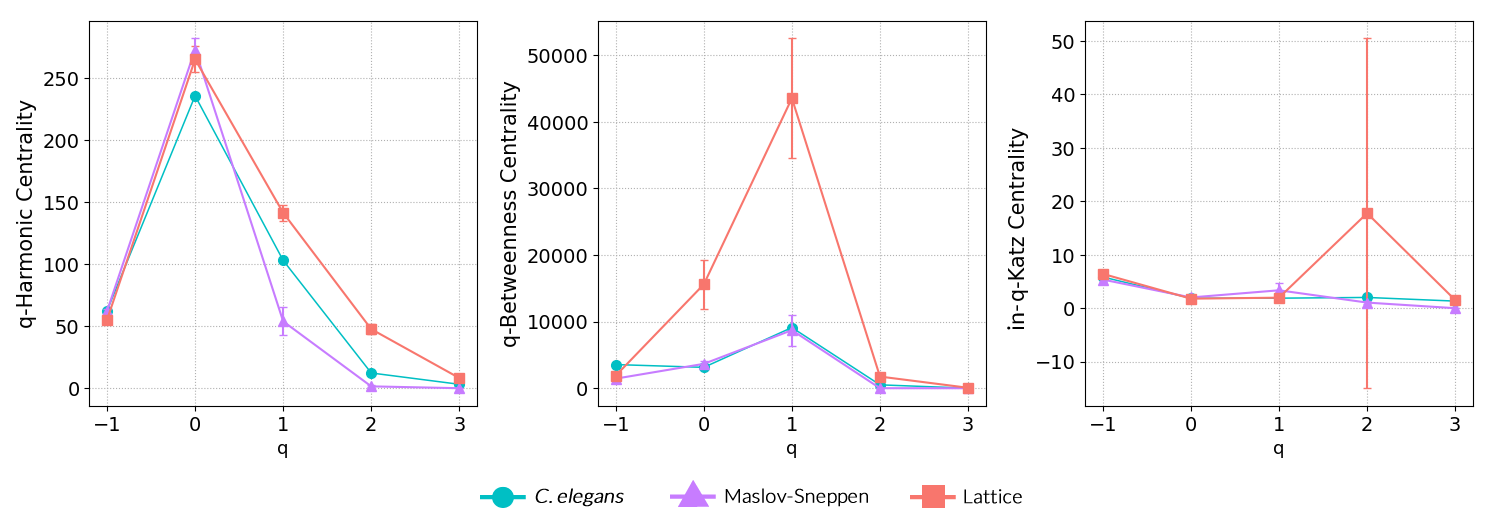}
\caption{Results of the three local measures for each level $q=-1,0,1,2,3$ obtained for the \textit{C. elegans} network and for the two null models (means and standard deviations for the null models).}
\label{fig:results}
\end{figure}

%{ \renewcommand{\arraystretch}{2}
\begin{table}[h!]
\footnotesize
\centering
\caption{Z-scores for each null model at each level $q=-1,0,1,2,3$. Z-scores associated with p-values $p < 0.05$ are highlighted in bold.}
\begin{tabular}{c c c c c c c c c c c}
	\toprule % <-- Toprule here
	\multicolumn{1}{c}{} & \multicolumn{5}{c}{\textbf{q (Maslov-Sneppen) }} & \multicolumn{5}{c}{\textbf{q (Lattice) }} \\
	\cmidrule(rl){2-6} \cmidrule(rl){7-11} \textbf{Measure} & \textbf{-1} & \textbf{0} & \textbf{1} & \textbf{2} & \textbf{3}  & \textbf{-1} & \textbf{0} & \textbf{1} & \textbf{2} & \textbf{3}\\
	\midrule % <-- Midrule here
	HC & 0.18& \textbf{-3.67}& \textbf{3.97}&  \textbf{13.79}& -& \textbf{3.75}& \textbf{-2.51}& \textbf{-5.22}& \textbf{-7.38}& -1.59\\
	BC & \textbf{19.65}& -1.23& 0.15&  \textbf{467.48}& - & \textbf{11.84}& \textbf{-3.04}& \textbf{-3.45}& \textbf{-2.02}& -0.86\\
	KC &  0.98& \textbf{-2.34}& -1.01&  \textbf{2.45}& - & -0.88& \textbf{2.34}& -1.56& -0.43& -0.91 \\
	
	\bottomrule % <-- Bottomrule here
\end{tabular}
\label{Table2}
\end{table}
%}

The following comments apply: 

\begin{itemize}

\item \textbf{$q$-Harmonic centrality (HC):} The \textit{C. elegans} network was significantly different from the MS model at levels $q=0,1,2$, but not at level $q=-1$, whereas it was statistically different from the lattice model at all levels except at $q=3$. 

\item \textbf{$q$-Betweenness centrality (BC):} There was a statistically significant difference between the \textit{C. elegans} network and the MS model at levels $q=-1,2$, but not at levels $q=0,1$, whereas it was statistically different at all levels except at level $q=3$ in the case of the lattice model.  

\item \textbf{in-$q$-Katz centrality (KC):} The \textit{C. elegans} network was statistically different from the MS model at levels $q=0,2$, but not at levels $q=-1,1$, whereas it was statistically different from the lattice model only at level $q=0$.

\end{itemize}

These results show that the \textit{C. elegans} network approximates neither of these two random models in terms of topological organization levels. More importantly, for both random models, some of the usual graph measures were unable to detect significant differences, suggesting that common topological analysis at the node level ($q=-1$) is not sufficient to reveal nuances in the network's topological organization and demands the present $q$-digraph formalism and associated quantitative approaches to reveal topological differences that would otherwise go unnoticed using usual graph analysis approaches.

%%%%%%%%%%%%%%%%%%%%%%%%%%%%%%%%%%%%%%%%%%
\section{Conclusions}
\label{sec:conclu}

Here, we developed a rigorous mathematical formalism for topology research on digraphs based on their directed clique complexes and their directed higher-order connectivity, within the framework of a more complete directed Q-analysis and quantitative techniques based on complexes and on the concept of maximal/lower $q$-digraph (including the weighted case), namely, through simplicial characterization measures and simplicial similarity comparison methods. We also demonstrated this via the simple \textit{C. elegans} frontal neuronal network example, the lower $q$-digraph formalism and the associated quantitative methods can provide a more detailed analysis of digraph topology, thereby revealing properties beyond those available from standard graph-theoretic techniques.

Given this work's focus on establishing the pertinent mathematics, many of the simplicial characterization measures (both maximal and lower variants) and simplicial similarity comparison methods must still be applied to other real-world data. Also, comparisons between the weighted and unweighted versions of these measures were left for future work.

%%%%%%%%%%%%%%%%%%%%%%%%%%%%%%%%%%%%%%%%%%
\vspace{6pt} 

%%%%%%%%%%%%%%%%%%%%%%%%%%%%%%%%%%%%%%%%%%

\noindent \textbf{Acknowledgments:}
This work has been supported by the JSPS KAKENHI grant number JP25K15019, FAPESP grants 2013/07699-0, 2023/18337-3, 2024/03261-4, and 2024/09195-3, CNPq grants 306811/2022-7, 402309/2024-3, and 443972/2024-9, CAPES (finance code 001), the Alexander von Humboldt Foundation, the MEXT Cooperative Research Project Program, the Medical Research Center Initiative for High-Depth Omics and CURE: JPMXP1323015486 for MIB, and the Pan-Omics Data-Driven Innovation Research Center, Kyushu University. The infrastructure of the Omics Science Center Secure Information Analysis System, Medical Institute of Bioregulation at Kyushu University, provides (part of) the computational resource.

\vspace{6pt}

%%%%%%%%%%%%%%%%%%%%%%%%%%%%%%%%%%%%%%%%%%%%%
%%%%%%%%%%%%%%%%%%%%%%%%%%%%%%%%%%%%%%%%%%%%%

%\bibliographystyle{apa}
%\bibliographystyle{apalike}
%\bibliographystyle{abbrv}
{\footnotesize
\bibliographystyle{acm}
\bibliography{bibliography}
}

%\includepdf[pages=-]{img/supplementary_material.pdf}

%=====================================================================
%=====================================================================
\newpage
\clearpage
\setcounter{section}{0}
\setcounter{table}{0}
\setcounter{figure}{0}
\setcounter{page}{1}

\begin{center}

\begin{spacing}{1.8}
{\LARGE Supplementary Material for ``Directed Q-Analysis and Directed Higher-Order Connectivity on Digraphs: A Quantitative Approach"}
\end{spacing}

\bigskip

\author{
Heitor Baldo$^{1,2}$,
Luiz A. Baccal\'a$^{3}$,
Andr\'e Fujita$^{1,4}$,
and Koichi Sameshima$^{5}$\\[1em]
\textit{
$^{1}$Department of Computer Science, Institute of Mathematics, Statistics, and Computer Science, University of S\~ao Paulo, Brazil\\
$^{2}$Institute for Globally Distributed Open Research and Education (IGDORE),
Sweden\\
$^{3}$Department of Telecommunications and Control Engineering,
Escola Polit\'ecnica, University of S\~ao Paulo, Brazil\\
$^{4}$Division of Network AI Statistics, Medical Institute of Bioregulation,
Kyushu University, Japan\\
$^{5}$Department of Radiology and Oncology, School of Medicine,
University of S\~ao Paulo, Brazil
}
}
\end{center}

\bigskip

\begin{abstract}
Supplementary material for ``Directed Q-Analysis and Directed Higher-Order Connectivity on Digraphs: A Quantitative Approach" includes reference tables, worked examples, and simulation studies.
\end{abstract}

\bigskip

%---------------------------------------------------------
%---------------------------------------------------------
\section{Conventions: Maximal and Lower $q$-Digraph Variants}

Throughout this text, $\mathrm{dFl}(G)$ denotes the directed flag complex of a simple digraph $G$, and $\mathcal{G}_q =(\mathcal{V}_q, \mathcal{E}_q)$ denotes its (maximal or lower) $q$-digraph with $q$-adjacency matrix $\mathcal{H}_q = (h_{\sigma\tau})$, for $0 \le q \le \dim \mathrm{dFl}(G)$. For each of the measures defined in Section 3, corresponding \emph{maximal} and \emph{lower} variants are obtained by replacing the generic notations as described in Table~\ref{tab:tablesup1} with the corresponding maximal or lower notations (e.g., $\mathcal{H}_{q}$ by $\mathcal{H}_{q}^{A}$ or $\mathcal{H}_{q}^{L}$, $\vec{d}_{q}$ by $\vec{d}^{A}_{q}$ or $\vec{d}^{L}_{q}$, and $\deg^{\bullet}_{q}$ by $\deg^{\bullet}_{A_{q^{*}}}$ or $\deg^{\bullet}_{L_{q}}$).

{\renewcommand{\arraystretch}{1.35}
\begin{table}[h!]
\centering
\small
\caption{Notation conventions for maximal and lower variants of the simplicial measures.  All reference tables of simplicial measures in this document use the \emph{generic} column.}
\begin{tabular}{lccc}
\hline
\textbf{Object} & \textbf{Generic notation} & \textbf{Maximal variant} & \textbf{Lower variant}\\
\hline

$q$-Digraph & $\mathcal{G}_{q}$        & $\mathcal{G}_{q}^{A}$        & $\mathcal{G}_{q}^{L}$ \\

$q$-Arc set &  $\mathcal{E}_{q}$        & $\mathcal{E}_{q}^{A}$        & $\mathcal{E}_{q}^{L}$ \\

$q$-Arc & $(\sigma, \tau)$  & $(\sigma, \tau)_{A}$   & $(\sigma, \tau)_{L}$  \\

$q$-Adjacency matrix & $\mathcal{H}_{q}$        & $\mathcal{H}_{q}^{A}$        & $\mathcal{H}_{q}^{L}$ \\

Directed $q$-distance & $\vec{d}_{q}$            & $\vec{d}_{q}^{A}$            & $\vec{d}_{q}^{L}$ \\

($\bullet$)-$q$-Star & $\mathrm{st}^{\bullet}_{q}$ & $\mathrm{st}^{\bullet}_{A_{q^{*}}}$ & $\mathrm{st}^{\bullet}_{L_{q}}$ \\

($\bullet$)-$q$-Degree & $\deg^{\bullet}_{q}$ & $\deg^{\bullet}_{A_{q^{*}}}$ & $\deg^{\bullet}_{L_{q}}$ \\

\hline
\end{tabular}
\label{tab:tablesup1}
\end{table}
}

\noindent The worked examples (Section 4) and the simulation studies (Section 5) use the \textit{maximal} variant throughout (i.e. $\mathcal{G}_{q}^{A}$, $\mathcal{H}_{q}^{A}$, $\vec{d}_{q}^{A}$, $\deg^\bullet_{A_{q^*}}$).

%---------------------------------------------------------
%---------------------------------------------------------
\section{Reference Tables of Symbols}
\label{sec:symbols}

%-----------
\subsection{Graph and Digraph Notations}

{\renewcommand{\arraystretch}{1.35}
\begin{longtable}{>{\raggedright}p{2.4cm} p{6.0cm} p{6.0cm}  }
\caption{Graph and digraph notation.}
\label{tab:table1}\\
\toprule
\textbf{Symbol} & \textbf{Meaning} & \textbf{Notes / References} \\
\midrule % <-- Midrule here

$G=(V,E)$ & Graph or digraph & $V$ = vertex set, $E$ = edge/arc set\\

\rowcolor{lightgray}
$|V|$ & Number of vertices (order) & Also written $n$\\

$|E|$ & Number of edges / arcs & Also written $m$\\

\rowcolor{lightgray}
$(v,u) \in E$ & Arc from $v$ to $u$ & Ordered pair; $(v,u)\ne(u,v)$\\

$\mathcal{N}(v)$ & Neighbourhood of $v$ & $\{u:(v,u)\in E\}$\\

\rowcolor{lightgray}
$\mathcal{N}^-(v)$ & In-neighbourhood & $\{u:(u,v)\in E\}$\\

$\mathcal{N}^+(v)$ & Out-neighbourhood & $\{u:(v,u)\in E\}$\\

\rowcolor{lightgray}
$\deg(v)$ & Degree of $v$ & $|\mathcal{N}(v)|$\\

$\deg^-(v)$ & In-degree of $v$ & $|\mathcal{N}^-(v)|$\\

\rowcolor{lightgray}
$\deg^+(v)$ & Out-degree of $v$ &  $|\mathcal{N}^+(v)|$\\

$G^\omega = (V,E,\omega)$ & Weighted digraph & $\omega:V\times V\to\mathbb{R}_{\ge0}$\\

\rowcolor{lightgray}
$\deg^\omega(v)$ & Weighted degree & $\sum_{u\in\mathcal{N}(v)}\omega(v,u)$\\

$\deg^-_\omega(v)$ & Weighted in-degree & $\sum_{u\in\mathcal{N}^-(v)}\omega(u,v)$\\

\rowcolor{lightgray}
$\deg^+_\omega(v)$ & Weighted out-degree & $\sum_{u\in\mathcal{N}^+(v)}\omega(v,u)$\\

$A = (a_{ij})$ & Adjacency matrix & $a_{ij} = 1$ iff $(i,j)\in E$\\

\rowcolor{lightgray}
$(k+1)$-clique & Complete induced subgraph & $k+1$ mutually adjacent vertices\\

\bottomrule % <-- Bottomrule here

\end{longtable}
}

%-----------
\subsection{Simplicial Complex and Directed Flag Complex Notations}

{\renewcommand{\arraystretch}{1.35}
\begin{longtable}{>{\raggedright}p{2.4cm} p{6.0cm} p{6.0cm}  }
\caption{Simplicial complex and directed flag complex notation.}
\label{tab:table2}\\
\toprule
\textbf{Symbol} & \textbf{Meaning} & \textbf{Notes / References} \\
\midrule % <-- Midrule here

$\mathcal{X}$ & Abstract simplicial complex (ASC) & Closed under inclusion\\

\rowcolor{lightgray}
$\sigma^{(n)}$ & $n$-simplex & $|\sigma|=n+1$; dimension $=n$\\

$\dim\sigma$ & Dimension of $\sigma$ & $|\sigma|-1$\\

\rowcolor{lightgray}
$\dim\mathcal{X}$ & Dimension of $\mathcal{X}$ & $\max_{\sigma\in\mathcal{X}}\dim\sigma$\\

$\mathcal{X}_k$ & Set of all $k$-simplices in $\mathcal{X}$ & $k$-th level\\

\rowcolor{lightgray}
$V_\mathcal{X}$ & Vertex set of $\mathcal{X}$ & $\bigcup_{\sigma\in\mathcal{X}}\sigma$\\

$\tau \subseteq \sigma$ & $\tau$ is a face of $\sigma$ & If $\tau\subset\sigma$: proper face\\

\rowcolor{lightgray}
$\hat{d}_i(\sigma^{(n)})$ & $i$-th face map & Deletes $i$-th vertex (Def. 3.4)\\

$\mathcal{X}_q$ & Simplices of dim $\ge q$ & $\{\sigma^{(n)}:q\le n\}$\\

\rowcolor{lightgray}
ADSC & Abstract directed simplicial complex & Ordered simplices (Def. 3.3)\\

$[v_0,\ldots,v_k]$ & Directed $k$-simplex & Totally ordered vertex set\\

\rowcolor{lightgray}
$\mathrm{Fl}(G)$ & Flag (clique) complex of $G$ & Undirected version (Def.\ 4.1)\\

$\mathrm{dFl}(G)$ & Directed flag complex (DFC) & Directed version (Def. 4.3)\\

\rowcolor{lightgray}
$(\mathcal{X},\tilde{\omega})$ & Weighted ADSC & Weight function $\tilde{\omega}:\mathcal{X}\to\mathcal{R}$\\

$\tilde{\omega}(i)$ & Node-weight & $\max(\deg^-_\omega(i),\deg^+_\omega(i))$ (Def. 5.2) \\

\rowcolor{lightgray}
$\tilde{\omega}(\sigma^{(n)})$ & Product-weight of simplex & $\prod_{i=0}^n\tilde{\omega}(i)$ (Def. 5.3) \\

\bottomrule % <-- Bottomrule here

\end{longtable}
}

%-----------
\subsection{Directed Q-Analysis Notations}

{\renewcommand{\arraystretch}{1.35}
\begin{longtable}{>{\raggedright}p{2.4cm} p{6.0cm} p{6.0cm}  }
\caption{Directed Q-Analysis: adjacency and connectivity relations.}
\label{tab:table3}\\
\toprule
\textbf{Symbol} & \textbf{Meaning} & \textbf{Notes / References} \\
\midrule % <-- Midrule here

$\sigma \sim_q \tau$ & $q$-near (undirected) & Share a $q$-face (Def. 6.1)\\

\rowcolor{lightgray}
$\sigma \bm{\sim}_{\mathbf{q}} \tau$ & $q$-connected (undirected) & Chain of $q$-connections (Def. 6.1)\\

$\sigma \sim^{+}_{q} \tau$ & Out-$q$-near ($+$) & $(q,\hat{d}_i,\hat{d}_j)$-near with $i\le j$ (Def. 6.6)\\

\rowcolor{lightgray}
$\sigma \sim^{-}_{q} \tau$ & In-$q$-near ($-$) & $(q,\hat{d}_i,\hat{d}_j)$-near with $i\ge j$ (Def. 6.6)\\

$\sigma \sim^{\pm}_{q} \tau$ & Bidirectional $q$-near & Both $+$ and $-$ (Def. 6.6)\\

\rowcolor{lightgray}
$\bullet\in\{-,+,\pm\}$ & Direction symbol & $+$ = out, $-$ = in, $\pm$ = bidir.\\

$\sigma \bm{\sim}^{\bullet}_{\mathbf{q}}\tau$ & $({\bullet})$-$q$-connected & Directed $q$-chain (Def. 6.8)\\

\rowcolor{lightgray}
$\sigma\sim^{\bullet}_{L_{q}}\tau$ & Lower $({\bullet})$-$q$-adjacent & Same as $({\bullet})$-$q$-near (Def. 6.10)\\

$\sigma\sim^{\bullet}_{L_{q^{*}}}\tau$ & Strictly lower adjacent & $q$-near, not $(q+1)$-near (Def. 6.10)\\

\rowcolor{lightgray}
$\sigma\sim^{\bullet}_{U_{p}}\tau$ & Upper $({\bullet})$-$p$-adjacent & Nested in common $p$-simplex (Def. 6.12)\\

$\sigma\sim^{\bullet}_{U_{p^*}}\tau$ & Strictly upper adjacent & $p$-near, not $(p+1)$-near (Def. 6.12)\\

\rowcolor{lightgray}
$\sigma\sim^{\bullet}_{A_{q}}\tau$ & General $({\bullet})$-$q$-adjacent & Strict lower \& not upper (Def. 6.13)\\

$\sigma \sim^{\bullet}_{A_{q^*}}\tau$ & Maximal $({\bullet})$-$q$-adjacent & Cannot extend (Def. 6.13)\\

\rowcolor{lightgray}
$\sigma \bm{\sim}^{\bullet}_{A_{q^*}} \tau$ & Maximal $({\bullet})$-$q$-connected & Via maximal $q$-walk (Def. 6.16)\\

$S^s_q$ & Strongly $q$-conn.\ relation & Equivalence relation (Prop. 6.18)\\

\rowcolor{lightgray}
$S^{w}_{q}$ & Weakly $q$-conn.\ relation & Disregards direction\\

$\mathrm{dFl}_{q}(G)$ & All simplices of dim $\ge q$ & Subcomplex (Def. 6.14)\\

\rowcolor{lightgray}
$\mathrm{dFl}^{*}(G)$ & Set of maximal simplices & Not face of any other (Def. 6.14)\\

$\mathrm{dFl}^{*}_{q}(G)$ & Maximal simplices, dim $\ge q$ & Vertices of $\mathcal{G}_q$ (Def. 6.14)\\

\bottomrule % <-- Bottomrule here

\end{longtable}
}

%-----------
\subsection{$q$-Digraph Notations (Generic, Lower, and Maximal)}

{\renewcommand{\arraystretch}{1.35}
\begin{longtable}{>{\raggedright}p{2.4cm} p{6.0cm} p{6.0cm}  }
\caption{$q$-Digraph symbols.  Generic notation applies to both maximal ($A$) and lower ($L$) variants; see Table~\ref{tab:tablesup1} for the substitution rules.}
\label{tab:table4}\\
\toprule
\textbf{Symbol} & \textbf{Meaning} & \textbf{Notes / References} \\
\midrule % <-- Midrule here

$\mathcal{G}_q = (\mathcal{V}_q, \mathcal{E}_q)$ & Generic $q$-digraph & See Table~\ref{tab:tablesup1} for variants\\

\rowcolor{lightgray}
$\mathcal{G}_{q}^{A} = (\mathcal{V}_q, \mathcal{E}_{q}^{A})$ & Maximal $q$-digraph & Def. 6.21\\

$\mathcal{G}_{q}^{L} = (\mathcal{V}_q, \mathcal{E}_{q}^{L})$ & Lower $q$-digraph & Def. 6.21 (lower variant)\\

\rowcolor{lightgray}
$\mathcal{V}_q$ & Vertex set of $\mathcal{G}_q$ & $=\mathrm{dFl}^*(G)$ (common to both variants)\\

$\mathcal{E}_q$ & Arc set of $\mathcal{G}_q$ & $q$-arcs\\

\rowcolor{lightgray}
$|\mathcal{V}_q|$ & Number of vertices in $\mathcal{G}_q$ & Varies with $q$\\

$(\sigma, \tau)$ & Generic $q$-arc from $\sigma$ to $\tau$ & When $\sigma \sim^+_{A_{q^*}} \tau$ or $\sigma \sim^+_{L^q} \tau$\\

\rowcolor{lightgray}
$\mathcal{H}_q = (h_{\sigma\tau})$ & Generic $q$-adjacency matrix & $\mathcal{H}_{q}^{A}$ Eq. (9) or $\mathcal{H}_{q}^{L}$ (lower variant, Eq. (10)\\

$\mathcal{G}^{\tilde\omega}_q$ & Weighted $q$-digraph & Def. 6.24\\

\rowcolor{lightgray}
$\mathcal{H}^{\tilde\omega}_q$ & Weighted $q$-adjacency matrix & Eq. (12)\\

$\vec{d}_q(\sigma, \tau)$ & Generic directed simplicial $q$-distance & $\vec{d}_{q}^{A}$ (maximal) or $\vec{d}_{q}^{L}$ (lower) (shortest $q$-walk length (Def. 6.17))\\

\rowcolor{lightgray}
$\vec{d}^\omega_q(\sigma,\tau)$ & Weighted $q$-distance & Uses weight-to-dist. function $F$\\

$s^q_{\sigma \to \tau}$ & Shortest directed simplicial $q$-walk & Attains $\vec{d}_q(\sigma,\tau)$\\

\rowcolor{lightgray}
$N_q$ & Size of giant $q$-component & Largest weakly $q$-conn.\ component\\

$\mathrm{st}^{\bullet}_{L_{q}}(\sigma)$ & Lower $({\bullet})$-$q$-star & $\{\tau:\sigma\sim^{\bullet}_{L_{q}}\tau\}$ (Def. 6.26)\\

\rowcolor{lightgray}
$\mathrm{st}^{\bullet}_{A_{q^*}}(\sigma)$ & Maximal $({\bullet})$-$q$-star & $\{\tau: \sigma \sim^{\bullet}_{A_{q^*}}\tau\}$ (Def. 6.26)\\

$\deg^{\bullet}_{q}(\sigma)$ & Generic $({\bullet})$-$q$-degree & $\deg^\bullet_{A_{q^*}}$ or $\deg^\bullet_{L_q}$; see Table~\ref{tab:lower-maximal-conventions}\\

\rowcolor{lightgray}
$\deg^{\bullet}_{A_{q^*}}(\sigma)$ & Maximal $({\bullet})$-$q$-degree & $|\mathrm{st}^{\bullet}_{A_{q^*}}(\sigma)|$ (Def. 6.29)\\

$\deg^{\bullet}_{L_q}(\sigma)$ & Lower $({\bullet})$-$q$-degree & $|\mathrm{st}^{\bullet}_{L_{q}}(\sigma)|$ (Def. 6.29)\\

\rowcolor{lightgray}
$\mathrm{lk}^-(\sigma),\mathrm{lk}^+(\sigma)$ & In/out-link of $\sigma$ & Generalised neighbourhood (Def. 6.28)\\

$\mathrm{hub}(\mathcal{F})$ & Hub of simplicial a family $\mathcal{F}$ & $\bigcap_{\sigma\in\mathcal{F}}\sigma$ (Def. 6.27)\\

\rowcolor{lightgray}
$\mathbf{f}(\mathcal{X})$ & Structure vector & $(f_0,\ldots,f_{\dim\mathcal{X}})$ (Def.\ 6.1) (classical)\\

$\mathrm{Str}_n(\mathcal{X})$ & $n$-th structure vector & 5 variants (Def. 7.1)\\

\bottomrule % <-- Bottomrule here

\end{longtable}
}

%---------------------------------------------------------
%---------------------------------------------------------
\section{Reference Tables of Simplicial Measures}
\label{supmeasures}

\noindent \textbf{Generic notation.} All formulas in this section use the \emph{generic} symbols $\mathcal{G}_{q}^{L}$, $\mathcal{H}_{q}^{L}$, $\vec{d}_{q}$, and $\deg^\bullet_q$ (see Table~\ref{tab:tablesup1}). The \textit{maximal} variant is obtained by reading $\mathcal{G}_{q}^{A}$, $\mathcal{H}_{q}^{A}$, $\vec{d}_{q}^{A}$, $\deg^\bullet_{A_{q^*}}$; the \textit{lower} variant by reading $\mathcal{G}_{q}^{L}$, $\mathcal{H}_{q}^{L}$, $\vec{d}_{q}^{L}$, $\deg^\bullet_{L_q}$.

\medskip

%-----------
\subsection{Distance-Based Simplicial Measures}

{\renewcommand{\arraystretch}{1.35}
\begin{longtable}{>{\raggedright}p{3.5cm} p{5.6cm}  p{0.8 cm}  p{4.5cm} }
\caption{Distance-based simplicial measures on a $q$-digraph $\mathcal{G}_q$. $\mathcal{H}^{'}_{q}$ is the $q$-adjacency matrix of the underlying undirected $q$-graph. Convention: $\vec{d}_q(\sigma,\tau) = 0$ when no directed path exists (for $\vec{\bar{L}}_q$).}
\label{tab:table5}\\
\toprule
\textbf{Measure} & \textbf{Formula} & \textbf{Eq.} & \textbf{Graph Analogue} \\
\midrule % <-- Midrule here

Avg.\ shortest directed $q$-walk length $\vec{\bar{L}}_q$ & $\displaystyle\frac{1}{|\mathcal{V}_q|}\sum_{\sigma \in \mathcal{V}_q}\frac{\sum_{\tau \ne \sigma} \vec{d}_{q} (\sigma,\tau)}{|\mathcal{V}_q| - 1}$ & (19) & Characteristic path length\\

\rowcolor{lightgray}
Directed simplicial $q$-eccentricity $\mathrm{ecc}_q(\sigma)$ &
$\max_{\tau \in \mathcal{V}_q} \vec{d}_q(\sigma,\tau)$ &
(20) & Eccentricity\\

Directed simplicial $q$-diameter $\mathrm{diam}(\mathcal{G}_q)$ & $\max_{\sigma \in \mathcal{V}_q}\mathrm{ecc}_q(\sigma)$ &
(21) & Diameter\\

\rowcolor{lightgray}
Directed simplicial $q$-radius $\mathrm{rad}(\mathcal{G}_q)$ &
$\min_{\sigma \in \mathcal{V}_q}\mathrm{ecc}_q(\sigma)$ &
(22) & Radius\\

Directed simplicial global $q$-efficiency $\vec{E}^q_{\mathrm{glob}}$ & $\displaystyle\frac{1}{|\mathcal{V}_q|}\sum_\sigma\frac{\sum_{\tau \ne \sigma} \vec{d}_q(\sigma,\tau)^{-1}}{|\mathcal{V}_q| - 1}$ & (23) & Global efficiency\\

\rowcolor{lightgray}
Simplicial $q$-communicability $\mathrm{CM}_q(\sigma,\tau)$ &
$\big(\exp(\mathcal{H}_q)\big)_{\sigma\tau}$ & (24) & Communicability\\

Simplicial $q$-returnability $K_{r,q}$ & $\mathrm{Tr}(\exp(\mathcal{H}_q))-|\mathcal{V}_q|$ & (25) & Returnability\\

\rowcolor{lightgray}
Relative $q$-returnability $K'_{r,q}$ &
$\dfrac{\mathrm{Tr}(\exp(\mathcal{H}_q))-|Vq|}{\mathrm{Tr}(\exp(\mathcal{H}^{'}_{q}))-|\mathcal{V}_q|}$ &
(26) & Relative returnability\\

\bottomrule % <-- Bottomrule here

\end{longtable}
}

%-----------
\subsection{Simplicial Centrality Measures}

{\renewcommand{\arraystretch}{1.35}
\begin{longtable}{>{\raggedright}p{3.5cm} p{5.6cm}  p{0.8 cm}  p{4.5cm}  }
\caption{Simplicial centrality measures. $r_{\mathcal{G}_{q}}(\sigma)$ = vertices reachable from $\sigma$; $\vec{l}^q_{\tau'\tau}(\sigma)$ = number of shortest $q$-walks $\tau'\to\tau$ via $\sigma$; $\vec{l}^q_{\tau'\tau}$ = total shortest walks $\tau'\to \tau$; $\alpha < 1/\lambda^q_1$ for Katz convergence.}
\label{tab:table6}\\
\toprule
\textbf{Measure} & \textbf{Formula} & \textbf{Eq.} & \textbf{Graph Analogue} \\
\midrule % <-- Midrule here

Simplicial in-$q$-degree centrality $C^-_{\mathrm{deg}_q}(\sigma)$ & $\displaystyle\frac{\deg^-_{q}(\sigma)}{|\mathcal{V}_q|-1}$ & (27) & In-degree centrality\\

\rowcolor{lightgray}
Simplicial out-$q$-degree centrality $C^+_{\mathrm{deg}_q}(\sigma)$ &
$\displaystyle\frac{\deg^+_{q}(\sigma)}{|\mathcal{V}_q|-1}$ & (28) & Out-degree centrality\\

Directed simplicial $q$-closeness $\vec{Cl}_q(\sigma)$ &
$\frac{1}{\sum_{\tau\ne\sigma} \vec{d}_q(\sigma,\tau)}$ & (29) & Closeness centrality\\

\rowcolor{lightgray}
Normalised $q$-closeness $\vec{Cl}_q(\sigma)$ &
$\displaystyle\frac{N_q - 1}{\sum_{\tau \ne \sigma} \vec{d}_q(\sigma,\tau)}$ & (30) & Normalised closeness\\

Directed $q$-harmonic centrality $\vec{HC}_q(\sigma)$ &
$\displaystyle\sum_{\tau \ne \sigma}\frac{1}{\vec{d}_q(\sigma,\tau)}$ & (31) & Harmonic centrality\\

\rowcolor{lightgray}
Directed $q$-betweenness $\vec{B}_q(\sigma)$ &
$\displaystyle\sum_{\tau' \ne \tau\ne\sigma}\frac{\vec{l}^q_{\tau' \tau}(\sigma)}{\vec{l}^q_{\tau'\tau}}$ & (32) & Betweenness centrality\\

Normalised $q$-betweenness $\vec{B}_q(\sigma)$ &
$\displaystyle\frac{1}{(N_q-1)(N_q-2)}\sum_{\tau'\ne\tau\ne\sigma}\frac{\vec{l}^q_{\tau'\tau}(\sigma)}{\vec{l}^q_{\tau'\tau}}$ & (33) & Normalised betweenness\\

\rowcolor{lightgray}
Local $q$-reaching centrality $C_{R,q}(\sigma)$ &
$\displaystyle\frac{r_{\mathcal{G}_q}(\sigma)}{|\mathcal{V}_q|-1}$ & (34) & Reaching centrality\\

Global $q$-reaching centrality $\mathrm{GRC}_q$ &
$\displaystyle\frac{\sum_\sigma[C^{\max}_{R,q} - C_{R,q}(\sigma)]}{|\mathcal{V}_q|-1}$ & (35) & Global reaching centrality\\

\bottomrule % <-- Bottomrule here
\end{longtable}
}

%-----------
\subsection{Simplicial Segregation Measures}

{\renewcommand{\arraystretch}{1.35}
\begin{longtable}{>{\raggedright}p{3.5cm} p{5.6cm}  p{0.8 cm}  p{4.5cm}  }
\caption{Simplicial segregation measures.
	$\vec{T}_q(\sigma)$ = directed triangles at level $q$ containing $\sigma$; $\mathrm{deg}^{\mathrm{tot}}_{q}(\sigma) = \deg^-_{q}(\sigma)+\deg^+_{q}(\sigma)$;
	$f_{\sigma \tau}$ = dim of shared face; $E^{\pm}_{>k}$, $F^{\pm}_{>k}$ = edges/nodes with in/out-$q$-degree $>k$.}
\label{tab:table7}\\
\toprule
\textbf{Measure} & \textbf{Formula} & \textbf{Eq.} & \textbf{Graph Analogue} \\
\midrule % <-- Midrule here

Avg.\ directed $q$-clustering $\vec{\bar{C}}_q(\mathcal{G}_q)$ & {\tiny $\displaystyle\frac{1}{|\mathcal{V}_q|}\sum_\sigma\frac{\vec{T}_q(\sigma)}{\mathrm{deg}^{\mathrm{tot}}_{q}(\sigma)(\mathrm{deg}^{\mathrm{tot}}_{q}(\sigma)-1)-2\deg^\pm_{q}(\sigma)}$} & (36) & Directed clustering\\

\rowcolor{lightgray}
Simplicial $({\bullet})$-$q$-clustering $\vec{C}^{\bullet}_q(\sigma^{(n)})$ & $\displaystyle\sum_{\tau^{(m)}\in\mathrm{st}^{\bullet}_{q}(\sigma)}\frac{2^{1+f_{\sigma\tau}}-1}{2^{n+2m} - 1}$ & (37) & Simplicial clustering\\

Simplicial in-$q$-rich-club $\phi^-_{q,k}(\mathcal{G}_q)$ & $\displaystyle\frac{E^-_{>k}(q)}{F^-_{>k}(q)(F^-_{>k}(q)-1)}$ &
(38) & In-rich-club\\

\rowcolor{lightgray}
Simplicial out-$q$-rich-club $\phi^+_{q,k}(\mathcal{G}_q)$ &
$\displaystyle\frac{E^+_{>k}(q)}{F^+_{>k}(q)(F^+_{>k}(q)-1)}$ &
(39) & Out-rich-club\\

Directed simplicial local $q$-efficiency $\vec{E}^q_{\mathrm{loc}}(\mathcal{G}_q)$ & $\displaystyle\frac{1}{|\mathcal{V}_q|}\sum_{\sigma \in \mathcal{V}_q}\vec{E}^q_{\mathrm{glob}}(\mathcal{G}_q(\sigma))$ & (40) & Local efficiency\\

\bottomrule % <-- Bottomrule here
\end{longtable}
}

%-----------
\subsection{Spectrum-Related Simplicial Measures}

{\renewcommand{\arraystretch}{1.35}
\begin{longtable}{>{\raggedright}p{3.5cm} p{5.6cm}  p{0.8 cm}  p{4.5cm}  }
\caption{Spectrum-related simplicial measures; $\{\varsigma^q_k\}$ = singular values of $\mathcal{H}_q$.}
\label{tab:table8}\\
\toprule
\textbf{Measure} & \textbf{Formula} & \textbf{Eq.} & \textbf{Graph Analogue} \\
\midrule % <-- Midrule here

Simplicial $q$-energy $\varepsilon_q(\mathcal{G}_q)$ &
$\| (\mathcal{H}_q \mathcal{H}_{q}^{T})^{1/2} \|_{*} = \sum_k\varsigma^q_{k}$ &
(41) & Graph energy\\

\rowcolor{lightgray}
Simplicial in-$q$-Katz centrality $K^-_q(\sigma)$ &
$\big[ \big\langle \mathbf{1} \big|(I - \alpha \mathcal{H}_q)^{-1}\big]_\sigma$ & (42) & In-Katz centrality\\

Simplicial out-$q$-Katz centrality $K^+_q(\sigma)$ &
$\big[ \big(I - \alpha \mathcal{H}_q)^{-1} \big| \mathbf{1} \big\rangle \big]_\sigma$ & (43) & Out-Katz centrality\\

\rowcolor{lightgray}
Right $q$-eigenvector centrality $C^q_{e,r}(\sigma)$ &
$(v^q_1)_\sigma$ with $\mathcal{H}_q v^q_1 = \lambda^q_1 v^q_1$ & (44) & Eigenvector centrality\\

Left $q$-eigenvector centrality $C^q_{e,l}(\sigma)$ &
$(v^q_1)_\sigma$ with $\mathcal{H}_{q}^{T} v^{q}_{1} = \lambda^q_1 v^q_1$ & (45) & Left eigenvector centrality\\

\bottomrule % <-- Bottomrule here
\end{longtable}
}

%-----------
\subsection{Similarity and Comparison Measures}

{\renewcommand{\arraystretch}{1.35}
\begin{longtable}{>{\raggedright}p{3.8cm} p{7.7cm}  p{1.5cm}  }
\caption{Simplicial similarity comparison measures.
	$\mathrm{Str}^*_n(\mathcal{X}_2)$ fills with zeros to match dimension of $\mathrm{Str}_n(\mathcal{X}_1)$; $e(\mathcal{X}_1,\mathcal{X}_2)$ = edit distance; $K_J$ and $K_E$ equal $0$ when complexes are identical.}
\label{tab:table9}\\
\toprule
\textbf{Measure} & \textbf{Formula / Definition} & \textbf{Eq.} \\
\midrule % <-- Midrule here

1st Structure vector $\mathrm{Str}_1(\mathcal{X})$ & $(s^1_0,\ldots,s^1_N)$; $s^1_k=$ number of directed $k$-simplices in $\mathcal{X}$ & Def.\ 7.1\\

\rowcolor{lightgray}
2nd Structure vector $\mathrm{Str}_2(\mathcal{X})$ & $(s^2_0,\ldots,s^2_N)$; $s^2_q=$ number of weakly $q$-conn.\ components & Def.\ 7.1\\

3rd Structure vector $\mathrm{Str}_3(\mathcal{X})$ & $(s^3_0,\ldots,s^3_N)$; $s^3_q=$ number of strongly $q$-conn.\ components & Def.\ 7.1\\

\rowcolor{lightgray}
4th Structure vector $\mathrm{Str}_4(\mathcal{X})$ & $(s^4_0,\ldots,s^4_N)$; $s^4_q=1-s^2_q/|\mathcal{V}_q|$ (degree of weak connectedness) & Def.\ 7.1\\

5th Structure vector $\mathrm{Str}_5(\mathcal{X})$ & $(s^5_0,\ldots,s^5_N)$; $s^5_q=1-s^3_q/|\mathcal{V}_q|$ (degree of strong connectedness) & Def.\ 7.1\\

\rowcolor{lightgray}
Normalised $n$-th structure distance $\hat{T}^n_{sd}(\mathcal{X}_1,\mathcal{X}_2)$ &
$\dfrac{\|\mathrm{Str}_n(\mathcal{X}_1) - \mathrm{Str}^*_n(\mathcal{X}_2)\|_2}{\|\mathrm{Str}_n(\mathcal{X}_1)\|_2+\|\mathrm{Str}^*_n(\mathcal{X}_2)\|_2}$ $\in [0,1]$ & (46)\\

Histogram cosine kernel $K_{HC}$ &
$\dfrac{\langle\mathrm{Str}_1(\mathcal{X}_1),\mathrm{Str}^*_1(\mathcal{X}_2)\rangle}{\|\mathrm{Str}_1(\mathcal{X}_1)\|\,\|\mathrm{Str}^*_1(\mathcal{X}_2)\|}$ & (47)\\

\rowcolor{lightgray}
Jaccard kernel $K_J$ &
$\dfrac{|\mathcal{X}_1\cap\mathcal{X}_2|}{|\mathcal{X}_1\cup\mathcal{X}_2|}$ $ \in [0,1]$ & (48)\\

Edit kernel $K_E$ &
$1-\dfrac{2e(\mathcal{X}_1,\mathcal{X}_2)}{|\mathcal{X}_1|+|\mathcal{X}_2|+e(\mathcal{X}_1,\mathcal{X}_2)}$ & (49)\\

\rowcolor{lightgray}
Stratified edit kernel $K_{SE}$ &
$\dfrac{1}{|\mathcal{D}|}\sum_{k\in\mathcal{D}}K_E(\mathcal{X}^k_1,\mathcal{X}^k_2)$ ($\mathcal{D}$ $=$ dimension set) & (50)\\

\bottomrule % <-- Bottomrule here
\end{longtable}
}

%---------------------------------------------------------
%---------------------------------------------------------
\section{Interpretability of the Different Types of Directed Higher-Order Adjacencies}
\label{sec:interpretability}

In this section, we provide a brief discussion on the interpretability of the different types of directed higher-order adjacencies introduced in the text, each with a distinct structural meaning.

\textbf{Lower $(\bullet)$-$q$-adjacency} ($\sigma \sim_{L_{q}}^{\bullet} \tau$) holds when two directed simplices share a common $q$-face, capturing the most direct form of higher-order topological adjacency. It is the higher-order analogue of two communities overlapping at their boundary, with a certain direction of influence. The \emph{strictly lower} variant $\sim_{L_{q^{*}}}^{\bullet}$, however, requires that the highest dimensional face shared by the two directed simplices is \textit{exactly} $q$-dimensional.

\textbf{Upper $(\bullet)$-$p$-adjacency} ($\sigma \sim_{U_{p}}^{\bullet} \tau$) holds when both directed simplices are faces of a common, larger directed $p$-simplex $\Theta^{(p)}$, and the face map indices indicate their relative position. It generalises co-membership in a bigger community, and at $p=1$ it reduces exactly to ordinary arc adjacency between nodes.  Analogously to the lower case, the \emph{strictly upper} variant $\sim_{U_{p^{*}}}^{\bullet}$ requires that the highest dimensional common directed simplex is \textit{exactly} $p$-dimensional.

\textbf{General $(\bullet)$-$q$-adjacency} ($\sigma \sim_{A_{q}}^{\bullet} \tau$) holds when the $q$-face shared by $\sigma$ and $\tau$ is \textit{not} a consequence of both being faces of a common, larger directed simplex. The structural meaning of this restriction is that the two directed simplices are intrinsically adjacent at the $q$-th organizational level and cannot be explained by a bigger community containing both.

\textbf{Maximal $(\bullet)$-$q$-adjacency} ($\sigma \sim_{A_{q^{*}}}^{\bullet} \tau$) retains only the pair of maximal directed simplices that share a face of dimension exactly equal to $q$, discarding sub-simplices that carry redundant adjacency information. This gives the most non-redundant representation of higher-order adjacency at level $q$.

%---------------------------------------------------------
%---------------------------------------------------------
\section{Worked Examples}
\label{sec:examples}

All examples below use the \textit{maximal} variant ($\mathcal{G}_q^A$, $\mathcal{H}_q^A$, $\vec{d}_q^A$,
$\deg^\bullet_{A_{q^*}}$) throughout, however, we will adopt the generic notation, according to the Table~\ref{tab:tablesup1}.

\medskip

%----------
%----------
\subsection{Example 1: Bridge Digraph}
\label{ex:ex1}

%----------
\subsubsection{Digraph definition and DFC}

\begin{definition}[Bridge Digraph]
Let $G_1 = (V, E)$ with
$$
V = \{ 0,1,2,3,4,5 \}, \quad
E = \{ (0,1),(0,2),(1,2),(2,3),(3,4),(3,5),(4,5) \}.
$$
\end{definition}

Figure~\ref{fig:G1} shows the graphical representation of $G_1$: two directed $3$-cliques, $[0,1,2]$ and $[3,4,5]$, connected by the bridge arc $[2,3]$.
\begin{figure}[h!]
\centering
\includegraphics[scale=1.3]{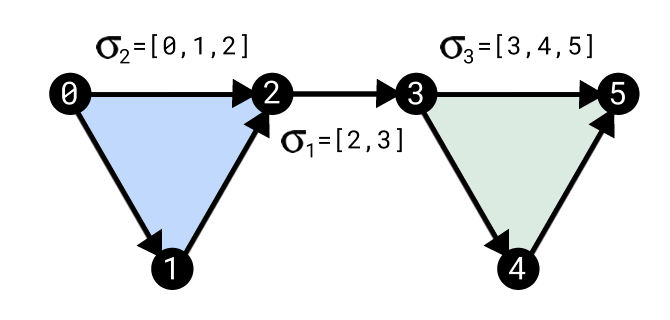}
\caption{Bridge Digraph $G_1$. Bridge arc: $\sigma_{1} = [2,3]$; directed $3$-clique $\sigma_{1} = [0,1,2]$ (blue); directed $3$-clique $\sigma_{3} = [3,4,5]$ (green).}
\label{fig:G1}
\end{figure}

Directed flag complex associated with $G_1$:

\begin{center}
\begin{tabular}{clc}
\toprule
\textbf{Dimension} & \textbf{Simplices} & \textbf{Count}\\
\midrule
0 & $[0],[1],[2],[3],[4],[5]$ & 6\\
1 & $[0,1],[0,2],[1,2],[2,3],[3,4],[3,5],[4,5]$ & 7\\
2 & $[0,1,2],\;[3,4,5]$ & 2\\
\bottomrule
\end{tabular}
\end{center}

\smallskip

\begin{itemize}
\item  \textbf{Maximal simplices:} $\mathrm{dFl}^*(G_1) = \{ \sigma_{1} = [2,3], \sigma_{2} = [0,1,2], \sigma_{3} = [3,4,5]\}$.

\item \textbf{1st Structure vector}: $Str_{1} = (6,7,2)$.
\end{itemize}

Note that arc $\sigma_{1} = [2,3]$ is maximal because it is not contained in any $2$-simplex (no $3$-clique spans $\{2,3,x\}$ for any $x$). The two directed triangles $[0,1,2]$ and $[3,4,5]$ are not connected through any shared directed face at $q \ge 1$.

%-------------
\subsubsection{Level $q = 0$}

Since none of the maximal simplices share faces of dimensions greater than $0$, to verify maximal $(\bullet)$-$q$-adjacencies for each pair $(\sigma_i, \sigma_j)$, $\sigma_i \sim^+_{A_{0^{*}}} \sigma_j$, it suffices to verify the lower $(\bullet)$-$q$-adjacencies:

\begin{itemize}[noitemsep]
\item $\sigma_2 = [0,1,2] \sim^{\pm}_{A_{0^*}}\sigma_1 = [2,3]$: $\hat{d}_1([0,1,2]) = [0,2]$ and $\hat{d}_1([2,3]) = [2]$; $2\in[0,2]$, $i = j = 1$. 

\item $\sigma_1=[2,3]\sim^{\pm}_{A_{0^*}}\sigma_2 = [0,1,2]$: $\hat{d}_1([2,3])=[2]$ and $\hat{d}_1([0,1,2])=[0,2]$; $2\in[0,2]$, $i= j = 1$. 

\item $\sigma_1 = [2,3]\sim^{+}_{A_{0^*}}\sigma_3 = [3,4,5]$: $\hat{d}_0([2,3]) = [3]$ and $\hat{d}_1([3,4,5]) = [3,5]$; $3\in[3,5]$, $i=0\le j=1$. 

\end{itemize}

The vertex set of $\mathcal{G}_{0}$ is $\mathcal{V}_{0} = \{ \sigma_1, \sigma_2, \sigma_3 \}$, and the arc set is given by the previous relations:
$$
\mathcal{E}_{0} = \{ (\sigma_1, \sigma_2), (\sigma_2, \sigma_1), (\sigma_1, \sigma_3) \}.
$$

The $0$-adjacency matrix of $\mathcal{G}_0$ (rows/cols: $\sigma_1,\sigma_2,\sigma_3$) is:
$$
\mathcal{H}_{0} =  
\bbordermatrix{
& \sigma_1 & \sigma_2 & \sigma_3 \cr
\sigma_1   & 0  & 1 & 1\cr 
\sigma_2   & 1  & 0 & 0\cr 
\sigma_3   & 0  & 0 & 0\cr 
}.
$$

%-------------
\subsubsection{Level $q \ge 1$}

\noindent \textbf{Level $q = 1$}: Strictly lower $(+)$-$1$-adjacency requires sharing a $1$-simplex but not a $2$-simplex. The only $1$-simplex candidate between the three maximal simplices is $\sigma_1 = [2,3]$, but it has no $1$-simplex face in common with $\sigma_2$ or $\sigma_3$. Therefore $\mathcal{H}_1 = \mathbf{0}$.

\medskip

\noindent \textbf{Level $q \ge 2$}: Since $\sigma_2$ and $\sigma_3$ share no vertices, we have $\mathcal{H}_q = \mathbf{0}$, for all $q \ge 2$.

%-------------
\subsubsection{Simplicial Measures for $\mathcal{G}_{0}$}

\begin{table}[h!]
\centering
\caption{Selected local directed simplicial measures for $\mathcal{G}_{0}$.}
\begin{tabular}{l c c c}
\toprule
\textbf{Local measure} & $\sigma_1=[2,3]$ & $\sigma_2=[0,1,2]$ & $\sigma_3=[3,4,5]$\\

\midrule
$0$-Communicability to $\sigma_1$ $\mathrm{CM}_{0}(\cdot, \sigma_1)$ & 1.5431 & 1.1752 & 0.00\\
$0$-Communicability to $\sigma_2$ $\mathrm{CM}_{0}(\cdot, \sigma_2)$ & 1.1752 & 1.5431 & 0.00\\
$0$-Communicability to $\sigma_3$ $\mathrm{CM}_{0}(\cdot, \sigma_3)$ & 1.1752 & 0.5431 & 1.00\\

In-$0$-degree centrality $C^-_{\mathrm{deg}_0}$ & 0.50 & 0.50 & 0.50\\

Out-$0$-degree centrality $C^+_{\mathrm{deg}_0}$ & 1.00 & 0.50 & 0.00\\

$0$-Harmonic centrality $\vec{HC}_0$ & 2.00 & 1.50 & 0.00\\

$0$-Closeness centrality $\vec{Cl}_0$ & 0.50 & 0.33 & 0.00\\
$0$-Betweenness centrality $\vec{B}_0$ (non-norm.) & 1.00 & 0.00 & 0.00\\

Local $0$-reaching centrality $C_{R,0}$ & 1.00 & 1.00 & 0.00\\

\bottomrule
\end{tabular}
\label{tab:table11}
\end{table}

\begin{table}[h!]
\centering
\caption{Selected global directed simplicial measures for $\mathcal{G}_{0}$.}
\begin{tabular}{l c}
\toprule
\textbf{Global measure} & \textbf{Value}\\
\midrule

Avg.\ shortest $0$-walk length $\vec{\bar{L}}_0$ & 0.8333\\

Global $0$-efficiency $\vec{E}^0_{\mathrm{glob}}$ & 0.5833\\

$0$-Diameter & 2\\

$0$-Radius & 1\\

$0$-Returnability $K_{r,0}$ (non-norm.) & 1.0862\\

$0$-Energy $\varepsilon_0$ ($=1+\sqrt{2}$) & 2.4142\\

Singular values of $\mathcal{H}_0$ & $\sqrt{2}, 1, 0$\\

Global $0$-reaching centrality $\mathrm{GRC}_0$ & 0.5000\\

Weakly $0$-connected components & 1\\

Strongly $0$-connected components & 2\\

\bottomrule
\end{tabular}
\label{tab:table12}
\end{table}

Based on the results of the simplicial measures (Table~\ref{tab:table11} and Table~\ref{tab:table12}), we conclude that the bridge arc $\sigma_1=[2,3]$ is the most central simplex: it has the highest harmonic centrality ($2.0$), the only non-zero betweenness ($1.0$), and the highest out-degree ($1.0$). The only directed path from $\sigma_2$ to $\sigma_3$ passes through $\sigma_1$, confirming its structural role as a topological bottleneck.

%----------
%----------
\subsection{Example 2: Directed $4$-Clique}

%----------
\subsubsection{Digraph definition and DFC}

\begin{definition}
The \emph{directed $n$-clique} on $n$ vertices has arc set $E = \{(i,j):i < j\}$. For $n = 4$, we have $G_{2} = \{ V, E\}$ with
$$
V = \{ 0,1,2,3 \},\quad
E = \{ (0,1),(0,2),(0,3),(1,2),(1,3),(2,3) \}.
$$
\end{definition}

Figure~\ref{fig:G2} shows the graphical representation of $G_2$: vertex $0$ is the unique source (in-degree $0$), and vertex $3$ is the unique sink (in-degree $3$).

\begin{figure}[h!]
\centering
\includegraphics[scale=1.3]{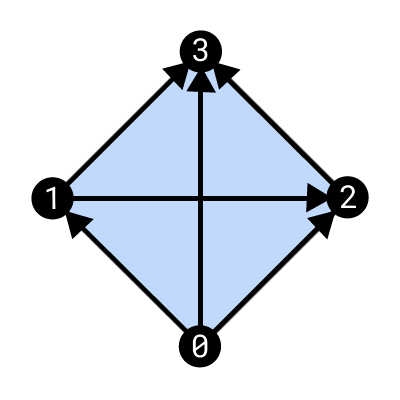}
\caption{Directed $4$-clique $G_2$.  }
\label{fig:G2}
\end{figure}

Directed flag complex associated with $G_2$:

\begin{table}[h!]
\centering
\begin{tabular}{clc}
\toprule
\textbf{Dimension} & \textbf{Simplices} & \textbf{Count}\\
\midrule
0 & $[0],[1],[2],[3]$ & 4\\
1 & $[0,1],[0,2],[0,3],[1,2],[1,3],[2,3]$ & 6\\
2 & $[0,1,2],[0,1,3],[0,2,3],[1,2,3]$ & 4\\
3 & $[0,1,2,3]$ & 1\\
\bottomrule
\end{tabular}
\end{table}

\medskip

\begin{itemize}
\item \textbf{Maximal simplexes:} $\mathrm{dFl}^{*}(G_2) = \{[0,1,2,3]\}$ (one single directed 4-clique).

\item \textbf{1st Structure vector}: $\mathrm{Str}_1 = (4,6,4,1)$.
\end{itemize}

%-------------
\subsubsection{Level $q \ge 0$}

For every $q = 0,1,2,3$, we have $\mathrm{dFl}^*_q(G_2) = \{ [0,1,2,3] \}$, therefore $\mathcal{G}_q$ has exactly \emph{one vertex and no arcs}:
$$
\mathcal{H}_q = [0], \qquad \forall q \in \{0,1,2,3\}.
$$

%-------------
\subsubsection{Simplicial Measures for $\mathcal{G}_{q}$}

As we just discussed, the entire vertex set of $G_2$ forms a single maximal simplex $[0,1,2,3]$, therefore, every $q$-digraph  has a minimal higher-order interaction structure, which is clearly captured by the simplicial measures (Table~\ref{tab:table13}).

\begin{table}[h!]
\centering
\caption{Selected directed simplicial measures for $\mathcal{G}_{q}$, $q=0,1,2,3$.}
\label{tab:table13}
\begin{tabular}{lcccc}
\toprule
\textbf{Measure} & $q=0$ & $q=1$ & $q=2$ & $q=3$\\
\midrule
Vertex set cardinality $|\mathcal{V}_q|$ & 1 & 1 & 1 & 1\\
Arc set cardinality  $|\mathcal{E}_q|$ & 0 & 0 & 0 & 0\\
All centralities & --- & --- & --- & ---\\
$q$-Energy $\varepsilon_q$ & 0 & 0 & 0 & 0\\
$q$-Returnability $K_{r,q}$  & 0 & 0 & 0 & 0\\
\bottomrule
\end{tabular}
\end{table}

%----------
%----------
\subsection{Example 3: Directed 4-Cycle}

%----------
\subsubsection{Digraph definition and DFC}

\begin{definition}
The \emph{directed $4$-cycle} is the digraph $G_3 = (V, E)$ with
$$
V = \{0,1,2,3\},\quad
E = \{ (0,1),(1,2),(2,3),(3,0) \}.
$$
\end{definition}

Figure~\ref{fig:G3} shows the graphical representation of $G_3$: no three consecutive vertices form a directed triangle (cycles prevent it).

\begin{figure}[h!]
\centering
\includegraphics[scale=1.3]{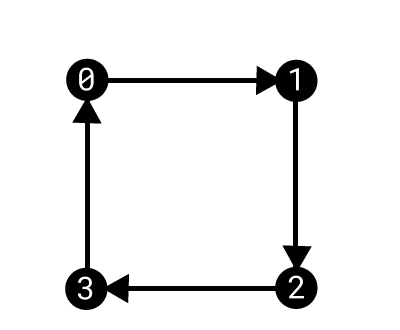}
\caption{Bridge Digraph $G_2$.}
\label{fig:G3}
\end{figure}

Directed flag complex associated with $G_3$:

\begin{table}[h!]
\centering
\begin{tabular}{clc}
\toprule
\textbf{Dim} & \textbf{Simplices} & \textbf{Count}\\
\midrule
0 & $[0],[1],[2],[3]$ & 4\\
1 & $[0,1],[1,2],[2,3],[3,0]$ & 4\\
\bottomrule
\end{tabular}
\end{table}

\medskip

\begin{itemize}
\item \textbf{Maximal simplices:} $\mathrm{dFl}^*(G_2) = \{ \sigma_1 = [0,1], \sigma_2 = [1,2], \sigma_3 = [2,3], \sigma_4 = [3,0] \}$ (four arcs).

\item \textbf{1st Structure vector}: $\mathrm{Str}_1 = (4, 4)$.
\end{itemize}

%-------------
\subsubsection{Level $q = 0$}

Similarly to Example~\ref{ex:ex1}, since none of the maximal simplices share faces of dimensions greater than $0$, to verify maximal $(\bullet)$-$q$-adjacencies for each pair $(\sigma_i, \sigma_j)$, $\sigma_i \sim^+_{A_{0^{*}}} \sigma_j$, it suffices to verify the lower $(\bullet)$-$q$-adjacencies:

\begin{itemize}[noitemsep]
\item $\hat{d}_0(\sigma_1) = [1] = \hat{d}_1(\sigma_2)$; $i = 0 \le j = 1$ $\rightarrow$ $\sigma_1  \sim^{+}_{A_{0^*}} \sigma_2$. 

\item $\hat{d}_0(\sigma_2) = [2] = \hat{d}_1(\sigma_3)$; $i = 0 \le j = 1$ $\rightarrow$ $\sigma_2  \sim^{+}_{A_{0^*}} \sigma_3$.

\item $\hat{d}_0(\sigma_3) = [3] = \hat{d}_1(\sigma_4)$; $i = 0 \le j = 1$ $\rightarrow$ $\sigma_3  \sim^{+}_{A_{0^*}} \sigma_4$.

\item $\hat{d}_0(\sigma_4) = [0] = \hat{d}_1(\sigma_1)$; $i = 0 \le j = 1$ $\rightarrow$ $\sigma_4  \sim^{+}_{A_{0^*}} \sigma_1$.

\end{itemize}

Therefore, the four arcs form a directed $4$-cycle on the $\mathcal{G}_0$ vertex set, $\mathcal{V}_0 = \{  \sigma_1, \sigma_2, \sigma_3, \sigma_4 \}$,
$$
\mathcal{E}_0 = \{ (\sigma_1, \sigma_{2}), (\sigma_2, \sigma_{3}), (\sigma_3, \sigma_{4}), (\sigma_4, \sigma_{1}) \},
$$

\noindent and the $0$-adjacency matrix of $\mathcal{G}_0$ is
$$
\mathcal{H}_{0} =  
\bbordermatrix{
& \sigma_1 & \sigma_2 & \sigma_3 & \sigma_4 \cr
\sigma_1   & 0  & 1 & 0 & 0\cr 
\sigma_2   & 0  & 0 & 1 & 0\cr 
\sigma_3   & 0  & 0 & 0 & 1\cr
\sigma_4   & 1  & 0 & 0 & 0\cr 
}.
$$

%-------------
\subsubsection{Level $q \ge 1$}

No two distinct arcs share a $1$-simplex, therefore, $\mathcal{H}_q = \mathbf{0}$, for all $q \ge 1$.

%-------------
\subsubsection{Simplicial Measures for $\mathcal{G}_{0}$}

\begin{table}[h!]
\centering
\caption{Selected directed simplicial measures for $\mathcal{G}_{0}$.}
\label{tab:tablesup4}
\begin{tabular}{lccl}
\toprule
\textbf{Measure} & \textbf{Value} \\
\midrule

Vertex set cardinality $|\mathcal{V}_0|$ & 4\\

Arc set cardinality $|\mathcal{E}_0|$ & 4 \\

Avg. shortest directed $0$-walk length $\vec{\bar{L}}_0$ & $2.0000$  \\

Global $0$-efficiency $\vec{E}^0_\text{glob}$  & $0.6111$  \\

$0$-Returnability $K_{r,0}$ (non-norm.) &  $0.1668$ \\

$0$-Energy $\varepsilon_0$ & $4.0000$\\

Eigenvalues of $\mathcal{H}_0$ &  $1,-1,i,-i$ \\

In-$0$-degree centrality $C^{-}_{\mathrm{deg}_0}(\sigma)$ & $0.333$ (all vertices)\\

Out-$0$-degree centrality $C^{+}_{\mathrm{deg}_0}(\sigma)$ &  $0.333$ (all vertices) \\

$0$-Harmonic centrality $\vec{HC}_0(\sigma)$ & $1.833$ (all vertices)\\

$0$-Betweenness centrality $\vec{B}_0(\sigma)$ &  $3.0$ (all vertices)\\

Local $0$-reaching centrality $C_{R,0}(\sigma)$ & $1.0$ (all vertices) \\

Global $0$-reaching centrality $\mathrm{GRC}_0$ & $0.0000$\\
\bottomrule
\end{tabular}
\label{tab:table14}
\end{table}

We can see from the results of the simplicial measures (Table~\ref{tab:table14}) that all vertices are structurally equivalent, all centralities are identical, and the global reaching centrality is exactly zero (no hierarchy whatsoever). very low returnability ($K_{r,0}\approx 0.17$) because closed walks require four steps.

%----------
%----------
\subsection{Example 4: Two Directed $3$-Cliques Sharing a $1$-Face}

%----------
\subsubsection{Digraph definition and DFC}

\begin{definition}
Let $G_4 = (V,E)$ be the digraph with
$$
V = \{ 0,1,2,3 \},\quad
E = \{ (0,1), (0,3), (1,2), (1,3), (2,3) \},
$$

\noindent that is, the digraph formed by two directed $3$-cliques $\sigma = [0,1,3]$ and $\tau = [1,2,3]$.
\end{definition}

Figure~\ref{fig:G4} shows the graphical representation of $G_4$: the directed $3$-cliques $\sigma = [0,1,3]$ and $\tau = [1,2,3]$ share the arc $[1,3]$. Vertex $0$ is the unique source of $\sigma$ and vertex $2$ is the unique source of $\tau$.

\begin{figure}[h!]
\centering
\includegraphics[scale=1.3]{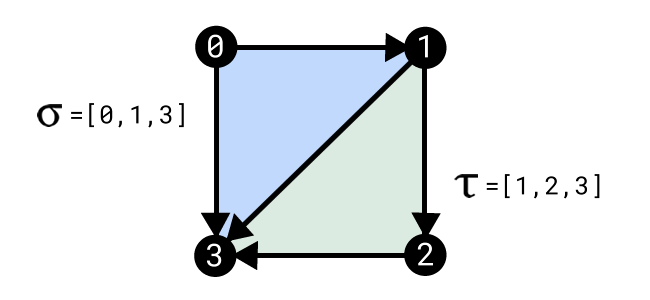}
\caption{Two directed $3$-cliques $\sigma=[0,1,3]$ (blue) and $\tau=[1,2,3]$ (green) sharing the directed $1$-face $[1,3]$.}
\label{fig:G4}
\end{figure}

Directed flag complex associated with $G_4$:

\begin{center}
\begin{tabular}{clc}
\toprule
\textbf{Dim} & \textbf{Simplices} & \textbf{Count}\\
\midrule
0 & $[0],[1],[2],[3]$ & 4\\
1 & $[0,1],[0,3],[1,2],[1,3],[2,3]$ & 5\\
2 & $[0,1,3],\;[1,2,3]$ & 2\\
\bottomrule
\end{tabular}
\end{center}

\medskip

\begin{itemize}
\item \textbf{Maximal simplices:} $\mathrm{dFl}^{*}(G_4) = \{\sigma = [0,1,3], \tau = [1,2,3] \}$.

\item \textbf{1st Strucutre vector}: $\mathrm{Str}_1 = (4,5,2)$.
\end{itemize}

Note that both $\sigma$ and $\tau$ have dimension $2$ and neither is a face of any higher-dimensional simplex (no directed $4$-clique exists).  The shared arc $[1,3]$ is the unique $1$-simplex that belongs to both simplices.

%-------------
\subsubsection{Level $q = 0$}

First, let us verify the lower $q$-adjacencies for $q=0,1$:

\begin{itemize}[noitemsep]
\item $\tau \sim^{\pm}_{0} \sigma: \hat{d}_0(\tau) = [2,3] \supseteq [3] \subseteq \hat{d}_0(\sigma)=[1,3]; i = j = 0$.

\item $\tau \sim^{+}_{0} \sigma: \hat{d}_0(\tau) = [2,3] \supseteq [3] \subseteq \hat{d}_1(\sigma)=[0,3]; 0 = i \le j = 1$.

\item $\tau \sim^{\pm}_{0} \sigma: \hat{d}_1(\tau) = [1,3] \supseteq [3] \subseteq \hat{d}_1(\sigma)=[0,3]; i = j = 1$.

\item $\tau \sim^{+}_{0} \sigma: \hat{d}_1(\tau)=[1,3] \supseteq [1] \subseteq \hat{d}_2(\sigma)=[0,1]; 1= i \le j = 2$.

\item $\tau \sim^{\pm}_{0} \sigma: \hat{d}_2(\tau)=[1,2] \supseteq [1] \subseteq \hat{d}_2(\sigma)=[0,1]; i = j = 2$.

\item $\sigma \sim^{+}_{1} \tau: \hat{d}_0(\sigma) = [1,3] = \hat{d}_1(\tau); 0 = i \le j = 1$.

\end{itemize}

Thus the conditions $\sigma \sim^{+}_{0} \tau$ and $\sigma \sim^{+}_{1} \tau$  hold. However, strict lower $(\bullet)$-$0$-adjacencies also require that \emph{neither} simplex is $(+)$-$1$-near \emph{nor} $(-)$-$1$-near; thus, the strict-lower condition fails for $q=0$, and consequently the maximal $0$-adjacency matrix is
$$
\mathcal{H}_0 = \mathbf{0} \quad (\text{trivial 0-digraph}).
$$

Whenever two maximal $k$-simplices share a $(k-1)$-face, they are mutually $(k-1)$-adjacent in the face-map sense, which blocks all strictly lower $q$-adjacency for $q <k-1$.  Here $k=2$ and the shared face has dimension 1, so $q = 0$ is blocked.  The first non-trivial level is $q= k-1 =1$.

%-------------
\subsubsection{Level $q = 1$}

As we showed previously, the condition $\sigma \sim^{+}_{1} \tau$  holds. In addition, $\sigma$ and $\tau$ share no common $2$-face, since $\sigma \cap \tau =  \{1,3\}$.  Therefore:
$$
\sigma \sim^{+}_{A_{1^{*}}} \tau.
$$

We conclude that $\mathcal{G}_1$ is a single arc $\sigma \to \tau$, and then its $1$-adjacency matrix is
$$
\mathcal{H}_{0} =  
\bbordermatrix{
& \sigma & \tau\cr
\sigma   & 0  & 1\cr 
\tau   & 0  & 0\cr 
}.
$$

%-------------
\subsubsection{Level $q \ge 2$}

Both $\sigma$ and $\tau$ have dimension $2$, so they belong to $\mathrm{dFl}^*_2(G_4)$.  For $(+)$-$2$-adjacency we would need
$|\hat{d}_i(\sigma)\cap\hat{d}_j(\tau)|\ge3$, but every face map of a $2$-simplex produces a $1$-simplex, so the maximum possible intersection is $2$.  Hence $\mathcal{H}_q = \mathbf{0}$, for all $q \ge 2$.

%-------------
\subsubsection{Simplicial Measures for $\mathcal{G}_{1}$}

All non-trivial structure is concentrated at $q=1$: the only meaningful $q$-digraph is $\mathcal{G}_1$, which is a single arc. Tables~\ref{tab:sup15} and \ref{tab:sup16} show the results of some simplicial measures for $\mathcal{G}_{1}$; Table~\ref{tab:sup17} shows the results for $\mathcal{G}_{q}$, $q=0,1,2$.

\begin{table}[h!]
\centering
\caption{Selected local simplicial measures for $\mathcal{G}_{1}$.}
\label{tab:sup15}
\begin{tabular}{l c c }
\toprule
\textbf{Measure} & $\sigma=[0,1,3]$ & $\tau=[1,2,3]$\\
\midrule

In-$1$-degree centrality $C^-_{\mathrm{deg}_1}$  & $0.00$ & $1.00$\\

Out-$1$-degree centrality $C^+_{\mathrm{deg}_1}$  & $1.00$ & $0.00$\\

$1$-Harmonic centrality $\vec{HC}_1$  & $1.00$ & $0.00$\\

$1$-Closeness $\vec{Cl}_1$  & $1.00$ & $0.00$\\

$1$-Betweenness $\vec{B}_1$  & $0.00$ & $0.00$\\

$1$-Reaching centrality $C_{R,1}$  & $1.00$ & $0.00$\\

$1$-Communicability $\mathrm{CM}_1(\cdot,\sigma)$  & $1.00$ & $0.00$\\

$1$-Communicability $\mathrm{CM}_1(\cdot,\tau)$  & $1.00$ & $1.00$\\
\bottomrule
\end{tabular}
\end{table}

\begin{table}[h!]
\centering
\caption{Selected global simplicial measures for $\mathcal{G}_{1}$.}
\label{tab:sup16}
\begin{tabular}{lcl}
\toprule
\textbf{Global measure} & \textbf{Value} \\
\midrule

Avg. shortest $1$-walk length $\vec{\bar{L}}_1$ & $0.50$ \\

Global $1$-efficiency $\vec{E}^1_\mathrm{glob}$ & $0.50$ \\

Diameter $\mathrm{diam}_1$ & $1$\\

Radius $\mathrm{rad}_1$ & $1$  \\

$1$-Returnability $K_{r,1}$ & $0.00$ \\

$1$-Energy $\varepsilon_1$ & $1.00$ \\

Singular values & $\{1,\,0\}$ \\

Eigenvalues & $\{0,\,0\}$ \\

Global $1$-reaching centrality $\mathrm{GRC}_1$ & $1.00$ \\

\bottomrule
\end{tabular}
\end{table}

\begin{table}[h!]
\centering
\caption{Simplicial measure summary across all $q$ levels.}
\label{tab:sup17}
\begin{tabular}{l ccc}
\toprule
\textbf{Measure} & $q=0$ & $q=1$ & $q=2$\\
\midrule
Vertex set cardinality $|\mathcal{V}_q|$ & 2 & 2 & 2\\

Arc set cardinality $|\mathcal{E}_q|$ & 0 & 1 & 0\\

Avg. shortest $q$-walk length $\bar{L}_q$ & --- & 0.50 & ---\\

Global $q$-efficiency $\vec{E}^q_\mathrm{glob}$ & --- & 0.50 & ---\\

$q$-Returnability $K_{r,q}$ & 0 & 0.00 & 0\\

$q$-Energy $\varepsilon_q$ & 0 & 1.00 & 0\\

Global $q$-reaching centrality $\mathrm{GRC}_q$ & --- & 1.00 & ---\\

\bottomrule
\end{tabular}
\end{table}

%---------------------------------------------------------
%---------------------------------------------------------
\section{Simulation Studies: Erdös-Rényi Random Digraphs}
\label{sec:simulations}

%-----------
\subsection{Methodology}

We generated small random Erdös--Rényi digraphs $G(n,p)$ with $n = 10$ nodes and arc probability $p \in \{ 0.10, 0.20, 0.30, 0.40, 0.50, 0.60, 0.70, 0.80 \}$. We performed $30$ trials per probability $p$. For each digraph, we enumerated directed cliques up to dimension $d_{\max}$ (4 for $p\le0.4$, and 3 for $p \ge 0.5$) and extracted the maximal directed simplices, and we built the maximal $q$-digraphs $\mathcal{G}_q=(\mathcal{V}_q,\mathcal{E}_q)$ for $q=0,1,2$ (truncated to $60$ nodes if $|\mathcal{V}_q| > 60$). Finally, we computed the following simplicial measures:

\begin{itemize}[noitemsep]
\item Mean vertex set cardinality ($|\mathcal{V}_q|$);
\item Mean arc set cardinality ($|\mathcal{E}_q|$);
\item Mean global $q$-efficiency ($E^{q}_{\mathrm{glob}}$);
\item Mean $q$-energy ($\varepsilon_q$);
\item Mean global $q$-reaching centrality ($\mathrm{GRC}_q$);
\item Mean maximum $q$-harmonic centrality ($\widehat{HC}_q$).
\end{itemize}

%-----------
\subsection{Simulation Results}

All tables below (Tables~\ref{tab:q0-dist}, \ref{tab:q1}, and \ref{tab:q2}) report the mean value of the measures with their respective standard deviations in parentheses.

%--------
%\subsubsection{Results for $q = 0$}

\begin{table}[h!]
\centering
\small
\caption{Mean and standard deviation (parentheses) for the selected simplicial measures at $q=0$.}
\label{tab:q0-dist}
\setlength\tabcolsep{4pt}
\begin{tabular}{l c c c c c c c c}
\toprule
& \multicolumn{8}{c}{\textbf{Probability $p$}}\\
\cmidrule{2-9}
\textbf{Measure} & \textbf{0.10} & \textbf{0.20} & \textbf{0.30} & \textbf{0.40}
& \textbf{0.50} & \textbf{0.60} & \textbf{0.70} & \textbf{0.80}\\
\midrule

$|\mathcal{V}_0|$ &
\musd{9.3}{1.5} & \musd{12.5}{1.8} & \musd{15.8}{3.2} & \musd{21.9}{8.3} &
\musd{51.9}{11.9} & \musd{60.0}{0.0} & \musd{60.0}{0.0} & \musd{60.0}{0.0}\\

$|\mathcal{E}_0|$ &
\musd{15.8}{9.2} & \musd{46.4}{12.8} & \musd{92.0}{39.7} & \musd{166}{102} &
\musd{672}{260} & \musd{684}{249} & \musd{538}{212} & \musd{379}{170}\\

$E^0_{\mathrm{glob}}$ &
\musd{0.306}{0.136} & \musd{0.600}{0.064} & \musd{0.675}{0.037} & \musd{0.655}{0.043} &
\musd{0.604}{0.048} & \musd{0.531}{0.119} & \musd{0.444}{0.179} & \musd{0.349}{0.157}\\

$\varepsilon_0$ &
\musd{8.96}{4.28} & \musd{18.82}{3.80} & \musd{26.27}{6.51} & \musd{34.41}{12.28} &
\musd{89.66}{26.38} & \musd{96.26}{23.68} & \musd{86.46}{18.62} & \musd{73.79}{19.44}\\

$\mathrm{GRC}_0$ &
\musd{0.249}{0.124} & \musd{0.042}{0.076} & \musd{0.000}{0.000} & \musd{0.000}{0.000} &
\musd{0.004}{0.014} & \musd{0.018}{0.045} & \musd{0.020}{0.043} & \musd{0.072}{0.069}\\

$\widehat{HC}_0$ &
\musd{4.39}{1.81} & \musd{8.36}{1.31} & \musd{11.74}{2.52} & \musd{16.64}{6.67} &
\musd{39.43}{9.73} & \musd{40.90}{7.38} & \musd{34.28}{13.00} & \musd{29.42}{11.15}\\

\bottomrule
\end{tabular}
\end{table}

%--------
%\subsubsection{Results for $q = 1$}

\begin{table}[h!]
\centering
\small
\caption{Mean and standard deviation (parentheses) for the selected simplicial measures at $q = 1$.}
\label{tab:q1}
\setlength\tabcolsep{4pt}
\begin{tabular}{l c c c c c c c c}
\toprule
& \multicolumn{8}{c}{\textbf{Probability $p$}}\\
\cmidrule{2-9}
\textbf{Measure} & \textbf{0.10} & \textbf{0.20} & \textbf{0.30} & \textbf{0.40}
& \textbf{0.50} & \textbf{0.60} & \textbf{0.70} & \textbf{0.80}\\
\midrule

$|\mathcal{V}_1|$ &
\musd{7.7}{2.3} & \musd{12.3}{1.9} & \musd{15.8}{3.2} & \musd{21.9}{8.3} &
\musd{51.9}{11.9} & \musd{60.0}{0.0} & \musd{60.0}{0.0} & \musd{60.0}{0.0}\\

$|\mathcal{E}_1|$ &
\musd{0.20}{0.54} & \musd{8.47}{9.81} & \musd{40.83}{30.04} & \musd{127}{110} &
\musd{1021}{452} & \musd{1419}{193} & \musd{1514}{198} & \musd{1646}{148}\\

$E^1_{\mathrm{glob}}$ &
\musd{0.009}{0.031} & \musd{0.086}{0.099} & \musd{0.297}{0.154} & \musd{0.505}{0.100} &
\musd{0.665}{0.045} & \musd{0.698}{0.029} & \musd{0.711}{0.030} & \musd{0.731}{0.023}\\

$\varepsilon_1$ &
\musd{0.20}{0.54} & \musd{4.68}{4.53} & \musd{16.57}{8.41} & \musd{31.74}{14.67} &
\musd{115.08}{38.19} & \musd{144.19}{19.88} & \musd{147.54}{19.62} & \musd{157.36}{18.03}\\

$\mathrm{GRC}_1$ &
\musd{0.030}{0.097} & \musd{0.178}{0.124} & \musd{0.236}{0.105} & \musd{0.078}{0.097} &
\musd{0.004}{0.013} & \musd{0.000}{0.000} & \musd{0.000}{0.000} & \musd{0.000}{0.000}\\

$\widehat{HC}_1$ &
\musd{0.13}{0.34} & \musd{2.50}{2.12} & \musd{7.38}{3.21} & \musd{14.49}{6.93} &
\musd{39.63}{10.11} & \musd{46.91}{1.55} & \musd{46.52}{1.09} & \musd{46.96}{1.42}\\

\bottomrule
\end{tabular}
\end{table}

%--------
%\subsubsection{Results for $q = 2$}

\begin{table}[h!]
\centering
\small
\caption{Mean and standard deviation (parentheses) for the selected simplicial measures at $q = 2$.}
\label{tab:q2}
\setlength\tabcolsep{4pt}
\begin{tabular}{l c c c c c c c c}
\toprule
& \multicolumn{8}{c}{\textbf{Probability $p$}}\\
\cmidrule{2-9}
\textbf{Measure} & \textbf{0.10} & \textbf{0.20} & \textbf{0.30} & \textbf{0.40}
& \textbf{0.50} & \textbf{0.60} & \textbf{0.70} & \textbf{0.80}\\
\midrule

$|\mathcal{V}_2|$ &
\musd{0.6}{0.7} & \musd{4.7}{3.2} & \musd{12.3}{4.6} & \musd{21.0}{8.5} &
\musd{51.8}{12.0} & \musd{60.0}{0.0} & \musd{60.0}{0.0} & \musd{60.0}{0.0}\\

$|\mathcal{E}_2|$ &
\musd{0.0}{0.0} & \musd{0.1}{0.5} & \musd{5.8}{12.2} & \musd{74.8}{131} &
\musd{606}{294} & \musd{871}{200} & \musd{976}{175} & \musd{1071}{164}\\

$E^2_{\mathrm{glob}}$ &
\musd{0.000}{0.000} & \musd{0.005}{0.020} & \musd{0.036}{0.064} & \musd{0.173}{0.130} &
\musd{0.462}{0.087} & \musd{0.542}{0.073} & \musd{0.590}{0.049} & \musd{0.626}{0.038}\\

$\varepsilon_2$ &
\musd{0.0}{0.0} & \musd{0.13}{0.50} & \musd{2.77}{4.36} & \musd{17.90}{16.02} &
\musd{110.76}{38.27} & \musd{147.72}{14.98} & \musd{162.30}{11.72} & \musd{172.91}{7.32}\\

$\mathrm{GRC}_2$ &
\musd{0.000}{0.000} & \musd{0.011}{0.041} & \musd{0.094}{0.135} & \musd{0.240}{0.087} &
\musd{0.074}{0.081} & \musd{0.027}{0.046} & \musd{0.006}{0.021} & \musd{0.005}{0.028}\\

$\widehat{HC}_2$ &
\musd{0.00}{0.00} & \musd{0.07}{0.25} & \musd{1.44}{2.17} & \musd{8.11}{6.79} & \musd{32.82}{9.78} & \musd{40.83}{3.62} & \musd{41.83}{2.51} & \musd{43.92}{2.65}\\

\bottomrule
\end{tabular}
\end{table}

\begin{figure}[h!]
\centering
\includegraphics[scale=1.5]{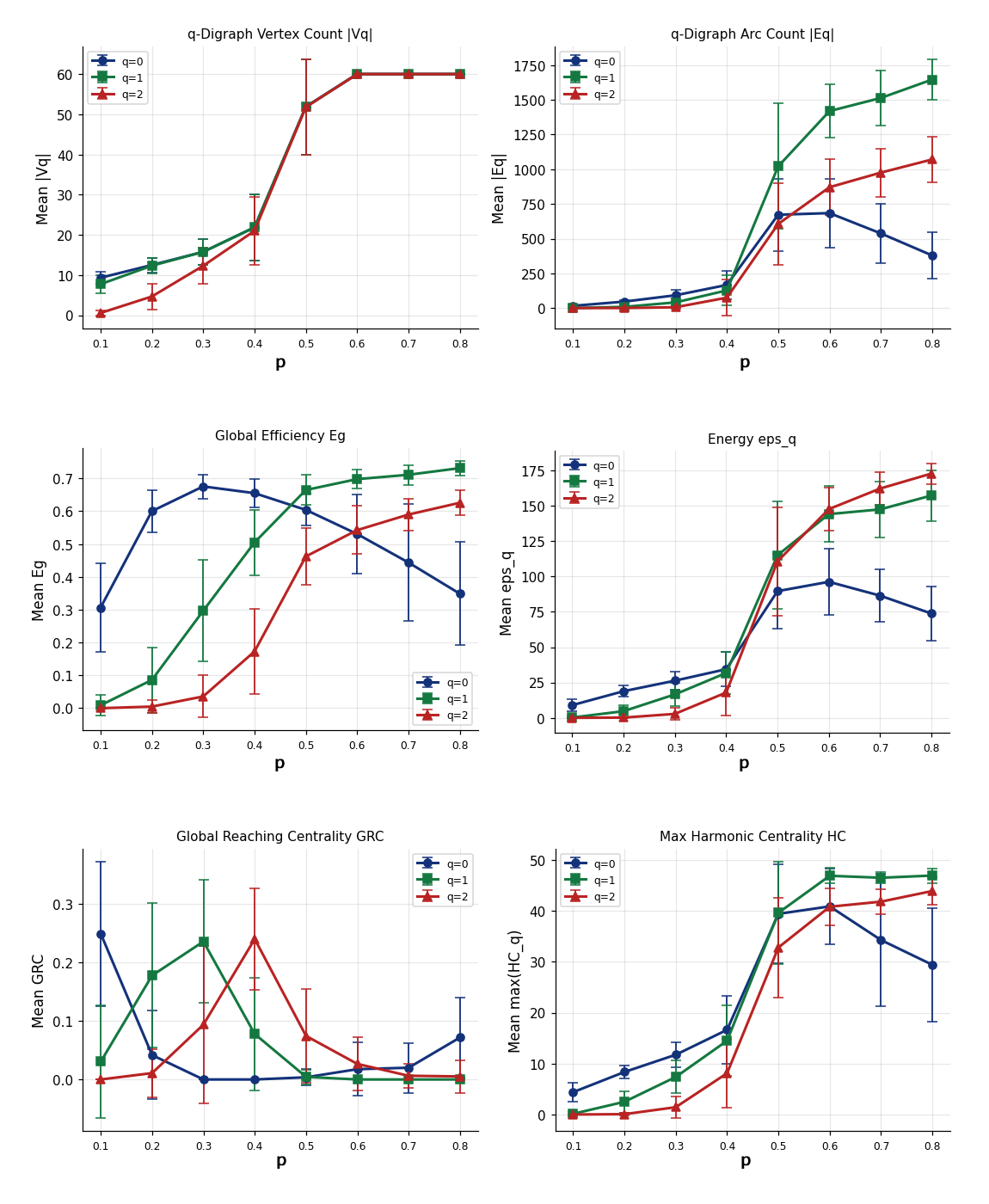}
\caption{Simulation results ($n=10$, 30 trials). Mean values and standard deviations.}
\label{fig:sim-plots}
\end{figure}

\subsection{Discussion}

\noindent \textbf{Number of $q$-ars (non-monotone at level $q = 0$)}: At level $q = 0$, the number of arcs $|\mathcal{E}_0|$ increases from $p = 0.1$ to $p = 0.5$ (peaking at $\approx 672$), then decreases for $p \ge 0.6$ ($\approx379$ at $p=0.8$).  This non-monotone behaviour is invisible to usual graph analysis, which sees a monotonically increasing number of arcs in the original digraph.  This is because at high arc density, a small number of large maximal cliques form, and these cliques have fewer strict lower-adjacency pairs between them (see Figure~\ref{fig:sim-plots}).

\medskip

\noindent \textbf{Global $q$-efficiency (non-monotone at level $q = 0$)}: At level $q = 0$, $E^0_{\mathrm{glob}}$ peaks at $p\approx0.3$ ($\approx0.68$) then decreases for higher values of $p$, because the $0$-digraph first becomes well connected then fragments into weakly disconnected components as the clique structure collapses into a few large isolated maximal simplices. On the other hand, $E^1_{\mathrm{glob}}$ and $E^2_{\mathrm{glob}}$ are monotonically increasing (see Figure~\ref{fig:sim-plots}).

\medskip

\noindent \textbf{Global $q$-reaching centrality}: At level $q = 0$, $\mathrm{GRC}_0$ peaks at $p=0.1$ ($\approx0.25$), then drops to zero at $p = 0.3$, and slight recovery at $p \ge 0.5$. $\mathrm{GRC}_1$ peaks at $p=0.3$ ($\approx0.24$), and $\mathrm{GRC}_2$ peaks at $p=0.4$ ($\approx0.24$). These distinct peaks might reflect the fact that each higher $q$ level requires a higher number of arcs in the original digraph before the $q$-digraph develops a non-trivial directed higher-order structure (see Figure~\ref{fig:sim-plots}).

\end{document}